\documentclass[12pt]{amsart}
\usepackage{color,graphicx,array, amssymb, amscd,slashed}

\newtheorem{Theorem}{Theorem}[section]
\newtheorem{Lemma}[Theorem]{Lemma}
\newtheorem{Proposition}[Theorem]{Proposition}
\newtheorem{Corollary}[Theorem]{Corollary}
\theoremstyle{definition}
\newtheorem{Definition}[Theorem]{Definition}
\theoremstyle{remark}
\newtheorem{Example}[Theorem]{Example}
\newtheorem{Remark}[Theorem]{Remark} 

\numberwithin{equation}{section}
\newcommand{\R}{\mathbb R}
\newcommand{\Rim}{\mathcal R}
\newcommand{\C}{\mathbb C}
\newcommand{\Z}{\mathbb Z}
\newcommand{\D}{\mathbb D}
\newcommand{\Min}{\mathbb L^3}
\newcommand{\Ha}{\mathbb H}

\newcommand{\St}{\mathbb S}

\newcommand{\cc}{\mathfrak c}
\newcommand{\s}{\mathfrak s}

\newcommand{\LGL}{\Lambda {\rm GL}_2 \mathbb C_{\sigma}}

\newcommand{\ISU}{{\rm SU}_{1, 1}}
\newcommand{\LISU}{\Lambda {\rm SU}_{1, 1 \sigma}}
\newcommand{\Uone}{{\rm U}_1}
\newcommand{\id}{\operatorname{id}}

\newcommand{\isu}{\mathfrak{su}_{1, 1}}
\newcommand{\Lisu}{\Lambda \mathfrak{su}_{1, 1 \sigma}}
\newcommand{\SL}{{\rm SL}_2 \mathbb C}

\newcommand{\LSL}{\Lambda {\rm SL}_2 \mathbb C_{\sigma}}
\newcommand{\LSLPM}{\Lambda^{\pm} {\rm SL}_{2} \mathbb C_{\sigma}}
\newcommand{\LSLM}{\Lambda^{-} {\rm SL}_{2} \mathbb C_{\sigma}}
\newcommand{\LSLMI}{\Lambda_*^{-} {\rm SL}_{2} \mathbb C_{\sigma}}
\newcommand{\LSLP}{\Lambda^+ {\rm SL}_{2} \mathbb C_{\sigma}}
\newcommand{\LSLPI}{\Lambda_*^+ {\rm SL}_{2} \mathbb C_{\sigma}}

\newcommand{\ad}{\operatorname{Ad}}
\newcommand{\di}{\operatorname{diag}}
\newcommand{\tr}{\operatorname{Tr}}
\newcommand{\Nil}{{\rm Nil}_3}
\newcommand{\iso}{\operatorname{Iso} (\Nil)}
\newcommand{\isoo}{\operatorname{Iso}_{\circ} (\Nil)}
\renewcommand{\Re}{\operatorname {Re}}
\renewcommand{\Im}{\operatorname {Im}}
\newcommand{\sdz}{(dz)^{1/2}}
\newcommand{\sdzb}{(d\bar z)^{1/2}}
\newcommand{\cNil}{{\rm center(Nil_3)}}
\newcommand{\pz}{\partial}
\newcommand{\pzb}{\bar \partial}

\newcommand{\Aut}{\operatorname{Aut}}

\renewcommand{\l}{\lambda}
\def\Vec#1{\mbox{\boldmath $#1$}}
\newcommand{\Red}[1]{{\color{black} #1}} 

\usepackage{amsmath}	
\begin{document}
\title{Minimal surfaces with non-trivial 
 geometry in the three-dimensional Heisenberg group}
\dedicatory{Dedicated to the memory of Uwe Abresch}
 \author[J. F.~Dorfmeister]{Josef F. Dorfmeister}
 \address{Fakult\"at f\"ur Mathematik, 
 TU-M\"unchen, 
 Boltzmann str. 3,
 D-85747, 
 Garching, 
 Germany}
 \email{dorfm@ma.tum.de}
\author[J.~Inoguchi]{Jun-ichi Inoguchi}
\address{Institute of Mathematics, 
University of Tsukuba, 
Tsukuba 305-8571, Japan}
\email{inoguchi@math.tsukuba.ac.jp}
\thanks{The second named author is partially supported by 
Kakenhi 15K04834, 19K03461}
 \author[S.-P.~Kobayashi]{Shimpei Kobayashi}
 \address{Department of Mathematics, Hokkaido University, 
 Sapporo, 060-0810, Japan}
 \email{shimpei@math.sci.hokudai.ac.jp}
 \thanks{The third named author is partially supported by Kakenhi 26400059, 18K03265 and Deutsche Forschungsgemeinschaft-Collaborative Research Center, TRR 109, ``Discretization in Geometry and Dynamics''.}
 \date{\today}
\pagestyle{plain}
\begin{abstract}
 We study symmetric minimal surfaces in the three-dimensional 
 Heisenberg group $\mathrm{Nil}_3$ using the generalized Weierstrass 
 type representation, the so-called loop group method. 
 In particular, we will present a general scheme for how to construct 
 minimal surfaces in $\Nil$ with non-trivial geometry. Special emphasis
 will be put on equivariant minimal surfaces.
 Moreover, we will classify equivariant 
 minimal surfaces given by one-parameter subgroups of the isometry group
 $\mathrm{Iso}_{\circ}(\mathrm{Nil}_3)$ of $\mathrm{Nil}_3$.
\end{abstract}
\maketitle    
 In every class of surfaces those with a large group of 
 symmetries have usually particularly nice properties.
 The most well known examples are rotationally 
 invariant surfaces, namely \textit{surfaces of revolution} in 
 Euclidean $3$-space $\mathbb{R}^3$. 
 More generally, surfaces in $\mathbb{R}^3$ invariant under 
 helicoidal motion have been studied extensively. 
In particular, do Carmo and Dajczer proved that 
the associated family of a non-zero
 constant mean curvature (CMC in short)  
surface of revolution consists of helicoidal surfaces of constant 
mean curvature \cite{dCD}. 

As is well known, the constancy of mean curvature for surfaces 
in $\mathbb{R}^3$ is equivalent to the harmonicity of the 
Gauss map. 
Based on this fundamental connection between 
CMC surfaces and harmonic maps, we can construct 
CMC surfaces via the loop group theoretic 
Weierstrass type representation of 
harmonic maps (now referred as to 
 the generalized Weierstrass type representation)
due to Pedit, Wu and the first named 
author of the present paper \cite{DPW}. 
From the harmonic map point of view, we notice the fundamental fact that 
the Gauss map of helicoidal CMC surfaces in $\mathbb{R}^3$, 
especially CMC surfaces of revolution in $\mathbb{R}^3$, 
are symmetric harmonic maps into the unit $2$-sphere $\mathbb{S}^2$. 
 Haak \cite{Haak} gave an 
alternative proof of the do Carmo-Dajczer theorem 
by using the generalized Weierstrass type representation.
 The general theory of symmetry of CMC surfaces in 
$\mathbb{R}^3$ is well organized \cite{DH1, DH2}. 
It is known that rotationally symmetric harmonic maps of Riemann surfaces 
are characterized as those with a 
{ many surface classes.
For example, in our previous paper \cite{DIKAsian}, the present authors 
established a generalized Weierstrass type representation for minimal surfaces 
in the $3$-dimensional Heisenberg group $\mathrm{Nil}_3$ which is one of the model 
spaces of Thurston geometries \cite{Thurston}. 
In this paper we study symmetric{ minimal surfaces in 
$\mathrm{Nil}_3$ via the generalized Weierstrass type representation 
 established in \cite{DIKAsian}.

To illustrate the methods discussed in this paper, 
we present here a brief account of the geometry 
 of symmetric minimal surfaces in the Heisenberg group.
\begin{itemize}
\item In 1995, Caddeo, Piu and Ratto studied 
rotational minimal surfaces in $\mathrm{Nil}_3$. 
On the other hand, in 1999, 
Figueroa, Mercuri and Pedro
studied helicoidal CMC surfaces as well as 
translation invariant CMC surfaces (including minimal ones) 
in $\mathrm{Nil}_3$.
\item Berard and Cavalcante studied the stability 
of rotational minimal surfaces \cite{BC}. 
\end{itemize}

Since we only know few examples of 
symmetric minimal surfaces above constructed 
using exclusively methods of classical differential geometry,
it is difficult to describe the moduli spaces of 
minimal surfaces with symmetry in $\Nil$. To describe a  moduli space, 
one needs first a systematic construction of symmetric minimal surfaces. 

For this purpose we use the 
generalized Weierstrass representation (loop group method) for 
minimal surfaces in $\mathrm{Nil}_3$. 
The starting point of the generalized Weierstrass representation
is to connect minimal surfaces in $\mathrm{Nil}_3$ and 
harmonic maps into the hyperbolic $2$-space $\mathbb{H}^2$
as well as loops of flat connections (see 
Appendix \ref{app:Pre} of the present paper).  

Those interactions between minimal surfaces, harmonic maps and 
loops of flat connections are derived from 
the following important discoveries:
\begin{itemize}
\item 
In 2009, Fernandez and Mira found a correspondence 
between minimal surfaces in $\mathrm{Nil}_3$ and (non-maximal) 
spacelike CMC surfaces in Minkowski $3$-space $\mathbb{L}^3$ (see 
\cite{Fer-Mira2}). 
\item In 2005, Berdinsky and Taimanov gave a spinor representation 
and nonlinear Dirac equations of the surfaces in $\mathrm{Nil}_3$ 
\cite{BT:Sur-Lie}. Berdinski \cite{Ber:Heisenberg} obtained a system of matrix valued functions which 
has spinor field solutions to the nonlinear Dirac equations given in \cite{BT:Sur-Lie}
(see Appendix \ref{subsec:BT}). In case of minimal surfaces, Berdinsky's system 
describes harmonic maps into the Riemannian symmetric space $\mathbb{H}^2=
\mathrm{SU}_{1,1}/\mathrm{U}_1$.
\end{itemize}

 It is crucial to understand the  serious differences between 
Euclidean CMC surface theory and minimal surface theory in $\mathrm{Nil}_3$.
In the  Euclidean case, the Gauss map of a CMC surface is a harmonic map 
into the unit $2$-sphere $\mathbb{S}^2=\mathrm{SU}_2/\mathrm{U}_1$. 
Next, the universal covering group of the Euclidean 
motion group is expressed as $\mathrm{SU}(2)\ltimes\mathfrak{su}(2)$. 
Thus the special unitary group $\mathrm{SU}(2)$ acts isometrically on 
both $\mathbb{S}^2$ and $\mathbb{R}^3$.

On the other hand, 
the normal Gauss map of a minimal surface in 
$\mathrm{Nil}_3$ takes value in the hyperbolic $2$-space 
$\mathbb{H}^2=\mathrm{SU}_{1,1}/\mathrm{U}_1$. However,  the 
identity component of the isometry group of $\mathrm{Nil}_3$ is 
$\mathrm{Nil}_3\rtimes\mathrm{U}_1$. Thus there is no 
isometric action of $\mathrm{SU}_{1,1}$ 
on $\mathrm{Nil}_3$. This difference means that 
we can not associate to each $g \in \mathrm{SU}_{1,1}$ an isometry of 
$\Nil$.

From a symmetry point of view, we realize that 
one-parameter subgroups of $\mathrm{SU}_{1,1}$ act on normal Gauss maps 
as isometries, but not on the  corresponding 
minimal surfaces in $\mathrm{Nil}_3$.

Thus we can not apply the general theory of symmetric harmonic maps 
\cite{Dorfmeister:OCU, DH1, DH2} to minimal surfaces in 
$\Nil$.

To overcome these difficulties, in the present paper, we investigate first 
the action of isometries on minimal surfaces in 
$\Nil$ and their 
effects on the normal Gauss maps. In addition we describe 
 these actions as monodromy of extended frames. This enables us to study 
minimal surfaces with symmetry via loop group method. 
Based on these fundamental facts, we establish a general theory of 
minimal surfaces in $\Nil$ with symmetry. 
\Red{In this paper we consider exclusively minimal surfaces in $\Nil$ without vertical points.
In particular, we consider only symmetries associated with
transformations in the identity component of $\mathrm{SU}_{1,1}$.}

This is the first time that the 
loop group method contributes to the study of minimal surfaces in $3$-dimensional 
homogeneous Riemannian spaces 
of \emph{non-constant curvature}. 
 This paper is organized as follows. 
In Section \ref{section1}, we start with introducing the notion of 
\emph{symmetry} for surfaces in $\Nil$.  
We give a fundamental characterization of symmetric minimal surfaces 
in $\Nil$ in terms of the property of corresponding normal Gauss maps 
(Theorem \ref{thm:symmetry}). Theorem \ref{thm:symmetry} clarifies the serious differences 
between minimal surface theory in $\Nil$ and that of CMC surfaces in Euclidean $3$-space. 
Based on Theorem \ref{thm:symmetry}, we will discuss how to construct 
minimal surfaces in $\Nil$ with \emph{non-trivial topology} via the generalized 
Weierstrass type representation \cite{DIKAsian}. 
 We will give a detailed study 
of the potentials invariant under all deck transformations. One of the key clues of these 
studies is the Iwasawa decomposition of the loop group of $\mathrm{SU}_{1,1}$. 
Because of the \emph{non compactness} of $\mathrm{SU}_{1,1}$, the Iwasawa decomposition of 
loop group is much involved, see \cite{BRS:Min,DIKAsian,Kellersch}. 
 Note that in case of 
CMC surfaces in $\mathbb{R}^3$ the key clue is the loop group of the 
\emph{compact} simple Lie group $\mathrm{SU}_2$. 
The non-compactness of $\mathrm{SU}_{1,1}$ causes 
case by case studies on monodromy matrices. 
To obtain detailed information on the behavior of extended frames under 
deck transformations, we  consider 
meromorphic extensions of minimal surfaces. As a result we obtain closing conditions 
for minimal surfaces with symmetry (Theorem \ref{thm:closing}, Corollary \ref{cor:closing}). 

In Section \ref{sc:Mincyl}, we will briefly discuss 
the construction of minimal cylinders by a method which is analogous to the one
introduced in \cite{DK:cyl} for CMC cylinders in Euclidean $3$-space. 
 In particular, we will show the existence of such cylinders
 which are \emph{not} equivariant, see Example \ref{ex:cyl}.
 In \cite{DK:cylNil}, we will discuss minimal cylinders in $\Nil$  detail. 
 For later use, in Section \ref{section3}, we recall the classification 
 of \emph{homogeneous minimal surfaces} in $\Nil$. 

In the final section, we start with an explicit description of 
one-parameter groups of isometries on $\Nil$. 
Lemma \ref{lm:rhophi} and Theorem \ref{thm:oneparameter} give a
complete description of one-parameter groups of isometries of 
$\Nil$ (compare with \cite{FMP}). 
These results themselves are valuable for the Riemannian geometry of $\Nil$. 
By our results, we can arrive at the classification of 
equivariant minimal surfaces in $\Nil$ (Corollary \ref{cor:eqminsurf}). 
It turned out that equivariant minimal surfaces in $\Nil$ (in the sense of 
Definition \ref{df:equivariant}) are exhausted by 
minimal helicoidal surfaces and minimal translation invariant surfaces. 
Our goal of the present paper is to give a construction method for 
equivariant minimal surfaces in $\Nil$ via the 
generalized Weierstrass type representation. To this end, 
we need to determine the potentials (data of generalized Weierstrass 
type representation) for equivariant minimal surfaces.  For the detailed analysis 
of one-parameter groups of automorphism on Riemann surfaces and 
compatible actions of one-parameter groups of isometries of $\Nil$, 
we will introduce the notion of $\mathbb{R}$-equivariant minimal surface and 
$\mathbb{S}^1$-equivariant minimal surface in $\Nil$.
We will determine potentials for those equivariant minimal surfaces. 
We will finally give a method of construction of all equivariant minimal surfaces 
by virtue of the generalized Weierstrass type representation. 
 An explicit construction of equivariant minimal surfaces will be done 
in a future publication  \cite{K:Explicit}.

Throughout this paper we will assume that all 
Riemann surfaces occurring are connected and denote by 
$\mathbb{S}^2$, $\mathbb{H}^2$, $\C$ 
the unit sphere in $\mathbb{R}^3$, 
the unit disk (sometimes equivalently replaced by the upper half-plane $\mathbb{H}$) 
and the complex plane, respectively.
Since there does not exist any compact minimal surface in $\Nil$ \cite{Figueroa2}, 
each Riemann surface occurring in this paper will have $\mathbb{H}^2$
or $\C$ as its universal cover.

As we have pointed out  before, in this paper we use the  generalized Weierstrass type 
representation established in \cite{DIKAsian}. For the convenience of the reader we have added a fairly extensive
Appendix. 
Here we recall results of \cite{DIKAsian} which are of relevance to this paper.
But we also expand the discussion of loc.cit., where it is useful for the goals of this paper.
In Appendix \ref{app:Pre} we recall the notation and the results of Sections 1--5 of \cite{DIKAsian}.
In Appendix \ref{app:Loop} we describe in some detail the various realizations of the 
normal Gauss map in
the unit disk $\mathbb{H}^2$, the upper hemisphere 
$\mathbb{S}^2_{+}$, and the hyperboloid $\mathbb{Q}^2$
(a model of the hyperbolic $2$-space).
 This clarifies the discussion of loc.cit. We also introduce the notion of a 
\textit{general extended frame}, which is contained implicitly in loc.cit, 
 but is needed explicitly for the investigation of symmetries in this paper. 
 Appendix \ref{app:DPW} presents details beyond loc.~cit  
relating to the  representation of extended frames of harmonic maps into any of the
three realizations of $\mathbb{H}^2$  listed above, and also to the validity of Theorem 6.1 of loc.~cit
under weaker assumptions. The latter actually presents the Sym formula in the way needed for loc.~cit and this paper. 
 We thus have corrected the phrasing of the statement of Theorem 6.1, loc.~cit. 
 The proof was given under the weaker assumptions already in loc.~cit.
 In Appendix \ref{sc:Meroextension} we prove that essentially all (anyway real analytic) geometric matrix functions
 occurring in this paper can be extended to globally meromorphic matrix functions in two independent complex variables. 
 This is needed in Section \ref{sbsc:MeroExt} of this paper.
 Finally, the last Appendix \ref{mysterious} gives a geometric meaning to the linear isomorphism 
 from $\isu$ to $\mathfrak{nil}_{3}$, used in the proof of the Sym formula Theorem 6.1. of \cite{DIKAsian}.


\section{Minimal surfaces with symmetries in $\Nil$}
\label{section1}
 In this section, we discuss symmetries of minimal surfaces in the 
 $3$-dimensional Heisenberg group $\Nil$. For fundamental properties
 of the homogeneous Riemannian space $\Nil$, we refer to our previous paper \cite{DIKAsian} or to
 Appendix \ref{subsc:nil3}. 
 Since there does not exist any compact minimal surface (without boundary) 
 in $\Nil$, we will discuss in this paper exclusively non-compact Riemann surfaces.

 A symmetry of some surface $S$ in some (metric) space $N$ 
 is an isometry $\rho$ of 
 $N$ which maps $S$ onto itself: $\rho(S) = S$. In this paper we consider 
 the case, where $\rho$ is an orientation preserving isometry of $\Nil$. 
 It turns out (see Theorem \ref{thm:X1}) that in some 
 cases a symmetry is implemented by a pair of maps $(\gamma, \rho)$ such that 
 the minimal surface $f: \Rim \rightarrow \Nil$ satisfies $f(\gamma.p) = \rho.f(p)$ 
 for all $p \in \Rim$, with some Riemann surface $\Rim$ and automorphism 
 $\gamma \in \Aut (\Rim)$.
 Thus we start from the following definition of symmetric surfaces in 
 a Riemannian manifold.
 We will denote by $\operatorname{Iso} (N)$ the group of isometries of 
 $N$ and by  $\operatorname{Iso}_{\circ} (N)$ its connected 
 identity component.
 \begin{Definition}
 Let $f: \Rim\to N$ be a map from a Riemann surface $\Rim$ 
 into a Riemannian manifold 
 $N$. Moreover,  let $\gamma$ and $\rho$ be elements of $\Aut (\Rim)$ and 
 $\operatorname{Iso} (N)$, respectively. 
 Then $f$ is \textit{symmetric with 
 respect to $(\gamma, \rho) \in \Aut (\Rim) \times  
 \operatorname{Iso}(N)$} 
 if 
\begin{equation} \label{defsym}
 f \circ \gamma = \rho \circ f
\end{equation}
 holds.
\end{Definition}

\subsection{Navigating between a Riemann surface and its universal cover}
\label{subsc:Navigating}
We will frequently consider a conformal immersion 
 $f:\mathcal{R} \rightarrow \Nil$ 
 from some Riemann surface  $\mathcal{R}$ into 
 the $3$-dimensional Heisenberg group $\Nil$ and its lift 
 $\tilde{f}: \widetilde{\mathcal{R}} \rightarrow \Nil$ 
 to the universal cover $\widetilde{\mathcal{R}}$ of 
 $\mathcal{R}$. Then
\[
 \tilde{f} = f \circ \pi_\Rim,
\]
where $\pi_{\Rim} : \widetilde{\mathcal{R} } \rightarrow \mathcal{R}$ 
denotes the natural projection. 

Following the procedure of \cite{DIKAsian} we need to consider 
a matrix valued function 
$\varPhi$, the generating spinors $\psi_j$ and an extended frame $F$ 
for the discussion of $f$ and the corresponding objects, 
capped with a ``$\sim$'' for $\tilde{f}$ 
(see Appendix \ref{subsc:SurfaceTheory}).

 Note that extended frames are always defined on the universal cover of a given Riemann surface, whence we always drop the superscript 
``$\sim$''  for extended frames.
Then we obtain, see also the appendix \ref{subsc:SurfaceTheory},
\[
 f^{-1}\pz f=  \varPhi =\sum_{k=1}^3 \phi_k e_k
\]
 with respect to the natural 
 basis $\{e_1,e_2,e_3\}$ of Lie algebra $\mathfrak{nil}_3$ of $\Nil$
 and the corresponding representation for $\tilde{f}$. Here $\pz$ and $\pzb$ are defined as
\[
 \pz = \frac{1}{2}\left( \frac{\partial}{ \partial x} - i \frac{\partial}{ \partial y}\right), \quad 
 \pzb = \frac{1}{2}\left( \frac{\partial}{ \partial x} + i \frac{\partial}{ \partial y}\right)
\]
 for a conformal coordinate $z = x + i y$.
 Hence
 \[
  \tilde{\varPhi} = \varPhi \circ \pi_\Rim\quad \mbox{and}\quad
 {\tilde{\phi}}_j = \phi_j \circ \pi_\Rim 
 \]
  for $j = 1$,$2$,$3$.
 It will be convenient to  abbreviate 
\[
f(z, \bar z)= (f_1(z, \bar z),f_2(z, \bar z),f_3(z, \bar z))
\]
by $f(z) = (f_1(z),f_2(z),f_3(z))$. Then 
 \begin{align}
f(z)^{-1} \pz f(z) &=  \frac{d}{dt}\Big|_{t=0} \left(f(z)^{-1} f(z + t)\right) \label{eq:derivative}\\
 &=  \left(\pz f_1(z), \pz f_2(z), \pz f_3(z) + 
 \frac{1}{2}(- f_1(z) \pz f_2(z) + f_2(z) \pz f_1(z))\right), \nonumber  
 \end{align}
 in view of the fact that the product in $\Nil$ is given by the formula 
 (see also appendix \ref{subsc:nil3}):
 \begin{equation}\label{eq:Nilmult}
  (x_1,x_2,x_3) \cdot (u_1,u_2,u_3) = \left(
 x_1+u_1, x_2+u_2, x_3+u_3 + \frac{1}{2}(x_1u_2 - x_2u_1)
 \right).
 \end{equation}
 Now let us consider the generating spinors 
$\psi_1(dz)^{1/2}$ and $\psi_2(d\bar{z})^{1/2}$ of the conformal immersion 
$f:\mathcal{R}\to\Nil$ (see \cite[Section 3]{DIKAsian} or appendix \ref{subsc:SurfaceTheory}).  
  
We need to express  $\tilde{\phi}_j $ and  $\phi_j$ by the $\tilde \psi_j$ and $\psi_j$ respectively. These functions are uniquely defined up to a sign and from the defining equation we obtain 
 $\tilde{\psi}_j^2 = \psi_j^2 \circ \pi_\Rim$.
 Since the choice of sign has no effect on the discussion of minimal surfaces in $\Nil$, without loss of generality we choose the sign such that 
 \[
  \tilde{\psi}_j = \psi_j \circ \pi_\Rim.
 \]
 Next we discuss the relation between the 
 normal Gauss maps of $f$ and $\tilde{f}$. 
 The left translated unit normal of $f$ in $\mathfrak{nil}_3$ to the origin
 of a conformal immersion
 $f:\mathcal{R}\to\Nil$ takes value in the hyperboloid model 
 $\mathbb{Q}^2$ of the hyperbolic 
 $2$-space $\mathbb{H}^2$ embedded in the Minkowski $3$-space $\Min$, see
 Appendix \ref{subsc:Realizing} or  \cite{DIKAsian}.
 Via its stereographic projection of $\mathbb{Q}^2$ onto $\mathbb{H}^2$, 
 we obtain a map $g$ into the Poincar{\'e} disk $\mathbb{H}^2$ and call it the 
 \textit{normal Gauss map}.
  
 Since the normal Gauss maps $g$ and $\tilde{g}$ of $f$ and 
 $\tilde{f}$ are expressed by the corresponding generating spinors 
 (which have the relation stated above) it is clear that we also have
\[
 \tilde{g} = g \circ \pi_\Rim.
\]
 Considering now a map $f$ with a symmetry $(\gamma, \rho)$, 
 that is,  
 satisfying equation \eqref{defsym}, 
 we obtain the corresponding equation 
\[
 \tilde{f} (\tilde{\gamma} . z )  = \rho.\tilde{ f} (z),
\]
 where $\tilde{\gamma}$ denotes the automorphism of 
 $\widetilde{ \mathcal{R} }$  induced by $\gamma$.

\subsection{The transformation behaviour of the generating spinors,
the normal Gauss map and the extended frame}

 First we recall from \cite{DIKAsian} that the isometry group of $\mathrm{Nil}_3$ has two connected components. 
 The identity component acts by orientation preserving diffeomorphisms 
 and the elements of the other connected component reverse the orientation. 
 In this paper we will consider exclusively orientation preserving 
 transformations and therefore will only consider 
 $\mathrm{Iso}_{\circ}(\mathrm{Nil}_3)$,
 the identity component of the isometry group $\iso$ of $\mathrm{Nil}_3$. 
 We recall that $\mathrm{Iso}_{\circ}(\mathrm{Nil}_3)$ is isomorphic to the 
 the semi-direct product $\mathrm{Nil}_3\rtimes \mathrm{U}_1$  of $\mathrm{Nil}_3$ and 
 $\mathrm{U}_1$ , $\mathrm{Iso}_{\circ}(\mathrm{Nil}_3) \cong \mathrm{Nil}_3\rtimes \mathrm{U}_1$, 
 with  the action:
\begin{equation}\label{eq:action}
 ((a_1, a_2, a_3), e^{i \theta}) . (x_1, x_2, x_3)
 =(a_1, a_2, a_3) \cdot (\cos \theta x_1 - \sin \theta x_2, \sin \theta x_1 + 
 \cos \theta x_2, x_3),
 \end{equation}
 where ``$\, \cdot\, $'' denotes the product in $\Nil$ defined by
 \eqref{eq:Nilmult}.
 Since $\Nil \subset \mathrm{Nil}_3\rtimes \mathrm{U}_1 $ 
 is normal in $\isoo$ we can write $\rho$ as
\[
 \rho = p s,
\]
 where $p \in \Nil$ and $s = e^{i \theta} \in  \mathrm{U}_1$.

 The Lie algebra $\mathfrak{iso}(\mathrm{Nil}_{3})$ of 
$\mathrm{Iso}_{\circ}(\mathrm{Nil}_3)$ is generated 
 by four Killing vector fields 
\begin{equation}\label{eq:Ej}
 E_1 = \frac{\partial}{\partial x_1} -\frac{1}{2} x_2 
 \frac{\partial}{\partial x_3}, \;
E_2 = \frac{\partial}{\partial x_2} +\frac{1}{2} x_1 \frac{\partial}{\partial x_3}, \;
 E_3= \frac{\partial}{\partial x_3}\quad\mbox{and}
 \quad E_4=-x_{2}\frac{\partial}{\partial x_1}+x_{1}\frac{\partial}{\partial x_2},
\end{equation}
 respectively.
 The commutation relations are respectively
\[
[E_4,E_1]=E_2,
\quad
[E_4,E_2]=-E_1
\quad\mbox{and}\quad
[E_1,E_2]=E_3.
\]
 Next we recall from Appendix \ref{subsc:SurfaceTheory} 
 of the appendix the notation:
 $f^{-1}\pz fdz=\varPhi{d}z$ on a simply connected domain 
 $\mathbb{D}$ that takes values 
 in the complexification $\mathfrak{nil}_3^{\mathbb{C}}$ of the 
 Lie algebra $\mathfrak{nil}_3$. With respect to the natural 
 basis $\{e_1,e_2,e_3\}$ of $\mathfrak{nil}_3$, we expand $\varPhi$ as 
 $\varPhi =\sum_{k=1}^3 \phi_k e_k$ and
 obtain  $(\phi_1)^2+(\phi_2)^2+(\phi_3)^2=0,$ since $f$ is conformal.

 \begin{Theorem} Let $f: \D \rightarrow \Nil$ be a minimal surface in $\Nil$ and $(\gamma,\rho)$ a symmetry of $f$. Writing $\rho = ps,$  where $p \in \Nil$ and $s = e^{i \theta} \in  \mathrm{U}_1$ as above, we obtain for $f \circ \gamma$ the transformation formula
 \begin{align}
  f(\gamma. z)^{-1} \hat \pz f(\gamma. z) &=  (s. f(z))^{-1} (s. \hat \pz f(z)) \\ 
 &=    \left(  \cc \hat \pz f_{1} - \s \hat \pz f_{2} ,
\s  \pz f_{1} +\cc  \hat \pz f_{2}, \hat \pz f_{3} +     
\frac{1}{2}(- f_1 \hat\pz f_{2} + f_2 \hat \pz f_{1}) \right), 
\nonumber
\end{align}
where $\hat \partial = \frac{\partial}{\partial(\gamma.z)}$, $\cc = \cos \theta$ and $\s = \sin \theta$.
 \end{Theorem}
 
\begin{proof} Recall that we will use the abbreviation  $f(z, \bar{z}) = f(z)$.
 Now consider the equation $\rho. f(z)  = (ps). f(z)$ 
 and differentiate. By the formula for the action of $\rho$ defined 
 in \eqref{eq:action}, we obtain $\hat \pz (ps. f(z)) = ps.(\hat \pz f(z))$, where  $ps.\pz f$ denotes 
 the action of $ps \in \isoo$ 
 on the tangent bundle 
 ${\rm T\Nil} \cong \Nil\ltimes\mathfrak{nil}_3$.
 Hence 
\[
 f(\gamma. z)^{-1} \hat \pz f(\gamma. z)= (\rho.f(z))^{-1} \hat \pz(\rho. f(z))
 = (s.f(z))^{-1}. p^{-1}.p.(s. \hat \pz f(z)), 
\]
 thus
\[
 f(\gamma. z)^{-1} \hat \pz f(\gamma. z) = (s. f(z))^{-1} (s. \hat \pz f(z)).
\] 
 Clearly, the right side only involves the ``fiber rotation'' given by $\theta$.
 From \eqref{eq:action}, we obtain
\begin{equation*}
 s.(\hat \pz f_{1}, \hat \pz f_{2}, \hat \pz f_{3})
 = (\cc \hat \pz f_{1} -\s  \hat \pz f_{2}, \s  \hat \pz f_{1} + 
 \cc  \hat \pz f_{2}, \hat \pz f_{3}),
\end{equation*}
 where $\cc = \cos \theta$ and $\s = \sin \theta$.
 Thus in view of the  formula given in \eqref{eq:derivative}, 
 we obtain
\begin{align*}
 (s.f)^{-1} (s.\hat \pz f) &= 
 \left( -(\cc f_1 -\s f_2),  -(\s f_1 + \cc f_2), -f_3\right) \cdot 
 ( \cc \hat \pz f_{1} - \s \hat \pz f_{2}, \s \hat \pz f_{1} + \cc 
 \hat \pz f_{2}, \hat \pz f_{3} )\\
&= \big( 
\cc \hat \pz f_{1} - \s \hat \pz f_{2},\; 
\s \hat \pz f_{1} +\cc  \hat \pz f_{2},  \\
&\phantom{==} \left.  \hat \pz f_{3} + 
 \tfrac{1}{2}\left\{ - (\cc f_1 -\s f_2)(\s \hat \pz f_{1} + \cc \hat \pz f_{2}) +
  (\s f_1 + \cc f_2)(\cc \hat \pz f_{1}-\s \hat \pz f_{2})\right\} \right) \\ 
&=\left(  \cc \hat \pz f_{1} - \s \hat \pz f_{2} ,
\s  \hat \pz f_{1} +\cc  \hat \pz f_{2}, \hat \pz f_{3} +     
\frac{1}{2}(- f_1 \hat \pz f_{2} + f_2 \hat \pz f_{1}) \right).
\end{align*}
This completes the proof.
\end{proof}
\begin{Corollary}
 Retaining the notation used above we obtain the following formula $:$
\begin{align*}
(f \circ \gamma)^{-1} \hat \pz (f \circ \gamma) = \left(\hat \phi_1, \hat \phi_2,  \hat \phi_3 \right)  
= \left( \overline{\hat \psi_2}^2 -\hat \psi_1^2, \;
i \left(\overline{\hat \psi_2}^2 + \hat \psi_1^2\right), \;
2 \hat \psi_1 \overline{\hat \psi_2}\right),
\end{align*}
 where $\hat \pz = \frac{\partial}{\partial (\gamma .z)}$, 
 $\hat{\phi}_j, (j=1, 2, 3)$ and $\hat{\psi}_j, (j=1, 2)$
 are the components of $(f \circ \gamma)^{-1} \hat \pz ( f \circ \gamma)$
 and the corresponding spinors respectively.
 From the last section we know 
 $f^{-1} \pz f = (\phi_1, \phi_2, \phi_3)$,
 hence
\[
\begin{pmatrix}
\hat \phi_1 \\ \hat \phi_2 
\end{pmatrix}
= \frac1{\partial (\gamma. z)}
\begin{pmatrix} 
\cc  & - \s  \\
\s & \cc
\end{pmatrix} 
\begin{pmatrix}\phi_1 \\ \phi_2 \end{pmatrix}.
\]
Moreover,
\[
 \hat{\psi}_1 = \epsilon e^{i \theta/2} \psi_1 
 (\partial (\gamma. z))^{-1/2}\quad \mbox{and} \quad 
 \hat{\psi}_2 = \epsilon e^{i \theta/2} \psi_2
 (\overline{\partial (\gamma. z)})^{-1/2},
\]
with $\epsilon = \pm 1$.
\end{Corollary}

\begin{proof}
It only remains to prove the last two claims. To verify this we observe that from the matrix equation we infer 
\begin{align*}
\begin{cases}
2 \overline{ \hat {\psi}_2}^2  = \hat \phi_1 - i \hat \phi_2
= \partial (\gamma.z)^{-1}\left(\cc \phi_1 -\s \phi_2 -i \s \phi_1 -i \cc \phi_2 \right)
 = 2 \partial (\gamma.z)^{-1}(\cc - i \s )\overline{\psi_2}^2 \\
2 { \hat{\psi}_1}^2  =  - \hat \phi_1  -  i \hat \phi_2 
=  - (\partial (\gamma.z))^{-1}\left(\cc \phi_1 - \s \phi_2 + i \s \phi_1 + i \cc \phi_2\right)
= 2(\partial (\gamma.z))^{-1} (\cc + i \s)\psi_1^2
\end{cases}.
\end{align*}
Thus we obtain in view of the relations discussed in the previous subsection$:$
\begin{equation} \label{trafopsihat}
 \hat{\psi}_1 = \epsilon_1 e^{i \theta/2} \psi_1 (\partial (\gamma. z))^{-1/2}\quad \mbox{and} \quad 
 \hat{\psi}_2 = \epsilon_2 e^{i \theta/2} \psi_2(\overline{\partial (\gamma. z)})^{-1/2},
\end{equation}
 with $\epsilon_j = \pm 1$.
 The equation above for $ (s.f)^{-1} (s.\pz f) $ shows that the third 
 component does not change with $\theta$. Therefore we have
 $\hat \psi_1 \overline{\hat \psi_2} \partial (\gamma. z)^{-1} =  \psi_1 \overline{\psi_2}$
 and $\epsilon_1 = \epsilon_2 =  \epsilon = \pm1$ follows.
\end{proof}
 As a consequence, the normal Gauss map satisfies the following transformation formula
\begin{equation}\label{eq:Gausstrans}
g(\gamma. z) = \frac{\>\>\hat \psi_2\>\>}{\overline{\hat \psi_1}}
= \frac{\>\>e^{i \theta/2}\>\psi_2\>\>}{e^{-i\theta/2}\>\overline{\psi_1}}
= e^{i \theta} g(z).
\end{equation}
This shows that $g(\gamma.z) = R. g(z)$,  that is, 
\begin{Corollary}
Retaining the notation above, the normal Gauss map has the transformation behaviour
$g\circ \gamma = R \circ g,$
where $R$ is the rotation about 
 $0 \in \mathbb H^2$ by the angle $\theta$.
 \end{Corollary}
 Next we consider an extended frame $F$ of the minimal surface $f$ in $\Nil$. 
 By equation \eqref{eq:twoextmin} we know   that any other extended 
 frame $\tilde F$ of $f$ is given by  $\tilde F = A F$
 for  some $A \in \LISU$ such that  $A|_{\l=1}= \id$.
 Applying this to $\hat{f} = f \circ \gamma$ we obtain in view 
 of  \eqref{trafopsihat} the equation
 \begin{equation}\label{eq:transofF}
 F(\gamma.z, \overline{\gamma.z}, \lambda =1) = M(\gamma, \lambda =1) F(z, \bar z, \lambda = 1) k(\gamma, z, \bar z)
 \end{equation}
  where
 \begin{equation} \label{trafoextframe}
 M(\gamma, \lambda=1)  = \di (e^{i\frac{\theta}{2}}, e^{-i \frac{\theta}{2}}),  
\quad 
 k(\gamma, z, \bar z) = \di 
 \left(\frac{\left|\sqrt{\partial(\gamma.z)}\right|}{\sqrt{\partial(\gamma.z)}},
 \frac{\left|\sqrt{\partial(\gamma.z)}\right|}{\overline{\sqrt{\partial(\gamma.z)}}}  \right) \in\Uone,
 \end{equation}
 and in particular $M( 1, \lambda=1) =\id$.
%

\subsection{Characterizing symmetries of a minimal immersion by symmetries 
of its associated normal Gauss map}
 In the theorem below we characterize symmetric minimal surfaces in $\Nil$
 by symmetric harmonic normal Gauss maps. 
 Note that the unit disk $\mathbb{H}^2$ is 
  represented in the form $\mathbb{H}^2=\mathrm{SU}_{1,1}/\mathrm{U}_1$ as a 
  \textit{Riemannian symmetric space}, where
 $\mathrm{SU}_{1,1}$ acts by M\"{o}bius transformations and the base point is $z = 0$.
\begin{Theorem}\label{thm:symmetry}
 Let $\mathcal{R}$  be a Riemann surface, $f : \Rim \to \Nil$ 
 a minimal  surface and $g : \Rim \to \mathbb H^2$ the normal 
 Gauss map of $f$.  Then the following statements hold$:$
\begin{enumerate}
\item[(a)] 
 If $f$ is symmetric relative 
 to $(\gamma, \rho)$,  then $g$ is symmetric relative to $(\gamma, R)$, 
 that is, 
\begin{equation*}\label{eq:X2}
 g \circ \gamma = R \circ g
\end{equation*}
 holds, where $R$ is a rotation about $0 \in \mathbb H^2$ such that 
 the angle of $R$ is given by that of the fiber rotation of $\rho$.

 \item[(b)] Conversely, if $g$ is symmetric with respect to $(\gamma, R)$ 
 such that $R$ is a rotation about 
 $0 \in \mathbb H^2$,  then $f$ is symmetric
  with respect to $(\gamma, \rho)$, that is, 
\begin{equation}\label{eq:X3}
 f \circ \gamma = \rho \circ f
 \end{equation}
 holds, where $\rho$ is an  element in $\isoo$  and such that the angle of the fiber rotation 
 of $\rho$ is given by that of $R$.
\end{enumerate}
\end{Theorem}

\begin{proof} Recall that we will use the abbreviation  $f(z, \bar{z}) = f(z)$.
 
 Part (a): 
The claim follows from \eqref{eq:Gausstrans}.

 Part (b): 
 Let $g :  \Rim \to \mathbb H^2$ be the normal Gauss map of $f$
 and assume $g(\gamma.  z) = e^{i\theta} g (z) = R. g(z)$ holds.
 Since $f$ is already defined on $\Rim$, it is easy to see that it suffices to verify 
 equation \eqref{eq:X3} on the universal cover. Hence we can assume 
 without loss of generality that $\mathcal{R}$ is simply-connected. 

 Let $F$ be an extended frame of the minimal surface $f$ 
 as in \eqref{special frame}
 such that the immersion $(\Xi_{\mathrm{nil}}\circ \hat{f})|_{\l = 1}$ 
 obtained by inserting $F$ 
 into the Sym formula \eqref{eq:symNil} at $\lambda=1$ becomes 
 the original minimal surface $f$. 
 (Note that such an extended frame exists by  Theorem \ref{thm:Sym}.) 
 
 Then the extended frame $F$ 
 of $f$ satisfies  
 \begin{equation}\label{eq:transext}
 \hat F( z, \bar z, \lambda) = F(\gamma. z, \overline{\gamma. z}, \lambda) =
 M(\gamma,\lambda) F(z, \bar z, \lambda) k(\gamma,z, \bar z),
 \end{equation}
 where
 \begin{equation*}
M(\gamma, \lambda=1)  = \di (e^{i\frac{\theta}{2}}, e^{-i \frac{\theta}{2}}),  
\quad  \mbox{and in particular} \quad  M( 1, \lambda=1) =\id,
 \end{equation*}
 and $k(\gamma, z, \bar z)$ is a 
 $\l$-independent $\mathrm{U}_1$-valued map, 
 see also Proposition \ref{prp:transformation}.
 So far, 
 in the last equation, $M$ and $k$ may not be defined uniquely. 
However, since the monodromy of $g$ is a one-parameter group, the lift $F$, 
for $\lambda = 1$,  inherits the property of having a one-parameter group of 
monodromy matrices. As a consequence, the matrix $k$ is a 
 crossed homomorphism, see also Section \ref{sbsc:inv}.
The introduction of $\lambda$ does not change $k$, whence the monodromy matrix is a  ($\lambda$-dependent) one-parameter group.
 From this the representation above follows uniquely.

 Now a straightforward computation 
 shows that $\hat f$ changes 
 by $\gamma$ as 
 \begin{align*}
 \hat f(\gamma. z) = &
 \left(\ad (M) f_{\Min}(z) +
 X\right)^o  \\
 &\nonumber+  \left(\ad(M) 
 \left(-\frac{i}{2} \lambda \partial_{\lambda} f_{\Min}(z)\right)
 + \frac{1}{2}[X, \ad(M) f_{\Min}(z)] 
 +Y\right)^{d},
\end{align*}
 and thus 
\begin{align*}
 \hat f(\gamma. z)|_{\l=1}  =& 
 \left.\left\{\ad(M) \hat f(z)  + X^o 
 + \frac{1}{2}\left( [X, \ad(M) f_{\Min}(z)]\right)^d 
 +Y^d\right\}\right|_{\lambda =1}
 \end{align*}
 where $X$ and $Y$ are defined by 
\[
 X =  -i \lambda (\partial_{\lambda} M) M^{-1}, \quad\mbox{and}\quad
 Y =- \frac{i}{2} \lambda \partial_{\lambda} X 
 = - \frac{1}{2}  \lambda \partial_{\lambda} (\lambda (\partial_{\lambda} M) M^{-1}),
\]
 respectively.\footnote{$X$ and $Y$ are slightly different from $X_{\lambda}$ and $Y_{\lambda}$
 defined in \cite{DIKAsian}, that is, $X = -X_{\lambda}$ and $Y = \frac{1}{2} Y_{\lambda}$, 
 respectively. }
 Note $[X, \ad (M) f_{\Min}(z)]^d = [ X^o, (\ad (M)  f_{\Min}(z))^o]^d$ 
 and $(f_{\Min}(z))^o = (\hat{f}(z))^o$.
 Then we set 
\begin{align*}
 X|_{\lambda =1} &=p \mathcal E_1 + q  \mathcal E_2 + * \mathcal E_3=
 \frac{1}{2}
 \begin{pmatrix} 
  * &  -q + i p  \\
  -q -i p &  *
 \end{pmatrix}, 
\intertext{and}
 Y|_{\lambda =1} &=*  \mathcal E_1 + *  \mathcal E_2 + r  \mathcal E_3=
\frac{1}{2}
 \begin{pmatrix} 
  - i r &  * \\
  *& i r
 \end{pmatrix},
\end{align*}
 where the basis $\mathcal E_i (i=1, 2, 3)$ was defined in \eqref{eq:basis}, 
 $p$, $q$, $r$ are some real constants. Altogether this shows
 \begin{align*}
 \hat f(\gamma. z)|_{\l=1}  =& 
 \left.\left\{\ad(M) \hat f(z)  
 + \frac{1}{2}\left( [X^o, (\ad(M) f_{\Min}(z))^o]\right)^d 
 + T \right\}\right|_{\lambda =1}
 \end{align*}
where
\begin{equation*}
T = \frac{1}{2}
 \begin{pmatrix} 
  - i r &  -q + i p \\
  -q -ip& i r
 \end{pmatrix},
\end{equation*}
 Hence $\hat f$ and thus the resulting minimal surface 
 $f= (\Xi_{\mathrm{nil}}\circ \hat{f})|_{\l = 1}$ 
 in $\Nil$ is symmetric with respect to 
 $(\gamma, \rho)$, that is, 
\[
 f(\gamma. z) = \rho. f(z),
\] 
 holds, where $\rho$ is given by $\rho = ((p, q, r), e^{i \theta})$.
 The angle of fiber rotation is clearly given by that of $R$.
\end{proof}
\begin{Remark}
\mbox{}
\begin{enumerate}
 \item 
 Part (a) in Theorem \ref{thm:symmetry} 
 is due to Daniel \cite{Daniel:GaussHeisenbergw} in the case 
 where either  $\rho$ is a translation 
 by an element of $\Nil$ or a rotation.
 \item The proof of part (a) above  works for general  
 $\rho \in \isoo$ and part (b) proves the converse of part (a).
\item  We would like to point out that part (a) actually holds 
 for any surface in $\Nil$. In the proof 
 of part (b) we used the Sym-formula for minimal surfaces. Thus 
 at this point we do not know whether it holds for any surface in $\Nil$, or not.
\end{enumerate}

\end{Remark}

\section{Minimal surfaces in $\Nil$ from non-simply-connected surfaces}
 In this section we will discuss how one can construct minimal surfaces 
 in $\Nil$ which are defined on a \emph{non-simply-connected} Riemann 
 surface $\Rim$.
 The description will use potentials as discussed in \cite{DIKAsian}. 
 We will discuss the corresponding closing conditions of the monodromy
 representation of the fundamental group $\pi_1(\Rim)$.
 There are naturally two parts in this discussion.
\subsection{Invariant potentials}\label{sbsc:inv}
 Let $\Rim$ be an arbitrary connected non-compact Riemann surface and 
 $\pi_\Rim: \widetilde{\Rim} \rightarrow \Rim$ its universal cover.
 Let $f:\Rim \rightarrow \Nil$  be a minimal surface.
 Then also $\tilde{f}: \widetilde{\Rim} \rightarrow \Nil$, 
 defined by $\tilde{f} = f \circ \pi_\Rim$ is a minimal surface. 
 Clearly, this surface satisfies $\tilde{f} \circ \tau = \tilde{f}$ 
 for all $\tau \in \pi_1(\Rim)$, where the latter group is considered 
 as the group of deck transformations of $\mathcal{R}$ acting on 
 $\widetilde{\Rim}$. For a minimal surface in $\Nil$
 we have always considered the corresponding normal Gauss map.
 In the present situation  we obtain two normal Gauss maps, 
 $g: \Rim \rightarrow \mathbb{H}^2$  for $f$ and $\tilde{g}: \widetilde{\Rim } :\rightarrow \mathbb{H}^2$  for $\tilde{f}$. 
 They are related by $\tilde{g} = g \circ \pi_\Rim$.
 Let $\tilde{F}$ denote the extended frame of $\tilde{g}$.
 (For more on the relation between the surface and its lift to the universal cover, see Section \ref{subsc:Navigating}.)

 Hereafter we use loop groups for our study. 
We refer to Appendix 
\ref{sc:loopgroups} for fundamental facts on loop groups 
used frequently in this paper.
 \begin{Proposition}\label{prp:transformation}
 For any extended frame $\tilde F$ of $\tilde g$ and for every $\tau \in \pi_1(\Rim)$, 
 there exists some diagonal matrix 
 $\tilde{k}(\tau, z, \bar{z})$ in $ \Uone$ and $M(\tau, \l)$ taking values in $\LISU$ 
 such that
 \begin{equation} \label{trafoFtilde}
 \tilde{F}(\tau.z, \overline{\tau.z}, \lambda) = M(\tau, \l)
 \tilde{F}(z, \bar{z}, \lambda) 
\tilde{k}(\tau, z, \bar{z})\quad \mbox{and} \quad 
 M(\tau,\l =1) = \id.
 \end{equation}
 \end{Proposition}
\begin{proof}
 Since $f$ is symmetric with respect to $(\tau, \id)$, 
 \eqref{eq:transofF} 
 can be rephrased as 
\[
 \tilde{F}(\tau.z, \overline{\tau.z}, \lambda=1) = M(\tau, \l=1)
 \tilde{F}(z, \bar{z}, \lambda=1) 
\tilde{k}(\tau, z, \bar{z})\quad \mbox{and} \quad 
 M(\tau,\l =1) = \id.
\]
 Therefore 
\[
 \tilde{F}(\tau.z, \overline{\tau.z}, \lambda) = M(\tau, \l)
 \tilde{F}(z, \bar{z}, \lambda) 
\tilde{k}(\tau, z, \bar{z}, \l)
\]
 follows, where $M$ and $\tilde{k}$ take values in $\LISU$ and $\Uone$, respectively.
 To show $\tilde k$ is independent of $\l$, 
 look at the Maurer-Cartan form $\tilde \alpha^{\lambda}$ of 
 $\tilde F$. Then the Maurer-Cartan forms of 
 $\tilde F(\tau. z, \overline{\tau. z}, \l)$ and $\tilde F(z, \bar{z}, \l)$ 
 have the same $\l$ distribution. Thus $\tilde{k}$
 is independent of $\lambda$. Therefore \eqref{trafoFtilde} holds.
\end{proof}
 Note that we also use $\tau$ for the induced action of $\tau$ on 
$ \widetilde{\Rim}$ and $ \tilde{k}(\tau, z, \bar{z})$ 
 satisfies the ``crossed-homomorphism''
 property:
\[
 \tilde{k}(\mu \tau, z, \bar{z}) = \tilde{k}(\tau,z,\bar{z}) 
 \tilde{k}(\mu, \tau.z, \overline{\tau.z}).
\]
 Then we have the following theorem.
\begin{Theorem} \label{invframe}
 Every crossed homomorphism $ \tilde{k}(\tau, z, \bar{z})$  
 occurring above is a 
 ``co-boundary'', that is, it can be written in the form 
\[
 \tilde{k}(\tau, z, \bar{z}) = \tilde{k}_0 (z, \bar z) 
 \tilde{k}_0^{-1} ( \tau.z, \overline{\tau. z}),
\] 
 where $\tilde{k}_0$ is a real-analytic $\Uone$-valued function.
 In particular, the frame $ \hat{F} = \tilde{F} \tilde k_0$  
 satisfies  $\hat{F}(\tau.z, \overline{\tau.z}, \lambda) = M(\tau, \l)
 \hat{F}(z, \bar{z}, \lambda) $ for  $\tau \in \pi_1(\Rim)$.
 As a consequence, for every minimal surface in $\Nil$ there exists 
 a frame defined on $\Rim$. 
 More precisely, 
\[
\hat{F}(\tau.z, \overline{\tau.z}, \lambda=1) = M(\tau, \l =1)
 \hat{F}(z, \bar{z}, \lambda =1) 		  
\]  
 for $\tau \in \pi_1(\Rim)$.
 \end{Theorem}
\begin{Remark}
 It is important to distinguish our extended frame built 
 from the $\psi_j$'s in \eqref{special frame} from the above ``invariant frame''.
\end{Remark}
 Before giving the proof we recall:
 Following the discussion for other surface classes, 
 like CMC surfaces in $\R^3,$ one will  construct
 an invariant potential. For this one usually needs to do two steps.
 The first step follows the Appendix of \cite{DPW}:
\begin{Theorem}[Lemma 4.11 in \cite{DPW}]
 If $\Rim$ is non-compact, then there exists some $($real analytic$)$ matrix function
 $\tilde V_+: \widetilde{ \Rim} \rightarrow \LSLP$ such that the matrix
 $\tilde C$ defined by
\[
 \tilde C(z,\lambda) := \tilde{F}(z, \bar{z},\lambda) \tilde{V}_+(z, \bar{z}, \lambda)
\] 
 is holomorphic in $z \in \widetilde{\Rim}$ and $\lambda \in \C^*$.
 \end{Theorem}
 Now $\tilde C$ inherits from its construction and 
 from $\tilde{F}$ the transformation behaviour
 \begin{equation}
  \tilde C(\tau.z, \lambda) = M(\tau, \lambda) \tilde C(z,\lambda) W_+(\tau, z, \lambda),
 \end{equation}
 where $\tau \in  \pi_1(\Rim)$  and  $W_+:\widetilde{ \Rim} \rightarrow \LSLP$ 
 is holomorphic in $z$ and $\lambda$. 
 The second step is to prove the existence of an invariant potential.
 \begin{Theorem} \label{invholframe}
  The matrix function $W_+ $ is a crossed homomorphism, that is, the identity
 \begin{equation*}
 W_+(\tau \mu,z,\lambda) = W_+(\tau, \mu.z,\lambda) W_+(\mu,z,\lambda) 
 \end{equation*}
  holds for all $\tau, \mu \in \pi_1(\Rim)$. 
 Moreover, there exists some holomorphic matrix function $P_+ : \widetilde{\Rim} \rightarrow \LSLP$ such that
\[
  W_+(\tau, z, \lambda) = P_+(z,\lambda)  P_+(\tau.z,\lambda)^{-1}. 
\] 
In particular, $C = \tilde C P_+$ satisfies
\[
 C(\tau.z,\lambda) = M(\tau, \lambda) C(z, \lambda)
\] 
  for all $\tau \in \pi_1(\Rim)$ and all 
 $\lambda \in \C^*$.
 \end{Theorem}
\begin{proof}
Following the proof of Theorem 3.2 of \cite{DH3} or the proof of Theorem 31.2 of \cite{Forster} and using Theorem 8.2 in \cite{Bungart} which implies the vanishing of 
$\mathrm{H}^1(\mathbb{D}, \Lambda^+{\rm SL}_2 \mathbb{C}_{\sigma})$, 
one obtains that the cocycle $W_+(\gamma,z,\lambda)$ splits in 
$\Lambda^+{\rm SL}_2 \mathbb{C}_{\sigma}$.
\end{proof}
 From Theorem \ref{invholframe} we immediately have the following 
 Corollary.
 \begin{Corollary}\label{Cor:invpot}
 The differential one-form $ \eta = C^{-1} dC$ 
is invariant under $\pi_1(\Rim)$ and is called an {\rm invariant 
 holomorphic potential}.
  In particular, each minimal surface of $\Nil$ can be constructed 
 from some invariant holomorphic potential.
 \end{Corollary}
\begin{proof}[Proof of {\rm Theorem \ref{invframe}}]
Let $\tilde{F}$ be as in Proposition \ref{prp:transformation} and  
 $C$ as in Theorem \ref{invholframe}. Then $\tilde{F} = C P_+^{-1}
 \tilde{V}_+^{-1} = C L_+$. Here $L_+$ is real analytic.
From the equation (\ref{trafoFtilde}) we now obtain
\[
 C(\tau.z, \lambda)L_+(\tau.z, \overline{\tau.z}, \lambda) = M(\tau, \l)
 C(z, \lambda) L_+(z, \bar{z}, \lambda) 
\tilde{k}(\tau, z, \bar{z}).
\]
Since $C(\tau.z,\lambda) = M(\tau, \lambda) C(z, \lambda)$ 
 this equation yields the equation
 \[
  L_+(\tau.z, \overline{\tau.z}, \lambda) = L_+(z, \bar{z}, \lambda) \tilde{k}(\tau, z, \bar{z})
 \]
 and this implies 
$\tilde k_0^{-1}(\tau.z, \overline{\tau.z}) = \tilde k_0^{-1}(z, \bar{z}) \tilde{k}(\tau, z, \bar{z})$, where $\tilde k_0^{-1}$ denotes the leading term of $L_+$, that is,
 the expansion of $L_+$ with respect to $\lambda$ is given by 
  $L_+ = \tilde k_0^{-1} + \lambda 
 L_{+1} + \cdots$.
  Note that in this equation we can assume without loss of generality
 that $\tilde k_0$ is unitary, and the claim follows. 
\end{proof}
  \subsection{From invariant potentials to surfaces}
 In this subsection we start from some Riemann surface $\Rim$ and consider  
 a holomorphic potential $\eta$ which is defined on the  simply-connected cover 
 $\widetilde{\Rim}$ 
 of $\Rim$ and is invariant under the fundamental group $\pi_1 (\Rim)$
 as in Corollary \ref{Cor:invpot}.
  Reversing the construction discussed above (which lead from 
 an immersion to an invariant potential), we first solve  the ODE
\[
 dC = C \eta,
\]
 with $C(z_0, \l) \in \LSL$ for some base point $z_0 \in 
 \widetilde \Rim$.
 It is easy to see that any such $C$ satisfies
\[
 C(\tau.z, \lambda) = \rho(\tau, \lambda) C(z,\lambda) 
\]
 for all $\tau \in \pi_1(\Rim)$ and  where 
 $\rho(-,\lambda): \pi_1(\Rim) \rightarrow \LSL$ is a homomorphism.
 From the discussion of the previous subsection we know that the 
 \emph{monodromy 
 matrix}  $\rho(\tau, \lambda)$ needs to be contained in $\LISU$. 
 We therefore need to consider two cases:

 { \bf The monodromy case 1:} The matrix 
 $\rho(\tau, \lambda)$ is contained in $\LISU$ for all $\tau \in \pi_1(\Rim)$. 
 This case will be discussed in Section \ref{subsection:mono1}.
  
 {\bf The monodromy case 2:} The matrix $\rho(\tau, \lambda) $ is  not contained in
  $\LISU$ for all $\tau \in \pi_1(\Rim)$, 
 but one can associate with $C$ another monodromy matrix 
 which is contained in $\LISU$. 
  This case will be discussed in Section \ref{subsection:mono2}.
\subsection{The monodromy case 1}\label{subsection:mono1}
 We want to retrieve the relation between $C$ and $F$.
For this purpose, we  quote \cite{Kellersch} (see also \cite[Theorem 2.1]{BRS:Min}):
\begin{Theorem}[Iwasawa decomposition]
 There is an open and dense subset $\mathcal{I} = \mathcal{I}_e \cup  \mathcal{I}_\omega $ of $\widetilde{\Rim}$  such that 
\[
 C(z, \lambda) \in  \LISU \cdot \LSLP
\] 
 if $z \in \mathcal{I}_e$,  and 
\[
 C(z, \lambda) \in  \LISU \cdot \omega_0 \cdot \LSLP 
\]
 if $z \in \mathcal{I}_\omega,$
 where  $\omega_0 = \left(\begin{smallmatrix}
			 0 & \l \\ -\l^{-1}& 0 
			\end{smallmatrix}\right).$ 
\end{Theorem}
 The open dense subset $\mathcal{I}$ will
 be called the \textit{Iwasawa core}. 
 It consists of two connected open cells, called \textit{Iwasawa cells}.
 The next step in our construction procedure will 
 be an Iwasawa decomposition of $C$.
 We distinguish the two cases listed in the theorem above.

 \begin{Theorem}\label{thm:invariantcase1}
 Let $\eta$ be an invariant potential on $\widetilde \Rim$ 
 and $C$ a solution to $dC = C \eta$.
 Assume that the monodromy representation $\rho$ of $C$ 
 relative to $\pi_1(\Rim)$ takes value in  $\LISU$. 
 For $z \in \mathcal{I}_e$, take the (unique) Iwasawa 
 decomposition
\begin{equation} \label{IwasawaIe}
C(z, \lambda) = F(z, \bar z, \lambda) V_+(z, \bar z, \lambda),
\end{equation}
 where the diagonal entries of $V_{+0}$
 for the expansion 
 $V_+ = V_{+0} + \lambda V_{+1} + \lambda^2 V_{+2} \cdots$ are
 assumed to be 
 positive.
 Then
\begin{enumerate} 
 \item  For each symmetry $(\tau,\rho(\tau, \lambda))$ of $C$
  the automorphism $\tau \in \pi_1 (\Rim)$
 leaves $\mathcal{I}_e$  and  $\mathcal{I}_\omega$ invariant and acts bi-holomorphically there. 
 \item  $F(\tau.z, \overline{\tau.z}, \lambda) = \rho(\tau, \lambda) F(z,\bar z,  \lambda)$ for all $z \in \mathcal{I}_e$.
\end{enumerate}
 \end{Theorem}

  \begin{proof}
$(1)$ By the definition of a symmetry we have $C(\tau.z,\lambda) 
 = \rho(\tau, \lambda) C(z,\lambda)$ with $\rho(\tau, \lambda) \in \LISU$. Using \eqref{IwasawaIe} we derive
  $C(\tau.z,\lambda) = \rho(\tau, \lambda) F(z, \bar z, \lambda) V_+(z, \bar z, \lambda)$. 
 This is an Iwasawa decomposition with factors 
  $\rho(\tau, \lambda) F(z, \bar z, \lambda) $ and 
 $V_+(z, \bar z, \lambda)$.
 Hence $\tau.z \in \mathcal{I}_e$.
 Let now $w \in \mathcal{I}_\omega$. Then $\tau(w) \notin \mathcal{I}_e,$ since 
 $\tau$ leaves $\mathcal{I}_e$ invariant. Since $\tau$ is an open map, the image of 
 $\mathcal{I}_\omega$ under $\tau$ can not attain a  point 
 in  $\widetilde{\Rim} \setminus \mathcal{I}_e \cup \mathcal{I}_\omega$ either.

 $(2)$ The general theory tells us   
 $F(\tau.z, \overline{\tau.z}, \lambda) = \rho(\tau, \lambda) F(z,\bar z,  \lambda) k(z, \bar z).$  On the other hand, we obtain from (\ref{IwasawaIe}) the equations $F(\tau.z, \overline{\tau.z}, \lambda) V_+(\tau.z, \overline{\tau. z}, \lambda)= 
 C(\tau.z,\lambda) = \rho(\tau, \lambda) C(z, \lambda) = 
 \rho(\tau, \lambda) F(z, \bar z, \lambda) V_+(z, \bar z, \lambda).$ Hence 
 $k(z, \bar z) = V_+(z, \bar z, \lambda)V_+(\tau.z, \overline{\tau. z}, \lambda)^{-1}$ and $k$ is actually the leading term of this product. But by assumption, the leading term is positive real, while $k$ is unitary. Therefore $k = \id$.
 \end{proof}
\begin{Remark}
 The frame $F$ obtained by Theorem \ref{thm:invariantcase1}
  is 
 a general extended frame of a harmonic map into $\mathbb H^2$ in 
 the sense of 
 Definition \ref{dfn:generalext}, and it is an extended frame of 
 some minimal surface in $\Nil$.
\end{Remark}
 Note, as a consequence of part $(1)$ above, $\tau$ also acts bijectively on $\widetilde{\Rim} \setminus \mathcal{I}_e \cup \mathcal{I}_\omega$.
 To discuss the behaviour of the extended frame under 
 $\tau \in \pi_1(\Rim)$ on $z \in \mathcal I_w$, 
 in the next subsubsection
 we consider an analytic continuation of a minimal surface 
 defined on $z \in \mathcal I_e$ to a minimal surface 
 defined on $z \in \mathcal I_w$ using 
 a unique meromorphic extension.
 

\subsubsection{Meromorphic extension of a minimal surface} \label{sbsc:MeroExt}
 In this subsubsection we extend a result of \cite{DIKH3} to the present paper. We start by explaining what this result means for the surfaces considered in  
 \cite[Section 9.3, 9.4]{DIKH3}, that is,
 the  constant mean curvature $0<H<1$ surfaces in the 
hyperbolic $3$-space $\mathbb{H}^3$. 
Let $\D$ be a simply connected domain in $\C$ and $e \in \D$.  Moreover, let
$\eta$ be a holomorphic potential for a surface of the class considered. Then, solving the ODE $dC = C \eta, C(e, \lambda) = \id$ we obtain a ``holomorphic extended frame'' 
 defined on $\D$.
 It turns out that the ``Gauss map'' has as target space a 
 non-compact $4$-symmetric space $\mathrm{SL}_2\mathbb{C}/\mathrm{U}_1$. 
 The Lie group $\mathrm{SL}_2\mathbb{C}$ defining this $4$-symmetric space is non-compact. 
 In particular, not each matrix in the twisted loop group of 
 $\mathrm{SL}_2\mathbb{C}$ associated to 
 the $4$-symmetric space $\mathrm{SL}_2\mathbb{C}/\mathrm{U}_1$
 has an Iwasawa decomposition of the form \eqref{IwasawaIe}.
 However, as in the case of the present paper, 
 there exist two open Iwasawa cells, $\mathcal I_e$ and  $\mathcal I_w$ 
 for which $C$ has a decomposition similar to what was stated in the 
 Iwasawa decomposition Theorem just above. Applying the Sym-formula 
 to the frame obtained by the Iwasawa decomposition for 
 $z \in  \mathcal I_e$ one obtains a surface of the type considered 
(actually a surface on each connected component of  $\mathcal I_e$.
 It is not difficult to show that these surfaces are uniquely determined by $C$.)
 One can apply a similar procedure for the set $\mathcal I_w$. 
This way one always obtains (at least) two surfaces, one on  $\mathcal I_e$ and one on  $ \mathcal I_w$.
 How are these surfaces related? One can show that, in general, any extended frame 
 defined from $C$ by Iwasawa decomposition experiences a catastrophic singularity along the boundary between 
 $\mathcal I_e$ and  $ \mathcal I_w$.
It is now of great importance, 
that each constant mean curvature $H<1$ surface in the 
hyperbolic $3$-space $\mathbb{H}^3$ defined by the extended frame 
(via the Sym formula for constant mean curvature $H<1$ surfaces in 
$\mathbb{H}^3$) has a meromorphic 
extension to two complex variables $(z,w) \in \D \times  \overline{\D}$. 
 Thus this extension is a complex(ified) meromorphic surface which restricts on  $\mathcal I_e \cup \mathcal I_w$ to 
 meromorphic surfaces of constant mean curvature $0<H<1$. 
 Loosely speaking, each  constant mean curvature $0 < H < 1$ surface in $\mathbb{H}^3$
 defined on the first cell $\mathcal{I}_e$ can 
 be analytically continued to the second cell $\mathcal{I}_\omega$. 
For more details we refer to \cite[Section 9.4]{DIKH3}(see also \cite[Theorem 3.2]{DK:cyl}).

 There is only little known about how these real surfaces are related. In general, these surfaces are highly singular along the boundary between  $\mathcal I_e$ and  $ \mathcal I_w$.
 But in some cases the surfaces extend smoothly across the boundary (with vanishing functional determinant, of course.) 
 See \cite{K:H3} for some results in this direction.
 
 Analogously, in the situation considered in this paper,   
 the Sym formula in \eqref{eq:symNil}
 for minimal surfaces in $\Nil$ 
 defined on $\mathcal{I}_e$ can be 
 analytically continued to $\mathcal{I}_\omega$. 
 This works as follows:
 Let $ C = F V_+$ be an Iwasawa decomposition for $z \in \mathcal{I}_e$.
 In view of \cite[Theorem 3.2]{DK:cyl}, which can be checked 
 to also hold in the present case, one can extend $F l$ 
 meromorphically to $\D \times \overline{\D}$, where $l$ is 
 a properly chosen  
 $\lambda$-independent diagonal matrix. Moreover, note that 
 the proof of \cite[Theorem 3.2]{DK:cyl} shows that
  $l_0^2>0$ for $z \in \mathcal I_{e}$ and $l_0^2<0$ 
 for $z \in \mathcal I_{w}$, where $l_0$ is the $(1, 1)$-entry of $l$. 
 These facts are proven in Appendix 
 \ref{sc:Meroextension} below in detail.
 Then the Sym formula $f_{\Min}$ for spacelike surface in $\Min$ 
 in \eqref{eq:SymMin} 
 can be rephrased as
 \begin{equation*} 
   f_{\Min}=-i \l (\partial_{\l} (F l) ) (Fl)^{-1} 
 - \frac{i}{2} \ad (Fl) \sigma_3,
\end{equation*}
 where $\sigma_3 =\left(\begin{smallmatrix}1 & 0 \\ 0 & -1 
 \end{smallmatrix}\right)$. Then $f_{\Min}$
 clearly has a meromorphic  extension to 
 $\D \times \overline{\D}$. Therefore the formula in \eqref{eq:symNil}
\begin{equation*}
 \hat f = 
    (f_{\Min})^o -\frac{i}{2} \l (\partial_{\l}  f_{\Min})^d, 
\end{equation*}
 and the whole Sym formula  have accordingly  
 a meromorphic extension to $\D \times \overline{\D}$.
 Note, so far we have used the meromorphic extension of the 
 frame obtained by an Iwasawa decomposition for values 
 in the \emph{first Iwasawa cell} $\mathcal{I}_e$.

 Next we want to express this formula for the immersion by a 
 formula using the frame occurring in the 
 Iwasawa decomposition of $C(z,\lambda)$ for $z \in  \mathcal{I}_\omega$.
 Let $C = \tilde F \omega_0 \tilde V_+$ be an Iwasawa decomposition 
 for   $z \in \mathcal I_{w}$. 
 On the one hand, choosing a $\lambda$-independent diagonal matrix
 $k$ with positive entries such that $k^{-2} = - l^{-2}$
 (note that the $(1, 1)$-entry $l_0$ of $l$ satisfies 
 $- l_0^{-2}>0$ for $z \in \mathcal I_w$), 
 we have that
 \begin{equation} \label{IwasawaIomega}
  C = (F l k^{-1} \omega_0^{-1} ) \omega_0 (k l^{-1} V_+)
 \end{equation}
 is the Iwasawa decomposition for  $z \in \mathcal I_{w}$, see 
 Appendix \ref{subsc:AppIwasawa} below.
 The formula just above yields, written out, the original formula
 $C = F V_+$.  This is also an Iwasawa decomposition 
 for the \emph{second Iwasawa cell}, thus $\tilde F = F l k^{-1} \omega_0^{-1}$.
 Therefore
 \[
 F l = \tilde F \omega_0 k.
\]
 Then, for $z \in \mathcal I_w$, $f_{\Min}$ can be rephrased as 
\[
 f_{\Min} = -i \l (\partial_{\l} (\tilde F \omega_0 )) 
 (\tilde F \omega_0)^{-1} 
 - \frac{i}{2} \ad (\tilde F \omega_0) \sigma_3
\]
 Thus it is natural to use for $z \in \mathcal{I}_w$  
 formula \eqref{eq:symNil}  and the
 whole Sym formula  and to use this formula for 
 $\tilde F \omega_0$. Therefore in the second Iwasawa cell 
 actually $\tilde F \omega_0$ is ``the frame'' to use.

\subsubsection{Symmetries of the meromorphic extension}
 Here we discuss symmetries of the meromorphic extension of 
 a minimal surface.
Like in \cite[Section 3]{DK:cyl}  we consider the pair of potentials 
$\left(\eta(z, \lambda), \varphi(\eta(w, \lambda))\right)$, where $\varphi$ denotes the involution of the loop algebra 
$\Lambda\mathfrak{sl}_{2}\mathbb{C}_{\sigma}$ defined by 
\eqref{eq:varphi} which determines the 
\emph{real form} $\Lambda\mathfrak{su}_{1,1\sigma}$, the 
Lie algebra of $\LISU$.

Assume that $\eta$ is an invariant potential 
 under $\pi_1(\Rim)$, thus $\varphi(\eta)$ is also invariant under 
 $\pi_1(\Rim)$.
Consider the pair of differential equations
 \begin{equation*}
  d (C, R) = (C, R) (\eta, \varphi (\eta)).
\end{equation*}
Then we obtain for the second potential the solution 
 $R(w, \l) = \varphi (C(w, \l))$, where $\varphi$ denotes the real 
 form involution on the group level. 
 Assume that 
\[
 \rho (\tau, \l) \in \LISU, 
\]
for some $\tau \in \pi_1 (\Rim)$.
 Then relative to $(\tau, \rho)$ both solutions have the same monodromy matrix, that is,
\[
 C (\tau. z) = \rho(\tau, \l)  C(z), \quad 
 R (\bar \tau. w) = \rho(\tau, \l)  R(w).
\]
 By using \eqref{eq:Birkhoffdouble} and \eqref{eq:U}, 
 we have 
\begin{equation*}
U(z,w,\lambda)=C(z,\lambda) V_+^{-1} (z, w, \lambda) 
= R(w,\lambda) V_-^{-1} (z, w, \lambda) B(z, w),
\end{equation*}
 whence 
\begin{equation} \label{doublespliting}
R(w,\lambda) ^{-1} C(z,\lambda) = V_{-}(z,w,\lambda)^{-1} B (z,w)  V_+(z,w,\lambda),
\end{equation}
 where $V_{-}(z,w,\lambda)$ and   $V_+(z,w,\lambda)$ have leading term 
 $\id$ and $B$ is diagonal.

 In this form all three factors are uniquely determined. Therefore, since the left side does not change, if one replaces $w$ by $\bar \tau .w$ and $z$ by $\tau .z$, this also holds for the three factors on the right side. 
Substituting this into \eqref{doublespliting},
 we obtain the equations
\begin{gather*} 
R(\bar{\tau}. w,\lambda) ^{-1} C(\tau.z,\lambda) = 
R(w,\lambda) ^{-1} C(z,\lambda), \\
V_{\pm}(\tau.z, \bar{\tau}. w,\lambda)  = V_{\pm}(z,w,\lambda),
 \quad\mbox{and}\quad B (\tau.z,\bar{\tau}.w) = B(z,w).
\end{gather*}
Then 
\begin{equation*}
 U(\tau.z, \bar{\tau}.w, \l) = 
 \rho(\tau,\lambda) U(z,w,\lambda)  S_+(z, w, \lambda),
\end{equation*}
 for some plus matrix $S_+$. Since $\hat \varphi U = UB^{-1}$, 
 it follows that $S_+$ is diagonal.
\subsubsection{The case $ C(z, \lambda) \in  \LISU \cdot \omega_0 \cdot  \LSLP$ in the monodromy case 1}
 For $z \in \mathcal{I}_{\omega}$ we choose the (unique) Iwasawa 
 decomposition
\begin{equation} \label{IwasawaIw}
C(z, \lambda) = \tilde F(z, \bar z, \lambda) \omega_0 \tilde V_+(z, \bar z, \lambda),
\end{equation}
 where the diagonal of the first term of $\tilde V_+$ is assumed positive.
In this subsubsection it is our goal to find a transformation formula for symmetries of the surface over $\mathcal{I}_\omega$ generated by some potential $\eta$. We  recall that one should use the Iwasawa decomposition formula \eqref{IwasawaIomega} and hence should use  
\[
\tilde{F}\omega_0 = F l k^{-1}
\]
in the usual Sym formula not $F$.
This was obtained above by using  \cite[Theorem 3.2]{DK:cyl} generalized to our present case, see Appendix  \ref{sc:Meroextension} 
 for  details.  To find the correct transformation formula for symmetries we need to proceed analogously.
\begin{Theorem} 
 Retain the assumptions of Theorem {\rm \ref{thm:invariantcase1}}
 and choose the unique Iwasawa 
 decomposition $C = \tilde F \omega_0 \tilde V_+$
 for $z \in \mathcal I_w$ as in \eqref{IwasawaIw}.
 Then for all $z \in \mathcal{I}_w$,
 \[
\tilde F(\tau.z, \overline{\tau.z}, \lambda) \omega_0 = \rho(\tau, \lambda) \tilde F(z,\bar z,  \lambda) \omega_0.
\]
\end{Theorem}
\begin{proof}
 The general theory tells us   
 $\tilde F(\tau.z, \overline{\tau.z}, \lambda) \omega_0(\l)= 
 \rho(\tau, \lambda) 
 \tilde F(z,\bar z,  \lambda) \omega_0(\l) \tilde k(z, \bar z).$  On the other hand, we obtain from (\ref{IwasawaIe}) the equations 
 \[
  \tilde F(\tau.z, \overline{\tau.z}, \lambda) \omega_0  \tilde V_+(\tau.z, \overline{\tau. z}, \lambda)= 
 C(\tau.z,\lambda) = \rho(\tau, \lambda) C(z, \lambda) = 
 \rho(\tau, \lambda) \tilde F(z, \bar z, \lambda) \omega_0 
 \tilde V_+(z, \bar z, \lambda).
 \] 
 Hence  $\tilde k(z, \bar z) = \tilde V_+(z, \bar z, \lambda)\tilde V_+(\tau.z, \overline{\tau. z}, \lambda)^{-1}$ and $\tilde k$ is actually the leading term of this product. But by assumption, the leading term is positive real, while $\tilde k$ is unitary. Therefore $\tilde k = \id$.
\end{proof}

\subsubsection{The closing condition}
 Let us consider next a single symmetry $(\tau,\rho(\tau, \lambda))$ 
 of $C(z,\lambda)$.
 Then from Theorem \ref{thm:symmetry} 
 we infer that $\tau$ can induce a symmetry of some minimal 
 surface in $\Nil$ if and only if $\rho(\tau, \lambda = 1)$ 
 has only unimodular eigenvalues.
 Let us consider now $\hat{F} = S F$, 
 where 
\begin{equation}\label{eq:diagonalU}
 \mbox{$S(\lambda)$ takes values in 
 $\LISU$ and $S(\l =1)$
 diagonalizes $\rho (\tau, \lambda = 1)$.}
\end{equation}
 Then we obtain the following theorem.
\begin{Theorem}\label{thm:closing}
  Retain the notation and the assumptions of Theorem 
 {\rm \ref{thm:invariantcase1}} and assume that $S$ satisfies 
 \eqref{eq:diagonalU}. Let $\hat{f}$ be the minimal surface in $\Nil$ 
 defined on  $\mathcal{I}_e$ or $\mathcal{I}_w$
 and defined from $\hat{F} = S F$ or $S \tilde F\omega_0$
 via the Sym formula \eqref{eq:symNil}.
 Then the monodromy matrix $M(\tau,\lambda) 
 = S(\lambda) \rho(\tau, \lambda) S(\lambda)^{-1}$ is in $\LISU$ 
 has only unimodular eigenvalues and  
 is diagonal
 for $\lambda=1$.  Moreover, $\hat{f}|_{ \lambda = 1}$ satisfies
\[
 \hat {f}(\tau. z, \overline{\tau. z}, {\lambda = 1})
 = \hat {f}(z, \bar z, {\lambda = 1})
\]
 for all $z\in \mathcal{I}_e$ or $z\in \mathcal{I}_w$
  if and only if
\begin{equation}\label{eq:closing}
 M(\lambda =1) =  \pm \id, \quad
 X^o(\lambda =1) = 0 \quad \mbox{and} \quad Y^d(\lambda =1)= 0
\end{equation}
 holds, where $ X = -i \l (\partial_{\l} M) M^{-1}$ and 
 $Y =  - \frac{1}{2}  \lambda \partial_{\lambda} (\lambda (\partial_{\lambda} M) M^{-1})$,
 respectively.
\end{Theorem}

\begin{proof} We abbreviate $\hat f(z, \bar z, \l =1)$ by $\hat f(z)$.
 We want to characterize what it means that 
 $\hat{f}(\tau.z) = \hat{f}(z)$ holds.
 Using the definition of the action of the group of isometries
 we obtain (setting  $\hat{f} =(\hat{f}_1, \hat{f}_2, \hat{f}_3)$) as in the proof of Part (b) in Theorem \ref{thm:symmetry} :
 \begin{align*}
(\hat{f}_1(\tau. z), \hat{f}_2(\tau.z), \hat{f}_3(\tau.z)) &= ((p, q, r), e^{i \theta}) . (\hat{f}_1(z), \hat{f}_2(z), \hat{f}_3(z)) \\
 &=(p, q, r) \cdot (\cos \theta \hat{f}_1(z) - \sin \theta \hat{f}_2(z), \sin \theta \hat{f}_1(z) + 
 \cos \theta \hat{f}_2(z), \hat{f}_3(z)),
\end{align*}
 where $\theta$ and $(p, q, r)$ are 
 defined by $M(\tau, \lambda = 1)  = \di (e^{i\frac{\theta}{2}}, e^{-i \frac{\theta}{2}})$, 
\[
 X|_{\lambda =1} 
 =\frac{1}{2} 
 \begin{pmatrix} 
  * &  -q + i p  \\
  -q -i p &  *
 \end{pmatrix}, 
 \quad \mbox{and} \quad 
 Y|_{\lambda =1} =
\frac{1}{2}
 \begin{pmatrix} 
  - i r &  * \\
  *& i r
 \end{pmatrix},
\]
 respectively. As a consequence, the 
 following conditions are equivalent to $ \hat{f}(\tau.z) =  \hat{f}(z)$ :
\begin{align*}
 &p + \cos \theta  \hat{f}_1(z) - \sin \theta  \hat{f}_2(z)  =  \hat{f}_1(z), \quad q + \sin \theta  \hat{f}_1(z) + \cos \theta  \hat{f}_2(z) =  \hat{f}_2(z) \\
 &r +  \hat{f}_3(z) + \frac{1}{2}( p( \sin \theta  \hat{f}_1(z) + \cos \theta  \hat{f}_2(z) ) - q(\sin \theta  \hat{f}_1(z) 
 + \cos \theta  \hat{f}_2(z) )= \hat{f}_3(z).
\end{align*}
 It is easy to verify that the first two equations only have a $z$-independent solution 
 if $\cos \theta \neq 1$. This does not make sense in our case, 
 since $f$ defines a surface.
 We thus can assume without loss of generality 
 that $\cos \theta = 1$. But in this case $p = q = r = 0$ 
 and the claim follows, since $M$, $X^{o}$ and $Y^{d}$
 clearly satisfy the conditions \eqref{eq:closing}.
\end{proof}
 The condition $ M (\tau, \lambda =1) =   \id$ implies that 
 we can choose without loss of generality
 $S(\lambda) \equiv \id$ above. Hence we obtain
\begin{Corollary}\label{cor:closing}
 Retain the notation and the assumptions of Theorem 
 {\rm \ref{thm:closing}}. Let $\hat{f}$ be the minimal 
 surface in $\Nil$ defined on $\mathcal{I}_e$ or $\mathcal{I}_w$
 and defined from 
 $F$ or $\tilde F \omega_0$ via the Sym formula  \eqref{eq:symNil}.
 In particular, assume that
 the monodromy matrices $M(\tau,\lambda) = \rho(\tau,\lambda)$ 
 are  in  $\LISU$ and all $\tau \in \pi_1(\Rim)$ 
 and attain the value $\id$  for $\lambda = 1$.
 Then  $\hat{f}|_{\lambda = 1}$ satisfies for all $z\in \mathcal{I}_e$ 
 or $z\in \mathcal{I}_w$ and all $\tau \in \pi_1(\Rim):$ 
\[
 \hat{f}(\tau. z, \overline{\tau. z} , \lambda = 1) = \hat{f}(z, \bar z, \lambda = 1)
\]
 if and only if the following holds$:$
\begin{equation*}
 X^o(\lambda =1) = 0 \quad \mbox{and} \quad Y^d(\lambda =1)= 0.
\end{equation*}
\end{Corollary}

\begin{Remark}
If the general extended frame is in one of the two open cells, then it will stay in the same open cell when subjected to the action of some symmetry. As a consequence, if a frame ever reaches the boundary between the two open Iwasawa cells, then it will stay there under the action of any symmetry. If $(\tau, \rho)$ denotes a symmetry of some $f$, then the image $f(\D)$ is the union of three parts:
$f(\mathcal{I}_e)$,  $f(\mathcal{I}_\omega),$ and $f(\mathcal{B})$, where $\mathcal{B}$ denotes the boundary between the open Iwasawa cells. 
\end{Remark}


\subsection{The monodromy case 2}\label{subsection:mono2}
 We respectively discuss the monodromy case 2 with 
 $z \in \mathcal I_e$ or $z \in \mathcal I_{\omega}$. 
\subsubsection{The case of $z \in \mathcal I_e$} 
 For the construction of a symmetry $(\gamma, \rho)$ one 
 frequently starts 
 from some potential $\eta$, which is (say up to a gauge) invariant under $\gamma$
\[
 \eta \circ \gamma = \eta \# W_+,
\]
where $W_+ : \mathbb D \to \LSLP$ and where $\#$ means ``gauging'', that is,
\[
 \eta \# W_+ = W_+^{-1} \eta  W_+ + W_+^{-1} d W_+.
\]
 Note that $\eta$ is an invariant potential under $\gamma$ 
 if $W_+= \id$.
 Then the solution $C(z, \lambda)$ to 
\[
 d C = C \eta
\]
 with some initial condition $C(z = z_0, \l) \in \LSL, \; z \in 
 \mathcal I_e$ satisfies 
\begin{equation}\label{eq:Cgamma}
 C(\gamma. z, \lambda) = L (\gamma, \lambda) C(z, \lambda) W_+(\gamma, 
 z, \lambda)
\end{equation}
 for some $L \in \LSL$. If $L \in \LISU$, then the Iwasawa decomposition $C = F V_+$
 implies 
\[
 F( \gamma. z, \overline{\gamma. z}, \lambda ) = L(\gamma, \lambda) F(z, \bar z, \lambda) 
 k(\gamma, z, \bar z),
\]
 for some diagonal matrix $k \in \Uone$.
 In general one will obtain $L \notin \LISU$.
 Then the formula just above can not be obtained. So it seems impossible to obtain a symmetry associated with the action of $\gamma$.
 However, in some cases a symmetry $(\gamma, \rho)$ does exist
 (see for example \cite{DH:periodic}).
 Then in addition to \eqref{eq:Cgamma} we also have 
\[
 C(\gamma. z, \lambda) = \rho (\gamma, \lambda) C (z, \lambda) Q_+ (\gamma, z, \lambda),
\]
with $\rho (\gamma, \lambda) \in \LISU$.
 Then 
\[
  L(\gamma, \lambda)^{-1} \rho (\gamma, \lambda) C (z, \lambda) = C (z, \lambda) W_+ (\gamma, z, \lambda)
 Q_+ (\gamma, z, \lambda)^{-1}.
\]
 Since we consider surfaces defined on $ \mathcal I_e$ we choose a base point 
$z_0 \in \mathcal I_e$ such that $C(z_0,\lambda) = \id$.
 Putting $z = z_0$ yields 
 \[
  L(\gamma, \lambda)^{-1}\rho(\gamma, \lambda) = W_+ (\gamma, z_0, \lambda) Q_+ (\gamma, z_0, \lambda)^{-1}.
 \]
 As a consequence 
\[
 \rho (\gamma, \lambda) = L(\gamma, \lambda) b_+ (\gamma, \lambda)  \in \LISU
\]
 and 
\[
 b_+(\gamma, \lambda) C (z, \lambda) = C(z, \lambda) B_+ (\gamma, z, \lambda)
\]
 with $B_+ (\gamma, z, \lambda) = W_+ (\gamma, z, \lambda) Q_+(\gamma, z, \lambda)^{-1}$.
\begin{Theorem}\label{thm:charactmono}
 Assume $\eta$ is a potential for a  minimal surface in $\Nil$ and 
 satisfies 
\[
 \eta\circ \gamma = \eta \# W_+
 \]
 for some $W_+ \in \LSLP$, $\gamma \in \Aut (\mathbb D)$ and 
 where $\#$ denotes gauging. Then for the solution to $d C  = C \eta, 
 C(z_0, \lambda) = \id$ for some fixed base point $z_0 \in \mathcal I_e$,  we obtain
 \[
  \gamma^* C = L C W_+,
 \]
 where $L \in \LSL$. 
 Moreover, the following statements are equivalent$:$
\begin{enumerate}
 \item There exists a $\rho \in \LISU$  such 
 that $(\gamma, \rho)$ is a symmetry of the minimal surface in $\Nil$
 associated with $\eta$. 
  \item There exists some $b_+ \in \LSLP$ such that 
 the following conditions are satisfied$:$
\begin{enumerate} 
\item $L(\lambda) b_+ (\lambda)^{-1} \in \LISU$, 
\item $b_+ (\lambda) C(z, \lambda) = C(z, \lambda) B_+ (z, \lambda)$
 for some $B_+ (z, \lambda) \in \LSLP$, 
 \item $L(\lambda) b_+ (\lambda)^{-1}|_{\l=1}$ has unimodular eigenvalues.
\end{enumerate}
\end{enumerate}
\end{Theorem}
\begin{proof}
 From the discussion above, the necessary part is clear.
 Thus we only need to prove sufficiency. But
 $C \circ \gamma = L C W_+ = Lb_+^{-1} b_+ CW_+  = \rho (\lambda)C B_+ W_+$
 with $\rho (\lambda) = L(\lambda) b_+ (\lambda)^{-1}$. Since $\rho$ is in $\LISU$, 
 the statement is proven.
\end{proof}
\begin{Remark}
\mbox{}
\begin{enumerate}
\item
 The third condition in $(2)$ of Theorem \ref{thm:charactmono}, 
 that is, $L(\lambda) b_+ (\lambda)^{-1}|_{\l=1}$ 
 has unimodular eigenvalues, is purely local, since in general the 
 eigenvalues of $L(\lambda) b_+ (\lambda)^{-1}$ on 
 $\l \in \mathbb{S}^1$
 are not unimodular, see Remark \ref{rm:eigenvalues}. 
 \item We will apply this result to the construction of 
 equivariant minimal surfaces with a complex period elsewhere.

 \item Note, the case just discussed can only happen, 
 if there exist several ``monodromy matrices'' $M(\gamma,\lambda)$
 and ``gauges'' $T_+(\gamma,z,\lambda)$ satisfying
 $ C(\gamma. z, \lambda) = M (\gamma, \lambda) C (z, \lambda) T_+ (\gamma, z, \lambda)$.
 In particular, the isotropy group of the dressing action is ``non-trivial'' at  the surface determined by $C(z,\lambda)$.
\end{enumerate}
\end{Remark}
\subsubsection{The case of $z \in \mathcal I_\omega$}
 This case is similar to the case of $ z \in \mathcal I_e$.
 We again consider some potential $\eta$, which is (say up to a gauge) invariant under $\gamma$
\[
 \eta \circ \gamma = \eta \# W_+,
\]
 where $W_+ : \mathbb D \to \LSLP$.
 Then any solution $C(z, \lambda)$ to 
\[
 d C = C \eta
\]
 with some initial condition $C(z = z_0, \l) \in \LSL, \; z_0 \in 
 \mathcal I_\omega$
 satisfies 
\begin{equation}\label{eq:Cgamma2}
 C(\gamma. z, \lambda) = L (\gamma, \lambda) C(z, \lambda) W_+(\gamma, 
 z, \lambda)
\end{equation}
 for some $L \in \LSL$. If $L \in \LISU$, then the Iwasawa decomposition $C = \tilde{F} \omega_0 \tilde{V}_+$
 implies 
\[
 \tilde{F}( \gamma. z, \overline{\gamma. z}, \lambda ) \omega_0 = L(\gamma, \lambda) \tilde{F}(z, \bar z, \lambda) \omega_0  H_+(z, \bar z,\lambda)
\]
for some matrix $H_+$. But since we have assumed $L$ to be in $ \LISU$, we obtain
$H_+ \in  \LISU$, whence $H_+(z, \bar z, \lambda ) = k(\gamma, z, \bar z)$
 for some diagonal matrix $k \in \Uone$.

 In general one will obtain $L \notin \LISU$.
 Then the formula just above can not be obtained. So it seems impossible to obtain a symmetry associated with the action of $\gamma$.
 However, in some cases a symmetry $(\gamma, \rho)$ does exist
 (see for example \cite{DH:periodic}).
 Then in addition to \eqref{eq:Cgamma2} we also have 
\[
 C(\gamma. z, \lambda) = \rho (\gamma, \lambda) C (z, \lambda) Q_+ (\gamma, z, \lambda),
\]
with $\rho (\gamma, \lambda) \in \LISU$.
 Then 
\[
  L(\gamma, \lambda)^{-1} \rho (\gamma, \lambda) C (z, \lambda) = C (z, \lambda) W_+ (\gamma, z, \lambda)
 Q_+ (\gamma, z, \lambda)^{-1}.
\]
Since we consider surfaces defined on $ \mathcal I_\omega$ we choose a base point 
$z_0 \in \mathcal I_\omega$ such that $C(z_0,\lambda) = \omega_0$.
 Putting $z = z_0$ in the last equation above yields 
 \[
  L(\gamma, \lambda)^{-1}\rho(\gamma, \lambda)  \omega_0 =
   \omega_0 W_+ (\gamma, z_0, \lambda) Q_+ (\gamma, z_0, \lambda)^{-1}.
 \]
 As a consequence, setting $b =  \omega_0 W_+ (\gamma, z_0, \lambda) Q_+ (\gamma, z_0, \lambda)^{-1} \omega_0^{-1}$, we derive
\[
 \rho (\gamma, \lambda) = L(\gamma, \lambda) b (\gamma, \lambda)  \in \LISU
\]
 and 
\[
 b(\gamma, \lambda) C (z, \lambda) = C(z, \lambda) B_+ (\gamma, z, \lambda)
\]
 with $B_+ (\gamma, z, \lambda) = W_+ (\gamma, z, \lambda) Q_+(\gamma, z, \lambda)^{-1}$ and 
 \[
 \omega_0^{-1}  \; b  \; \omega_0  \in \LSLP.
 \]
\begin{Theorem}\label{thm:charactmono2}
 Assume $\eta$ is a potential for a  minimal surface in $\Nil$ and 
 satisfies 
\[
 \eta\circ \gamma = \eta \# W_+
 \]
 for some $W_+ \in \LSLP$, $\gamma \in \Aut (\mathbb D)$ and 
 where $\#$ denotes gauging. 
 Then for the solution to $d C  = C \eta, 
 C(z_0, \lambda) = \omega_0$ for some fixed base point 
 $z_0 \in  \mathcal I_\omega$ we obtain
 \[
  \gamma^* C = L C W_+,
 \]
 where $L \in \LSL$. 
 Moreover, the following statements are equivalent$:$
\begin{enumerate}
 \item There exists a $\rho \in \LISU$  such 
 that $(\gamma, \rho)$ is a symmetry of the minimal surface in $\Nil$
 associated with $\eta$. 
  \item There exists some $b\in \LSL$ such that the following conditions are satisfied$:$
  \begin{enumerate}
  \item $\omega_0^{-1} b \omega_0 \in \LSLP$,
  
  \item $Lb \in  \LISU$, 
  
  \item $bC = C B_+$ for some  $B_+ (z, \lambda) \in \LSLP$, 
  
  \item $L(\lambda) b(\lambda )|_{\lambda = 1}$  has unimodular eigenvalues.
  \end{enumerate}
 \end{enumerate}
\end{Theorem}
\begin{proof}
 From the discussion above, the necessary part is clear.
 Thus we only need to prove sufficiency. But
 $C \circ \gamma = L C W_+ = Lb b^{-1} CW_+  = \rho (\lambda)C B_+^{-1} W_+$
 with $\rho (\lambda) = L(\lambda) b (\lambda)$, where we have used that item (b) above also holds for $b^{-1}$ and $B_+^{-1}$. Since $\rho$ is in $\LISU$, 
 the statement is proven.
\end{proof}

\section{Minimal cylinders}\label{sc:Mincyl}
The construction method for minimal surfaces in $\Nil$ outlined above applies to all minimal surfaces in $\Nil$ which
have a non-trivial fundamental group. The case of a trivial fundamental group has already 
been discussed in \cite{DIKAsian}. 

For most subclasses of minimal surfaces in $\Nil$, 
as generally for all (sub-)classes of ``integrable surfaces'', a thorough discussion usually requires additional and special techniques.
Most of the rest of this paper is devoted to a discussion of 
 ``equivariant'' minimal surfaces in $\Nil$.
This also includes the class of homogeneous surfaces mentioned in the next section.

Another natural class of surfaces consists of all minimal cylinders in $\Nil$.
A  thorough discussion of this class of minimal surfaces in $\Nil$ would go beyond the scope of this paper, but will be presented in  \cite{DK:cylNil}.

In this section we will present an example of a non-equivariant minimal cylinder in $\Nil$.
We have proven mathematically all the required properties 
(in particular the closing conditions for the period) in \cite{DK:cylNil}, but will  point out
here only the basic data and show some pictures computed following the loop group method presented in this paper.

\begin{Example}[A minimal cylinder in $\Nil$]\label{ex:cyl}
Let $\zeta$ be the  holomorphic potential, defined on $\C$,
\begin{equation}\label{eq:zetapot}
 \zeta(z, \l) = 
\l^{-1}
\begin{pmatrix} 
0 & \overline{v(\bar z)} \\
 - v(z) & 0  
 \end{pmatrix} dz
+ 
\l \begin{pmatrix}
 0 & - \overline{v(\bar z)}\\ 
 v(z) & 0 \end{pmatrix} dz,
\end{equation}
 where 
\Red{
\[
 v(z) = \frac{1- i \sin z}{(-i + \sin z)^2}.
 \]
}
  Clearly, the scalar function $v$, and consequently  the one-form $\zeta(z, \l)$,  are invariant
 under the transformation $z \mapsto z + p$, where $p $ is any integer multiple of 
 $2 \pi$. For our goal of constructing a minimal cylinder in $\Nil$ we consider this potential to have the period $p = 2 \pi$.

 Let us consider the solution $d C = C \zeta$ with $C(0, \l) = \id$.
 Then $C_0(z) =  C (z, \l =1)$ is given by 
\[
 C_0(z) = \id. 
\]
Note, that $C_0(z + p ) = C_0(z)$ holds for all $ z \in \C$.

 Since $\zeta$ takes values in $\Lisu$ for $z \in \R$, it is easy  to verify that for real $z$ the matrix function 
  $C (z, \lambda)$ is, up to a diagonal gauge,  an
 extended frame of  some minimal surface $f$ in $\Nil$.
 Moreover, one can verify that the matrix function $C(z, \lambda)$ defined above satisfies
 \begin{equation}
 C (z + p, \l) = M(\l) C(z, \l)
 \end{equation}
 with $M(\lambda) \in \LISU$ for $\lambda \in \mathbb S^1$  and 
\begin{equation}
 M(\lambda = 1) = C_0(p)= \id.
\end{equation}
 Now a straightforward computation  shows  $X^o|_{\l=1} = 0$ and 
 $Y^d|_{\l=1} = 0$, respectively.
 This proves that the minimal surface in $\Nil$ constructed by the potential stated 
 above yields, for $\lambda = 1$, a minimal cylinder in $\Nil$.
 This fact is illustrated by the following pictures:
\begin{figure}[htbp]
  \begin{center}
    \begin{tabular}{c}

      \begin{minipage}{0.44\hsize}
        \begin{center}
 \includegraphics[width=1.2\textwidth]{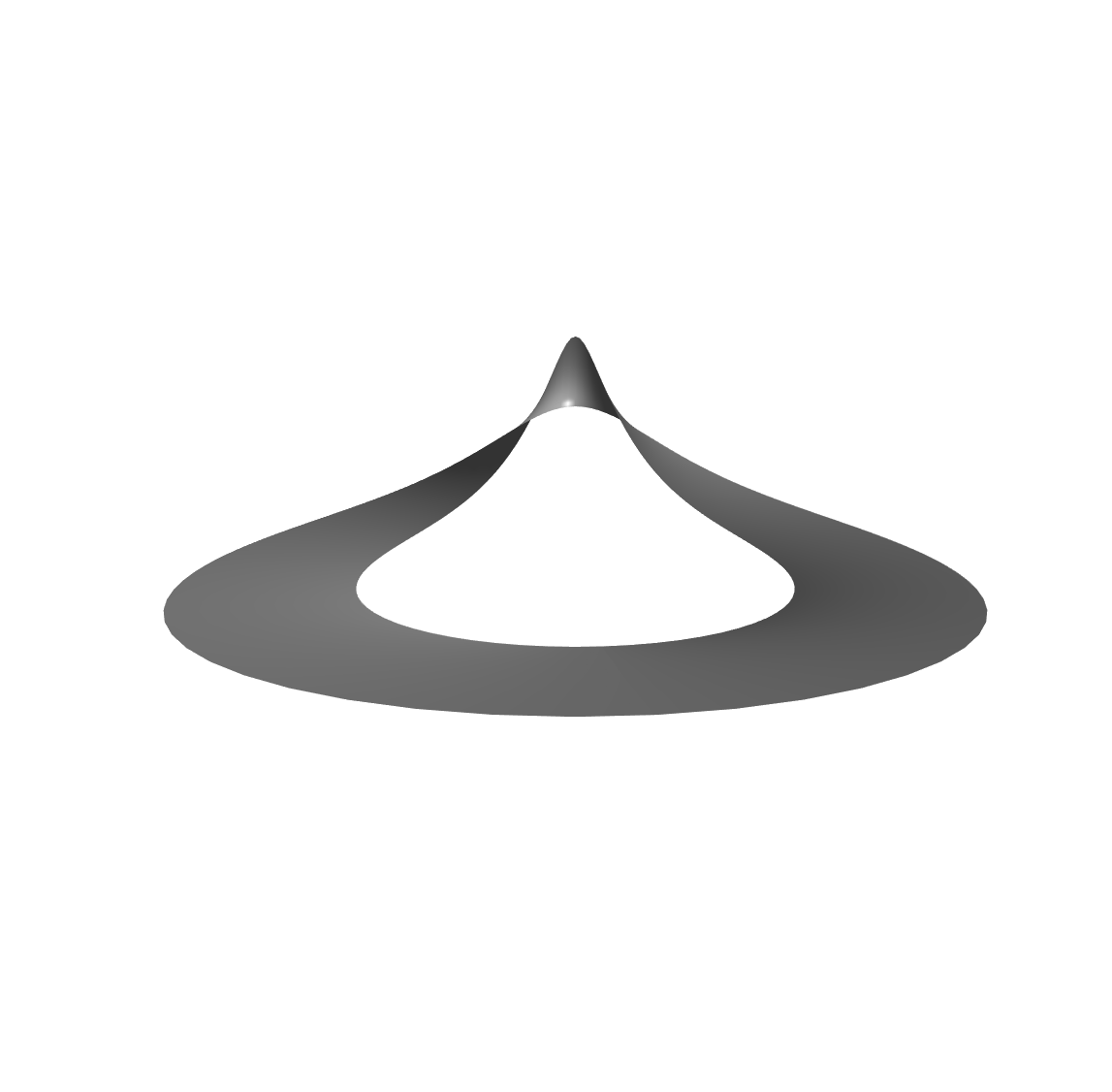}

        \end{center}
      \end{minipage}

      \begin{minipage}{0.56\hsize}
        \begin{center}
 \includegraphics[width=0.7\textwidth]{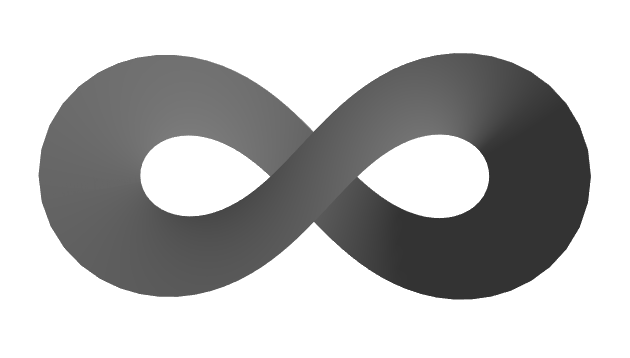}
        \end{center}
      \end{minipage}

    \end{tabular}
    \caption{Two views of the same minimal cylinder in $\Nil$ from 
  the hermitian potential $\zeta$ given in \eqref{eq:zetapot}.
  The right-hand side picture is a rotation view of the left-hand side picture. 
 The figures are 
  made from MATLAB program of the loop group construction outlined 
  in Appendix \ref{app:Loop} programmed by David Brander 
 (Technical University of Denmark).}
    \label{fig:cylinder}
  \end{center}
\end{figure}

 Finally we point out that the minimal cylinder just constructed is 
 not equivariant,
 since the Abresch-Rosenberg differential of the surface is 
 $4 (\cos^2 z  + \sin^2 3 z)dz^2$ which has zeros on $\C$ while
 it is constant on $\C$ for the equivariant case. 
\end{Example}

\section{Homogeneous minimal surfaces in $\Nil$}\label{section3} 
 The homogeneous minimal surfaces in $\Nil$ were classified in 
  Appendix B of \cite{DIKAsian}. For the sake of completeness we 
 recall this result.
 \subsection{Classification of homogeneous minimal surfaces} 
 A surface $f : M \to \Nil$ is called \textit{homogeneous} 
 if  there exists an injectively immersed  Lie group  $G \subset \isoo$
 which acts transitively on $f(M)$. 
 
 Since $\isoo$ acts transitively on all of $\Nil$, clearly
 $G \neq  \isoo$. If $\dim G =3$, 
 then, for every point in $f(M)$, there exists a $1$-dimensional 
 isotropy group. After left translation by same element in $\Nil 
 \subset \isoo$, we can assume that $f(M)$
 contains some element $c$ of the center of $\Nil$ 
 and we take this element as our base point. Since $\Nil$
 is normal in $\isoo$ one can write every 
 $h \in \iso$ in the form $h = p \phi$ where 
 $p \in \Nil$ and $\phi \in \mathrm{U}_1$ as we have used in 
 the proof of Theorem \ref{thm:symmetry}. 
 We obtain $c = h (c) = p c $, whence $p =\mathrm{id}$.
 This shows that the isotropy group is $\Uone$ and we can assume without
 loss of generality
 that $G$ contains a $2$-dimensional subgroup $G_0 \subset \Nil$ 
 which already acts transitively.
 A simple argument with Lie algebras shows that there is, 
 up to conjugacy, 
 exactly one $2$-dimensional subgroup permitting conjugacy 
 by elements of $\isoo$.
 
 Finally, assume that we have some $2$-dimensional 
 subgroup $G \subset \isoo$ which acts transitively 
 on  some minimal surface $f(\Rim)$ in $\Nil$. We can assume again that 
 $f(\Rim)$ in $\Nil$ contains an element $c \in \cNil$ and that $G $ 
 is not contained in $\Nil$.
\begin{Proposition}
 Homogeneous surfaces in $\Nil$ are congruent to one of the following surfaces$:$
\begin{enumerate}
\item 
An orbit of a normal subgroup 
\[
G(t)=\{(x_1,tx_1,x_3)\in \Nil\ \vert \;x_1,x_3\in \mathbb{R}\} 
\subset \Nil,
 \]
or
\[
G(\infty)=\{(0,x_2,x_3)\in \Nil\ \vert \; x_2,x_3\in \mathbb{R}\}
 \subset \Nil. 
\]
\item
An orbit of the Lie subgroup 
\[
\{((0, 0, s), e^{it}) \ \vert \  s, t \in \mathbb{R}\} \subset \Nil\rtimes \mathrm{U}_1.
\]
\end{enumerate}
In the former case, surfaces are vertical planes. Surfaces in the latter case are 
Hopf cylinders over circles. Thus the only homogeneous minimal surfaces in $\Nil$ are vertical planes. 
In particular the quadratic differential $B$
 vanishes identically on homogeneous surfaces.
\end{Proposition} 

\begin{Remark}
\mbox{}
\begin{enumerate}
 \item Note that  part $(1)$ follows from \cite{DIKAsian} and 
 part $(2)$ follows from  Theorem \ref{thm:oneparameter} below.

\item The homogeneous minimal surfaces in $\Nil$ are exactly 
 those minimal surfaces
 in $\Nil$ for which the function $w$ in \eqref{def-exp(w/2)}
 cannot be defined, that is, they are exactly those minimal 
 surfaces in $\Nil$ for which 
 the loop group approach does not work, that is, the case 
 of $B \equiv 0$.
\end{enumerate}
\end{Remark}

\section{Equivariant minimal surfaces in $\Nil$}
 In this section we will discuss minimal surfaces in $\Nil$ 
 which possess a one-parameter group of symmetries. 
 We begin by stating the following basic definition.
 
\begin{Definition}\label{df:equivariant}
 Let $f: \Rim \rightarrow \Nil$ be a surface. Then $f$ 
 is called \textit{equivariant}, if there exists a 
 pair of one-parameter groups 
 $(\gamma_t, \rho_t) \in  \Aut (\Rim) \times \isoo$ such that
\begin{equation}\label{eq:equivariance}
f \circ \gamma_t = \rho_t \circ f
\end{equation}
 holds.
\end{Definition}
 In Theorem \ref{thm:X1}, we will show that if 
 a minimal surface $S \subset \Nil$ is invariant 
 under a one-parameter group $\rho_t \in \isoo$, $\rho_t. S = S$, 
 there exists a special Riemann surface $\St$, an immersion $f: \St \to \Nil$
 with $f(\St) =S$ and a one-parameter group 
 $\gamma_t \in \Aut (\St)$ such that $f$ is
 equivariant in the sense of \eqref{eq:equivariance} 
 with respect to $(\gamma_t, \rho_t)$.

\subsection{One-parameter groups of $\isoo$}
To carry out our study of equivariant minimal surfaces
we will need a more detailed description
of the isometry group $\isoo$.
By  definition, each element of the isometry group $\isoo = \Nil \rtimes \Uone$ is of the form $((a_1, a_2, a_3), e^{i\theta})$.
Recall the group multiplication
\begin{equation*}
 (a_1,a_2,a_3) \cdot (x_1,x_2,x_3) =  \left(a_1 + x_1, a_2 + x_2, a_3 + x_3 + 
 \frac{1}{2} (a_1 x_2 - a_2 x_1)\right)
 \end{equation*}
of $\Nil$ and the action of $\isoo$ on $\Nil$:
\[
((a_1,a_2,a_3),e^{i\theta})\cdot (x_1,x_2,x_3)=
(a_1,a_2,a_3)\cdot (\cos \theta{x}_1-\sin\theta{x}_2,\sin\theta{x}_1+\cos\theta{x}_2,x_3).
 \]
Note, the isometry $((0,0,0),e^{i\theta})$ acts on $\Nil$ 
as a homomorphism of groups. It will be convenient to introduce a 
 ``shorthand writing''}
 for certain typical group elements. We will use
\[
 \alpha \equiv ((a_1, a_2,0),1), \quad \mathfrak{c} \equiv ((0,0, \mathfrak{c}),1), \quad e^{i\theta} \equiv ((0,0,0),e^{i\theta}).
\]
Then everything is expressed in terms of $\alpha = (a_1,a_2) = a_1 + i a_2, \mathfrak{c}$ and $e^{i\theta}$. In particular we have:
 Each element $\rho$ of $\isoo$ can be written uniquely in the form
\[
\rho = \alpha \mathfrak{c} e^{i\theta}. 
\]
Here is the list of the multiplications of the basic generators
with respect to the semi-direct product group structure introduced above:
\begin{enumerate}
\item The group of all $\cc$  is a one-dimensional group isomorphic to $\R$.
\item The group of all ${e^{i\theta}} $ is a one-dimensional group isomorphic to $\mathbb S^1$.
\item The centralizer of $\isoo$ consist exactly of all $\mathfrak{c}$.
\item For $\alpha, \beta \in \C \cong \R^2$, 
 $\alpha \beta = (\alpha+\beta) \mathfrak{c} (\alpha,\beta)$
 holds, where $\mathfrak{c}(\alpha, \beta) = \frac{1}{2}
\Im(\bar{\alpha} \cdot \beta)$ and 
 ``$\;\cdot \;$'' denotes the multiplication of the 
 complex numbers $\bar{\alpha}$ and $\beta$.
\item For $\beta \in \C \cong \R^2$, 
 $e^{i\theta} \beta=( e^{i \theta}\cdot \beta ) e^{i\theta}$,
 where ``$\;\cdot \;$'' again denotes the multiplication 
 of the complex numbers $\beta$ and $e^{i\theta}$.
\end{enumerate}
Putting this all together, one can easily verify
\[
(\alpha \mathfrak{c} e^{i\theta}) (\beta \mathfrak{d} e^{i\tau}) 
 = (\alpha + e^{i\theta}\cdot\beta) 
\left(\mathfrak{c} + \mathfrak{d} 
 + \frac{1}{2} \Im (\bar{\alpha} \cdot e^{i\theta} \cdot \beta) \right)
e^{i(\theta + \tau)}.
\]
Note that the identity element in $\isoo$ is $1 = ((0,0,0),1)$ 
and
\begin{equation}
(\alpha \mathfrak{c} e^{i\theta})^{-1} =e^{-i\theta} (-\mathfrak{c}) (-
 \alpha ) = (- e^{-i\theta} \cdot \alpha)
(-\mathfrak{c}) e^{-i\theta}.
\end{equation}
 Finally for $a = \alpha \cc \in \Nil$, we have
 $ e^{i \theta} a = (e^{i \theta} \cdot \alpha) \cc e^{i \theta}$ and denotes it by  
\begin{equation}\label{eq:rotationlaw}
 e^{i \theta} a = e^{i \theta}[a] e^{i \theta},
\end{equation}
 that is, $e^{i \theta}[a] = (e^{i \theta} \cdot \alpha) \cc$.
 In particular $e^{i \theta}[\cc] = \cc$ follows.
Finally we mention  that the one-parameter  group 
 $\rho_\theta\in \isoo$ generated by the Killing 
 vector field 
 $E_4=-x_2\frac{\partial}{\partial x_1}+
 x_1\frac{\partial}{\partial x_2}$ consists of 
 \textit{rotations} $\rho_\theta=((0,0,0),e^{i\theta})$ 
 of angle $\theta$ about the $x_3$-axis. 
 In our shorthand writing this is $\rho_\theta = e^{i\theta}$.

 An isometry $\rho^{(\cc)}_t\in \mathrm{Nil}_{3}
 \rtimes \mathrm{U}_1$ of the form  
\begin{equation*}
 \rho_{t}^{(\cc)} = (\cc t) e^{it} = ((0,0,\cc t),e^{it}),
\end{equation*} 
 where $\cc \in \cNil, t \in \R$, is called a  
 \textit{helicoidal motion with pitch $\cc$}. 
 By what was said above it is clear that this motion moves the points 
 in $\Nil$ along the $e_3$-axis $\R e_3$ and rotates them about this 
 axis simultaneously.
 The family of all transformations $\rho^{(\cc)}_t$ forms 
 for fixed $\cc$ a one-parameter group.
 In general, a \textit{helicoidal motion}  
 along the axis $\alpha+ \R e_3$ through the point 
 $ \alpha = (a_1 + i a_2, 0)
 =(a_1, a_2, 0) \in \R^2 \subset \Nil$ 
 and with pitch $\cc$
 has the form:
 \begin{equation} 
 \rho_t^{(\cc,\alpha)} =\alpha \{ ( t \cc) e^{it} \} \alpha^{-1} 
 = ( t \cc)\{ \alpha  e^{it} \} \alpha^{-1} \in \isoo.
 \end{equation}
  Clearly, the transformations $ \rho_t^{(\cc,\alpha)} \; (t \in \R)$ 
 form a one-parameter group. Moreover, a simple computation yields 
 the natural and unique representation:
 \begin{equation}\label{eq:helicoidalmotion}
  \rho_t^{(\cc,\alpha)}  =   (\alpha \cdot (1 - e^{it})) \left( \cc t -\frac{|\alpha|^2}{2} \sin t\right)  e^{it}.
\end{equation}
 A \textit{translation motion} $\rho_t \in \Nil$ in direction 
 $(a_1, a_2, \cc) \in \Nil$ is given by 
 \begin{equation}\label{eq:translationmotion}
\rho_t =  (t\alpha) ( t\mathfrak{c}) \in \isoo.
\end{equation}

 
%
 In general one can consider \textit{any} one-parameter group, not only
 a translation motion nor only a helicoidal motion along the axes 
 $\alpha + \R e_3, \alpha = a^h$.
 However, the following Theorem \ref{thm:oneparameter} 
 implies that actually any one-parameter group which is not given by 
 translations can be interpreted as a helicoidal motion, 
 (for example \cite[Theorem 2]{FMP}).
\begin{Lemma}\label{lm:rhophi}
 Let $\rho = p \phi \in \isoo$ with 
 $p = \pi_0 \mathfrak{p}_c$, where $\pi_0 \in \R^2$,
 $\mathfrak{p}_c \in \cNil$ and $\phi = e^{i q } \in \Uone$ 
 for some $q \notin 2 \pi \mathbb{Z}$. Then $\rho$ can be represented 
 uniquely in the form
\[
 \rho = \mathfrak{c} \alpha \phi \alpha^{-1}
\] 
 for some $ \alpha  \in \R^2 \subset \Nil$ 
 and $\mathfrak{c} \in \cNil$.
 \end{Lemma}
 
\begin{proof}
 We compute the coefficients of any expression of the form
\[
 \cc \alpha \phi \alpha^{-1}
\]
 with $\cc \in \cNil$, $ \alpha = a^h = a_1 + ia_2$ and $\phi = e^{iq}$.
 Since $\phi$ satisfies \eqref{eq:rotationlaw}, 
 $\phi \alpha = \phi[\alpha] \phi $ and we derive
\[
 \alpha \phi \alpha^{-1} =  \alpha \left(\phi \alpha^{-1} \phi^{-1}\right) \phi = \alpha 
 \phi [\alpha^{-1}] \phi.
\]
 Now a straightforward computation shows 
 that $ w = (\cc \alpha \phi \alpha^{-1}) \phi$ has the coefficients
\begin{align*} 
w_1 &= a_1 - a_1 \cos q + a_2 \sin q, \quad w_2 = a_2 - a_1 \sin q - a_2 \cos q,\quad 
w_3 &= \cc - \frac{1}{2} (a_1^2 + a_2^2)\sin q,
\end{align*}
 where we set $w = (w_1, w_2, w_3) \in \Nil$.
 Using $ q \notin 2 \pi \mathbb{Z}$ it is easy to prove that 
 $(a_1,a_2, \cc) \rightarrow  (w_1,w_2,w_3)$ is a diffeomorphism 
 from $\R^3$ to $\R^3$.
 Therefore the $p$ defined by $\rho$ can be derived from 
 some $(a_1,a_2,\cc)$ and the claim follows.
 \end{proof}

\begin{Theorem}\label{thm:oneparameter}
 Assume $\rho_t$ is a one-parameter group in $\isoo$ 
 which is not contained entirely in $\Nil$. 
 Then with the notation of Lemma {\rm \ref{lm:rhophi}}, $\rho_t$ can be represented in the form
\[
 \rho_t = \cc_t \alpha \phi_t \alpha^{-1},
\] 
 where $\cc_t  =  t \cc \in \cNil, \alpha = a^h \in \Nil$ is 
 independent of $t$, and $\phi_t = e^{itq}$ with $q \neq 0$.
\end{Theorem}

\begin{proof}
 Let $\rho_t$ denote the given one-parameter group. We can write
 $\rho_t = \pi_t \mathfrak{p}_{t} \phi_t$. 
 Assuming without loss of generality $q(0)=0$ this decomposition 
 is unique.
 By definition $\rho_{t +s} = \pi_{t+s} \mathfrak{p}_{t+s} \phi_{t + s}.$  Moreover, 
\[
 \rho_t \rho_s =  \pi_t \mathfrak{p}_{t} \phi_t\pi_s \mathfrak{p}_{s} \phi_s
  = \gamma_{t, s} \mathfrak{h}_{t, s} \phi_t \phi_s. 
\]
 The equality $\rho_{t +s} = \rho_{t} \rho_{ s}$ now
  implies that 
 $\phi_t$ is a one-parameter group. 
 Hence $\phi_t = e^{itq}$ where $q \neq 0$, otherwise $\rho_t$ would be 
 contained entirely in $\Nil$.
 Now we write  $\rho_t = \cc_t \alpha_t \phi_t \alpha_t^{-1}$ as in Lemma \ref{lm:rhophi}. 
 Then
 \[
  \rho_s \rho_r = (\cc_s \alpha_s \phi_s \alpha_s^{-1}) 
 (\cc_r \alpha_r \phi_r \alpha_r^{-1})  = \rho_{r+s}.
 \]
 Using formula $\phi a = \phi[a] \phi$ by \eqref{eq:rotationlaw}, 
 we rephrase the middle term above as
\begin{align*} 
 (\cc_s\alpha_s \phi_s \alpha_s^{-1}) (\cc_r \alpha_r \phi_r \alpha_r^{-1}) &=  
 (\cc_s \alpha_s\phi_s [\alpha_s^{-1}]) (\cc_r \phi_s
[\alpha_r\phi_r [\alpha_r^{-1}]] \phi_{s+r})  \\
&= (\cc_s \cc_r \alpha_s\phi_s [\alpha_s^{-1}]) (\phi_s
 [\alpha_r\phi_r [\alpha_r^{-1}]]  \phi_{r+s}),
 \end{align*}
 where we have also used that $\phi_s [\cc_r] = \cc_r$ holds, 
 since $\phi [\cc] = \cc$ for all $\cc \in \cNil$.
 Comparing this to $\rho_{r+s}$ we observe
 \begin{equation}\label{eq:firstcomponentofa}
 (\cc_r \cc_s  \alpha_s \phi_s [\alpha_s^{-1}])(
 \phi_s[\alpha_r \phi_r [\alpha_r^{-1}]])
 = \cc_{r+s} \alpha_{r+s}\phi_{r+s}[\alpha_{r+s}^{-1}].
\end{equation}
 Recall that $\alpha_t$ has no component in $\cNil$, that is, 
 $\alpha_t = a_t^h$, whence
 $\phi_t[\alpha_t] = e^{iqt} \cdot \alpha_t$. But then $\alpha_{r+s} \phi_{r+s}[\alpha^{-1}_{r+s}] =
 \alpha_{r+s} - e^{iq(r+s)} \cdot \alpha_{r+s}$ modulo $\cNil$ and 
 $(\alpha_s \phi_s[\alpha_s^{-1}])(\phi_s[\alpha_r \phi_r [\alpha_r^{-1}]]) =
 \alpha_s - e^{isq} \cdot \alpha_s + e^{isq} \cdot (\alpha_r - e^{irq} \alpha_r)$ 
 modulo $\cNil$ follows.
 As a consequence we obtain the following equation of complex numbers
 \begin{equation}\label{eq:alpharelation}
 (1- e^{iqs}) \cdot \alpha_s + e^{iqs} \cdot (1- e^{iqr}) 
 \alpha_r = (1 - e^{iq(r+s)}) \cdot 
 \alpha_{s+r}.
\end{equation}
 Differentiating \eqref{eq:alpharelation} for $s$ at $s=0$ we obtain
 $-iq  \cdot \alpha_0 + iq (1-e^{iqr})  \cdot \alpha_r = 
 -iq \cdot e^{iqr} \cdot \alpha_r + (1- e^{iqr}) \cdot \frac{d}{dr} \cdot \alpha_r$.
 This equation simplifies to yield
\begin{equation}\label{eq:alphas}
 iq\cdot ( \alpha_r -  \alpha_0) = (1 - e^{iqr}) \cdot \frac{d}{dr} \alpha_r.
\end{equation}
 Differentiating \eqref{eq:alpharelation} for $r$ at $r=0$, we obtain
 $e^{iqs}(-iq) \alpha_0 = -iqe^{iqs} \cdot \alpha_s + (1 - e^{iqs})  \cdot 
 \frac{d}{ds}\alpha_s$, 
 which simplifies to
\begin{equation}\label{eq:alphar}
 iq e^{iqs}\cdot (  \alpha_s -   \alpha_0) = (1 - e^{iqs}) \cdot \frac{d}{ds} \alpha_s.
\end{equation}
 From \eqref{eq:alphas} and \eqref{eq:alphar}, we obtain that $\alpha_t$ 
 is constant (say equal to $\alpha $). Since now $\alpha = \alpha_r = \alpha_s = \alpha_{r+s}$
 and since also  (\ref{eq:firstcomponentofa}) holds, we obtain
 \begin{equation}
 (\cc_r \cc_s  \alpha \phi_s [\alpha^{-1}])
 (\phi_s[\alpha\phi_r [\alpha^{-1}]])
 = \cc_{r+s} \alpha\phi_{r+s}[ \alpha^{-1}].
 \end{equation}
 Since $\phi_*$ is a homomorphism of $\Nil$, we obtain 
 \[
  (\phi_s[\alpha^{-1}]) (\phi_s[ \alpha\phi_r [\alpha^{-1}]]) =  
 (\phi_s[\alpha^{-1}]) (\phi_s [\alpha]) (\phi_s 
 [\phi_r [\alpha^{-1}]]).
 \]
 Therefore the factors on the right cancel. 
 This implies $\cc_r \cc_s = \cc_{r+s}$ and the claim follows.
\end{proof}
\begin{Remark}
 The theorem above was stated (without proof) 
 in Theorem 2 in \cite{FMP}.
\end{Remark}
 In view of  Theorem \ref{thm:oneparameter} above  we introduce the 
 following definition.
\begin{Definition}\label{def:helicoidal}
 Let $f:\Rim \to \mathrm{Nil}_3$ be a conformal immersion 
 from a Riemann surface $\Rim$ into $\Nil$.
\begin{enumerate}
 \item  $f$ is said to be a 
 \textit{helicoidal surface} if the image $f(\Rim)$ is 
 invariant under a one-parameter group of helicoidal motions  
 $\{\rho_t^{(\cc,\alpha)} \}_{t \in \R}$ 
 as defined in \eqref{eq:helicoidalmotion}, that is, 
\[
 f(\Rim) = \rho_t^{(\cc,\alpha)} . f(\Rim)
\]
 holds for all $t \in \R$. 
 In particular, $f$ is said to be a \textit{rotational surface} if the 
 helicoidal motion does not have a pitch.
\item  
 $f$ is said to be a 
 \textit{translation invariant surface} if the image $f(\Rim)$ is invariant under a one-parameter group of translation motions 
$\{\rho_t\}_{t \in \R}$  defined as in \eqref{eq:translationmotion}, 
 that is, 
\[
 f(\Rim) = \rho_t . f(\Rim)
\]
 holds for all $t \in \R$.
\end{enumerate}
\end{Definition}
 As a corollary of Theorem \ref{thm:oneparameter}, we have the 
 following.
\begin{Corollary}\label{cor:eqminsurf}
The family of equivariant minimal surfaces in the sense of 
 Definition {\rm \ref{df:equivariant}} consists of all minimal
helicoidal surfaces and all minimal translational surfaces.
\end{Corollary}
\begin{Example}
 The standard helicoid
\[
f(x_1,x_2)=(x_1,x_2,\cc \tan^{-1}(x_2/x_1))
\]
 is a helicoidal minimal surface in $\mathrm{Nil}_3$. 
 In fact this surface is invariant under the helicoidal motion of pitch $\cc$.
 \end{Example}
 
 \begin{Remark}
 {\rm Caddeo, Piu and Ratto \cite{CPR} 
 studied rotational surfaces of constant mean curvature 
 (including minimal surfaces) in $\mathrm{Nil}_3$ 
 via ``equivariant submanifold geometry''
 in the sense of W.~Y.~Hsiang \cite{Hsiang}.
 Moreover, Figueroa, Mercuri and Pedrosa \cite{FMP} 
 investigated surfaces of constant mean curvature invariant under
 some one-parameter isometry group.
 For minimal surfaces the  results of this paper recover their results. }
 The moduli space of all equivariant minimal surfaces in 
 $\mathrm{Nil}_3$ will be given in the forthcoming paper \cite{K:Explicit}.
\end{Remark}
\subsection{Equivariance induced by one-parameter groups of 
 $\isoo$}
 We now show that a one-parameter group of symmetries of 
 a conformal minimal immersion $f$ from a Riemann surface $\Rim$ in $\Nil$ induces  a minimal horizontal plane or a
 one-parameter group of symmetries for a conformal minimal
 immersion $\tilde{f}$ of a strip $\St$. More precisely we have the following theorem.
 
\begin{Theorem}\label{thm:X1}
 Let $f$ be a conformal minimal immersion from a Riemann surface $\Rim$ into 
 $\Nil$ and $\rho_t$ a one-parameter group in $\isoo$ acting as a group of symmetries 
 of $f$, that is, $\rho_t. f(\Rim) = f(\Rim)$ holds.
\begin{enumerate}
\item Assume that the one-parameter group $\rho_t$ acts with fixed points. Then  $f(\Rim)$ is a horizontal plane. 
\item Assume that the one-parameter group $\rho_t$ acts without fixed points. Then there exists an open strip $\mathbb S 
 \subset \mathbb C$
 containing the real axis  and an immersion $\tilde f : \mathbb S \to \Nil$
 such that $f(\Rim) = \tilde f(\mathbb S)$ and
 \[
  \rho_t .\tilde f (z)= \tilde f(\gamma_t. z),
 \]
 for all $z \in \mathbb S$, holds.
\end{enumerate}
\end{Theorem}

\begin{proof}
(1): Since $\rho_t$ is classified as in Definition
 \ref{def:helicoidal} and has fixed points by assumption, 
 it must be a rotation around  the
 axis through a point $(a, b, 0) \in \Nil$
 parallel to the $e_3$-axis. Then we can choose a simply-connected 
 domain $\tilde \D \subset \C$ which contains $z =0$ and a minimal 
 immersion  $\tilde f : \tilde \D \to \Nil$ 
 such that $\tilde f(\tilde \D) \subset f(\Rim)$ and
 $\tilde f(0)$ is one of fixed points of $\rho_t$. Moreover there exists
 a $\gamma_t : z \mapsto  z e^{it}$ as a local one-parameter group 
 of $\tilde \D$ such that $0 \in \tilde \D$ is a fixed point of 
 $\gamma_t$ and $\tilde f$ is equivariant with 
 respect to $(\gamma_t, \rho_t)$. 
 Then for the harmonic normal Gauss map $g: \D \rightarrow \Ha^2$ 
 and an extended frame $F$ of $f$ we obtain
\begin{equation}
 g(\gamma_t.z, \overline{\gamma_t.z}) = e^{i a t} g(z, \bar z)
\end{equation}
 for some $a \in \R$ and the extended frame $F$ satisfies
\begin{equation}
 F(\gamma_t.z, \overline{\gamma_t.z}, \lambda) 
 = M_t(\lambda) F(z, \bar {z}, \lambda) k(t,z, \bar z),
\end{equation}
 where $M_t(\lambda) \in \LISU$ and $M_t(\lambda = 1) = 
 \di (e^{a it/2}, e^{a it/2}) $ and 
 $k(t,z, \bar z) \in \Uone$, see Proposition \ref{prp:transformation}.
  For $z=0$ we infer
\begin{equation}
 F(0, \lambda) = M_t(\lambda) F (0,\lambda) k(t,0),
\end{equation}
 Replacing $F$ by $\hat F(z,\bar z, \lambda) = F(0, \lambda)^{-1} 
 F (z,\bar z, \lambda)$,
 we obtain $\hat F(0,\lambda) = \id $ and, setting 
 $\hat M_t(\lambda) =  F(0, \lambda)^{-1} M_t(\lambda)  F(0, \lambda)$  
 we derive
\begin{equation}
\hat F(\gamma_t . z, \overline{\gamma_t . z}, \lambda) 
 = \hat M_t(\lambda) \hat F(z,\bar z, \lambda) k(t,z, \bar z).
\end{equation}
 As a consequence we obtain
\begin{equation}
 \hat M_t(\lambda) = k(t,0)^{-1} = k_0(t).
\end{equation}
 In particular, $k_0(t) = \hat M_t(\lambda)$ is independent of 
 $\lambda$ and contained in $\Uone$.
 Hence $\hat M_t(\lambda)$ is diagonal.
 As a consequence we have two cases: 
 
 {\bf Case 1.} $\hat M_t(\lambda) =  \id$ for all $t \in \R$. In this case also $M_t(\lambda) =  \id $ 
 for all $t \in \R$. But then
 $f(e^{it}z) =  f(z)$ for all $t\in \R$ and $f$ is not a surface.
 
 {\bf Case 2.}  $\hat M_t(\lambda) = k_0(t) = \di (e^{iat/2}, e^{i a t/2})\neq  \id$, 
 that is, $a \in \R \setminus 2 \pi \Z$.
 Since $\hat F(0,\lambda) =\id$ we can perform the Birkhoff decomposition
 $\hat{F}(z, \bar z, \lambda) = \hat{F}_- (z,\lambda) \hat{L}_+(z, \bar z, \lambda)$
 around $z =0$
 and obtain
 \begin{equation} \label{trafoFminus}
 \hat{F}_-( \gamma_t. z, \lambda) = k_0(t) \hat{F}_- (z,\lambda) k_0(t)^{-1}.
 \end{equation}
Note that $\hat{F}(0,0,\lambda) = \id$ implies that $\hat{F}_-$ is holomorphic with respect to $z$ in an open neighbourhood 
 of $z = 0$.
 Let $\eta_-(z,\lambda) = \hat{F}_-^{-1} d\hat{F}_-$, then 
 $\eta_-(z,\lambda) = \lambda^{-1} \xi(z) dz$ is
 the normalized potential associated with the minimal surface $f$, 
 the normal Gauss map $g$, and the frame $\hat{F}$.
 Then we obtain from \eqref{trafoFminus}:
\begin{equation} \label{trafoxi}
\xi( e^{it} z) e^{it} = k_0(t) \xi(z) k_0(t)^{-1}.
\end{equation}
 Writing 
\begin{equation}
\xi = 
\lambda^{-1}
\begin{pmatrix} 
 0 & -p \\ Bp^{-1} & 0
\end{pmatrix},
\end{equation}
the equation \eqref{trafoxi} yields 
\begin{equation} \label{trafop}
p(e^{it} z) e^{it} = e^{ia t} p(z), 
\end{equation}
 and we have 
\begin{equation}
 B(e^{it} z) e^{2it} = B(z),
\end{equation}
 since $B(z)dz^2$ is a globally defined quadratic differential.
 Note that $a$ takes values in $\R \setminus 2 \pi \Z$.
 From the last equation it now follows that $B(z)$ is identically zero.
 
 From equation \eqref{trafop} we infer that $p$ is of the form 
 $p(z) = p_j z^j$ for some $j \in \Z$
 and $p_j \neq 0$. Moreover, $j+1 = a$ holds.
 
 Since we know that $\hat{F}_-$ is holomorphic at $z=0$ it follows that
  $p$ is holomorphic at $z=0$, whence $j \geq 0$ follows. 
 Now, if $j>0$, then the surface $f$ has a branch point at $z=0$, 
 a contradiction. As a consequence, $j =0$.
 This case has already been considered in 
 \cite[Section 6]{DIKAsian} and it was shown that 
 the corresponding minimal surfaces are horizontal planes.
 Then since the Abresch-Rosenberg differential $B dz^2$ vanishes on $\tilde f(\tilde \D) \subset f(\Rim)$, it vanishes
 on $f(\Rim)$ and the whole surface $f(\Rim)$ is the horizontal plane.

(2): Since $\rho_t$ acts without fixed points on $f(\Rim)$, around any  $p_0 \in f(\Rim)$   there exists a chart $\psi_0: \D_0 \rightarrow \Nil$ such that $\psi(0) = p_0$ and $\D_0$ is an open  rectangle containing the origin and with axes parallel to the usual coordinate axes of $\R^2$.
Moreover, for all $z \in \D_0$ and sufficiently small 
 $t \in I = (-\epsilon, \epsilon)$ we have with $f_0 = f \circ \psi_0$:
 \[
  f_0 (z + t)  = \rho_t.f_0 (z).
 \]
 This follows from the fact that the (never vanishing) vector field generating the one-parameter group action $\rho_t$ can be represented as $\frac{\partial}{\partial x}$ in some chart.
 
 Let $\St$ denote the strip parallel to the real axis and containing $\R$ which has the same height as $\D_0$. By \cite{Burstall-Kilian} there exists a Delaunay type matrix $D(\lambda)$ which generates a minimal immersion $f_o^\sharp$ on $\St$ which coincides with $f_0$ on $\D_0$,
 see also Theorem \ref{thm:equivariant}.
 
 We claim $ f_o^\sharp (\St) \subset f(\Rim)$. Suppose this is wrong, then there exists a line segment $L$ in $\St$, parallel to the $x$-axis,  such that  $f_o^\sharp $ leaves $f(\Rim)$ at some endpoint $l_0$ of $L$. Let us write $l_0 = (t_0, y_0)$ and let us assume without loss of 
 generality $t_0 > 0.$
 Let $s>0$ such that $(s, y_0) \in \D_0.$  Then $t_0 = ms + s^*$ with $0 < s^* < s$ and $m \in \mathbb{Z}$. As a consequence  $f_0^\sharp (0 + t_0)  = (\rho_s)^m.f_0 (s^*) \in f(\Rim)$.
 Now we can choose a small chart around $q_0 = f_0^\sharp(t_0)$ which corresponds to a small box in $\R^2$  centered at $(t_0,y_0)$ such that, analogous to the argument above, 
 $f_0^\sharp$ maps the small box into $f(\Rim)$. Hence $f_0^\sharp$ maps a strictly larger line segment $ L \subset L^\sharp$ into $f(\Rim)$. This contradiction implies $f_0^\sharp (\St) \subset f(\Rim)$. 
 
 Finally we want to show $f_0^\sharp (\St) = f(\Rim)$. 
 For this we choose $\St$ considered above as large as possible. Let us consider first the case, where  $\St$ ends in the upper half-plane at the line $T_u$ and in the lower half-plane at the line $T_l$, both parallel to the $x$-axis. If there exists any point in $f(\Rim)$ which is not contained in $f_0^\sharp (\St)$, then we choose a curve in  $f(\Rim)$ connecting such a point with $f_0^\sharp(0)$. This curve needs to intersect $f_0^\sharp(T_u)$ or 
 $f_0^\sharp(T_l).$ At a point of intersection we apply the argument above and obtain an open strip containing the corresponding boundary line of the image of $f_0^\sharp$.  Therefore 
  $f_0^\sharp$ can be extended beyond this boundary line, a contradiction.
  We thus only need to consider the case ,where the strip is either half-infinite or all of $\C$ and where in $f(\Rim)$ there is a boundary point $q_0$ of $f_0^\sharp (\St)$, which can be obtained by taking a limit to $\infty$ inside $\St$. Then by the argument above we obtain a finite open strip $\mathbb{B}$ containing $q_0$ and an equivariant conformal map 
  $f_0^{\prime}$ 
  such that on some sub-strip of  $\mathbb{B}$ and some half-plane the conformal maps  $f_0^\sharp$
  and $f_0^\prime$ have the same image. Since both maps are equivariant
 under real translations, they induce a bi-holomorphic change of coordinates of the type 
 $(x,y) \mapsto (x, h(y))$. Hence $h(y) = y + c$. 
 This is impossible, since one strip has infinite width 
 and the other one has only finite width.
 \end{proof}
 
 \begin{Remark}
 Above we have shown that the invariance of $f(\Rim)$ under a one-parameter group without fixed points can be realized by an immersion 
 of some open strip $\mathbb S$. However, in general it is not possible to define a one-parameter group on the original surface $\Rim$.
\end{Remark}
\subsection{One-parameter groups of $\Aut (\Rim)$} \label{list}
 It is well known that only a few Riemann surfaces admit a one-parameter group of 
 automorphisms. For non-compact simply-connected 
 Riemann surfaces only the following cases occur 
 (up to conjugation by bi-holomorphic automorphisms (see, for example
 \cite[Section V-4]{FK:Rim}):
 Let us denote the complex plane by $\C$ and
 the upper half plane by $\mathbb H$, that is $\mathbb H= \{ z \in \C \;|\; \Im z >0 \}$.
\begin{enumerate}
 \item[(1a)] $\C$ and all translations parallel to the $x$-axis,
 \item[(1b)] $\C$ and all multiplications $z \rightarrow e^{ta} z$ with $a \in \C^*$.\\[-0.3cm]
 \item[(2a)] $\mathbb{H}$ and all translations parallel to the real axis,
 \item[(2b)] $\mathbb{H}$ and all multiplications $z \rightarrow az$ with $a$ positive real,
 \item[(2c)] $\mathbb{H}$ and all automorphisms fixing the point $i$.
\end{enumerate}
 In the cases (1b) and (2c) 
 the Riemann surface contains a point which is fixed by the one-parameter group. 
 We will show in Theorem \ref{thm:Rotationwithfixed} below that these cases only consist of very special minimal surfaces.
 In case (1b), if one removes the origin and considers the map $\C  \rightarrow 
 \C^*=\C\smallsetminus \{0\}, w \rightarrow e^{aw}$, then the group action pulls back to translation parallel to the $x$-axis. A similar observation holds in case (2c), 
 if one interprets it as rotation about the origin of the unit disk.
 In case (2b), one can map $\mathbb{H}$  via $z \rightarrow \log (z) -i\pi /2 $ 
 to the strip parallel to the real axis between $y= \pi /2 $ 
 and $y= - \pi /2$ such that the 
 one-parameter group turns into the group of translations parallel to the real axis.

 In the following cases one can consider the universal cover and thus 
 obtains strips with the one-parameter group of translations parallel to the real axis.
\begin{enumerate}
 \item[(3a)] $\mathbb{H}^2_*=\mathbb{H}^2\smallsetminus\{0\}$ and all rotations about the origin,
\item[(3b)] $\mathbb{C}^*$ and all multiplications $z \rightarrow e^{ta}z$ with $a \in \C^*$,
\item[(3c)]$\mathbb{A}_{a,b}$ and all rotations about the origin, 
 where  $ 0 < a < b$ and  
 $\mathbb{A}_{a,b} = \{ z \in \C, 0 < a < |z| < b \}$.
\end{enumerate}
 Beyond the cases listed above, only tori admit one-parameter 
 groups of automorphisms. Note that above already all conformal 
 types of cylinders have been listed.
\begin{Definition}
  Equivariant surfaces for which the group acts by translations 
  (on a strip) will be called \textit{$\R$-equivariant}.
   Equivariant surfaces for which the group acts by rotations 
  (about a point) will be called \textit{$\mathbb S^1$-equivariant}.
\end{Definition}
 The cases (1b) and (3b) do not fall directly into these two categories.
 Note, all $\mathbb S^1$-equivariant cases have a natural fixed point 
 contained in the domain of definition, or not.
\begin{Theorem}\label{thm:Requivariant}
 Let  $f : \Rim \to \Nil$ be an equivariant minimal surface of the 
 type {\rm (3a)}, {\rm (3b)} or {\rm (3c)}.  Since the fixed point of $\gamma_t$ is not contained in $\Rim$,  one can 
 realize the universal cover $\St$ of $\Rim$ as a strip containing the $x$-axis,  
 such that  the  induced map $\tilde{f}: \St \rightarrow \Nil$ is 
 $\R$-equivariant relative to all real translations in the first two cases and in direction $a$ in the last case. 
 Moreover, in the cases {\rm (3a)} and {\rm (3c)} $\tilde{f}$ is periodic and 
 has a $($smallest$)$  positive real period, and 
 in the case {\rm (3b)} the period is $2 \pi /a$.
\end{Theorem}
\begin{proof}
We only need to prove the last assertion. Suppose there does not exist 
 a smallest positive period. Then there exists a sequence $p_n$ of positive 
 periods converging to $0$. Since $f$ is real analytic, $f$ is constant, a contradiction, since $f$ is assumed  to be
a surface. In the case {\rm (3b)} we consider the universal 
cover $\pi_a : \C \rightarrow \C^*, w \rightarrow e^{aw}$. Then the given action corresponds to 
$w \rightarrow w+t$. Hence the period is $2 \pi /a$.
\end{proof}
\subsection{Special equivariant minimal surfaces} \label{fixed}
Next we will show that $\mathbb S^1$-equivariant 
minimal surfaces with fixed point or vanishing Abresch-Rosenberg differential
are very special. 
\begin{Theorem}\label{thm:Rotationwithfixed}
\mbox{}
\begin{enumerate}
 \item 
 Consider an equivariant minimal surface in $\Nil$  
 with fixed point  in $\Nil$, that  is, it is in one of 
 the cases {\rm (1b)} or
 {\rm (2c)}.
 Then  the Abresch-Rosenberg differential vanishes identically and 
 such a minimal surface is only a horizontal plane.
\item  Consider an equivariant minimal surface in $\Nil$ without fixed point and  
 vanishing Abresch-Rosenberg differential. Then such a minimal surface is only a vertical plane.
\end{enumerate}
\end{Theorem}
\begin{proof}
 (1) The statement  follows directly from $(1)$ in Theorem \ref{thm:X1}. 

 (2) By Proposition 2 in \cite{DIKAsian}, such a minimal surface is only a horizontal 
plane or a vertical plane. The only vertical plane does not have any fixed point.
\end{proof}

\subsection{Basics about $\R$-equivariant minimal surfaces}
 By our discussion in Sections \ref{list} and \ref{fixed}, from here on   we only need to consider $\R$-equivariant surfaces
 which are defined on some strip $\St$ 
 and have non-vanishing Abresch-Rosenberg differentials. 
 Specific properties of the 
 different cases will be discussed elsewhere. 
 For simplicity of notation we will, as before, abbreviate a function $p(z, \bar z)$ by $p(z)$.
 Hence the expression $p(z)$ does not necessarily denote
 a function depending only on $z$.
\begin{Theorem}\label{thm:basics-equivariant}
 Let  $f:\St \to \Nil$ be an $\R$-equivariant minimal surface relative to the 
 one-parameter group $(\gamma_t, \rho_t)$,  $\gamma_t.z = z + t$, and $\rho_t$ 
 a one-parameter group in $\isoo$ which is not contained in $\Nil$.
 Let $g$ denote the $($non-holomorphic$)$  normal Gauss map of $f$. Then we obtain
 \begin{equation}\label{eq:fandgtrans}
   f(z+t) = \rho_t. f(z) \quad \mbox{and} \quad  g(z+t) = e^{iat} g(z)
 \end{equation}
 with $ 0 \neq a \in \R$. 

 Moreover,
\begin{enumerate}
 \item For an extended frame $F$ of $f$ as given in \eqref{special frame},  
 there exists some $k(t,z) \in \Uone$ satisfying
 \begin{equation}\label{eq:Fequiv}
   F(z+t, \lambda) =  M_t (\lambda) F(z, \lambda) k(t,z),
 \end{equation}
where $M_t \in \LISU$,  $M_{t}(\l = 1)  = \di ( e^{iat /2},e^{-ia t/2} ). $
 \item There exists a unitary diagonal matrix $\ell$ such that 
 the frame $F_{\ell} = F \ell$ satisfies $k_{\ell}(t,z) \equiv \id$.
\end{enumerate} 

 \end{Theorem}
\begin{proof}
 The transformation behaviour  (\ref{eq:Fequiv})  of $F$ 
 follows, since $F: \St \rightarrow \LISU$ is a lift of $g : \St \rightarrow 
 \mathbb H^2$, see also \eqref{eq:transext}.
 Also note, since $M_{t}|_{\l = 1}$ is a homomorphism, 
 it is easy to verify that $k(t,z)$ satisfies the cocycle condition
\begin{equation} \label{cocyc}
  k(t+ s, z) = 
 k(t, z)  k(s, z+t),
 \end{equation}
 and we obtain (see for example Theorem \ref{invframe}
 and \cite[Theorem 4.1]{DoMa}):
\begin{equation*}
 k(t,z) = \ell (z) \ell (z+t)^{-1},
\end{equation*}
where $\ell (z) = \ell (x +iy)  = k(x, iy)^{-1}$.
\footnote{
The following is a brief proof:
 \begin{align*}
 \ell (z) \ell (z + t)^{-1} &= k(x, iy)^{-1} k(x+t, iy) \\
 &=k(x, iy)^{-1} \left\{k(x, iy)  k(t,iy +x)\right\} = k(t ,z).
 \end{align*}
 Where we have used equation (\ref{cocyc}) with $z$ replaced by 
 $iy$, $t$ by $x$ and $s$ by $t$.
} 
 As a consequence, replacing the original frame $F$ by $F\ell$ 
 one obtains an extended frame as desired. 
\end{proof}

\begin{Remark}
\mbox{}
\begin{enumerate}
 \item 
 In the theorem above one could also permit one-parameter families $\rho_t$ which are contained in $\Nil$. This case will be discussed in Section \ref{subsc:translation} below.
\item Theorem \ref{thm:basics-equivariant} also holds for 
 any \textit{general} extended frame $F$ of a harmonic map $g$ which satisfies 
 \eqref{eq:fandgtrans}.
\end{enumerate}

\end{Remark}
\begin{Definition}
 A general extended frame satisfying
\begin{equation*}
F(z+t,\lambda) = M_t(\lambda) F(z, \lambda)
\end{equation*}
 will be called {\it $\R$-equivariant}.
\end{Definition}

\subsection{A chain of extended frames}
 For a detailed discussion of the relation between 
 spacelike CMC surfaces in Minkowski $3$-space $\Min$ 
 and minimal surfaces in $\Nil$ it is important to 
 use extended frames with specific additional properties.
 In \cite{DIKAsian}, also see \eqref{special frame}, 
 a specific extended frame was defined for all $\lambda =1$ 
 and the matrix entries were  (by definition) the spinors associated 
 with the associated family $\{f^\lambda\}_{\l \in \mathbb S^1}$ of $f$. 
 Note that the spinors of  a minimal surface in  $\Nil$  are defined 
 uniquely up to a common sign. By continuity in $\lambda$, 
 the choice of sign for the $\psi_j$ thus is the
 same for all $\lambda$, whence irrelevant.

 Hence the first extended frame in our chain is an extended frame 
 mentioned above  and denoted by $F(z,\lambda)$ in \eqref{eq:Fequiv}.
 As pointed out in Theorem \ref{thm:basics-equivariant}, 
 this extended frame will, in general,
 not be $\R$-equivariant
 under the action of the translational one-parameter group $z \rightarrow z+t$.
 But we have shown that there exists some function $\ell (z)$ such that
 $F_\ell = F \ell$ defines an $\R$-equivariant general extended frame for the
 translational one-parameter group. The frame $F_{\ell}$ is our second frame.
 Finally we consider an $\R$-equivariant extended frame which 
 also attains the value $\id$  at $z = 0$ for all $\lambda$: 
 $\hat F(z,\lambda) = F_{\ell} (0,\lambda)^{-1} F_{\ell} (z,\lambda)$. 
 
 Thus we have the following triple of extended frames
\begin{equation}\label{eq:correspondence}
F \longrightarrow F_{\ell} \longrightarrow \hat F.
\end{equation}

\begin{Remark}
 Note, the frames $F$ and $F_{\ell}$ generate the same surfaces in 
 $\Min$ and in $\Nil$ via the respective Sym formulas.
 The frame $\hat{F}$ generates in $\Min$ a surface which is isometric 
 to the previously generated surface, but the corresponding surface in $\Nil$ has, in general, no simple relation to the other (two) surfaces in $\Nil$.
 However, as will explained below, exactly this frame 
 yields a very simple 
 ``degree-one-potential''
 from which we will be able to construct what we want.
 Note, in such a chain, if one assumes  that any of these extended frames has 
 a translational one-parameter group of symmetries, 
 then all three frames  have such a symmetry.
 The frames $\hat F$ are $\R$-equivariant
 general extended 
 frames of the normal Gauss map $g$, where $g:\St \rightarrow \Ha^2$ is 
 non-holomorphic (since the surface has non-vanishing Abresch-Rosenberg differential) 
 harmonic, and also define 
 spacelike CMC surfaces in Minkowski 3-space $\Min$.
 For more details on $\R$-equivariant harmonic maps see, for example  
 \cite{Burstall-Kilian} 
 and for spacelike CMC surfaces in $\Min$ see, for example \cite{BRS:Min}. 
\end{Remark}
 
\subsection{The construction principle}\label{subsc:construction}
 In order to construct $\R$-equivariant minimal surfaces in $\Nil$ 
 we will start in general from some special potential and will 
 arrive at some 
 $\R$-equivariant general extended frame $\hat{F}$, assuming the 
 monodromy has the required properties. (In a sense just  reversing the arrows 
 in \eqref{eq:correspondence} above.)
 What special potentials we will need to start from 
 will be the contents of the next sections.

 At any rate, we will obtain the transformation behaviour 
 (for $t \in \R$ and $z \in \St^\prime $): 
\[
 \hat F(z+t,\lambda) = \hat M_t(\lambda) \hat F(z,\lambda)
\]
 and we also know $\hat{F}(0,\lambda) = \id$.
 We will apply  \cite{Burstall-Kilian} to construct all of such frames.
 Note, while the potential will be defined on some strip $\St$,  $\hat F$ may be defined on some smaller strip $\mathbb S^{\prime}$ only, see \cite{K:Explicit}.
 
 After  $\hat F$ has been constructed we want to use this frame to 
 construct  $\R$-equivariant minimal surfaces in $\Nil$.
 But for this  it is important to require that $\hat M_t$ is diagonalizable for
  $\l = 1.$ In particular, the eigenvalues of  $\hat M_t$ 
 need to be unimodular at $\l = 1$, see Theorem \ref{thm:basics-equivariant}.
  Therefore, in general, we need to change 
 the frame $\hat{F}$ to another frame, for which the monodromy is diagonal for $\lambda = 1$. This is achieved by putting $F = SF$, where $S$ diagonalizes the monodromy as required.
 (For more details see below.)
 Comparing to the chain of frames above we observe, that this new frame plays 
 the role of $F_\ell$. 
\begin{Remark}
 The general extended frame $\hat F$ with the right choice of  initial condition 
 $S$ gives the extended frame $F_\ell = S \hat F$, which is not $F$. However, 
 this is irrelevant for the resulting minimal immersion $f$. More precisely, 
 if one plugs $F$ and $F_\ell$ into the Sym formula,  then the resulting 
 minimal surfaces are the same. Thus we will only consider  $F_\ell$.
\end{Remark}
As pointed out already above, the change from $\hat{F}$ to
$F_\ell$ is by multiplication:
\[
 F_{\ell} (z,\lambda)= S(\lambda) \hat F(z,\lambda)
\]
 with $S(\lambda)  \in \LISU$.
 Note since  $\hat M_t$ is diagonalizable at $\l = 1$ for all $t \in \R$, 
 one can choose $S(\lambda)$ such that the monodromy 
 $M_t(\lambda) = S(\l) \hat M_t (\l) S(\l)^{-1}$ 
 of  $F_\ell$ is diagonal for $\lambda = 1$. 
 More precisely, since $\hat M_t(\lambda = 1)$ diagonalizable, we have two cases: 

 {\bf Case 1.} The eigenvalues of $\hat M_t(\lambda = 1)$ are both $1$.
 This means $\hat M_t(\lambda = 1) = \id$. Then 
 we can choose $S(\l) \in \LISU$ arbitrary. \\
{\bf Case 2.} The unimodular eigenvalues of $\hat M_t (\lambda = 1)$ are different. In this case there exists some matrix  $S\in \ISU$ such that 
 $S \hat M_t(\lambda=1) S^{-1}$ is diagonal.
 Inserting $\lambda$ and $\lambda^{-1}$ respectively off-diagonal into 
 $S$ 
 we obtain a matrix $S(\lambda) \in \LISU$ such that 
 $M_t (\lambda) = S(\lambda) \hat M_t(\lambda) S(\lambda)^{-1}$ 
 is diagonal for $\lambda = 1$.
 \footnote{$S$ may depend on $t$, however, 
  a straightforward computation shows that 
  $S = \di (u(t), u(t)^{-1}) \tilde{S}$ where $\tilde{S}$ is independent of $t$. Thus we can assume without loss of generality that $S$ is independent of $t$.}

 Altogether we obtain that $F_\ell (z,\lambda)= S(\lambda) \hat F(z,\lambda)$ 
 is a general extended frame for $g$ which has monodromy $M_t(\lambda)$, and 
 $M_t (\lambda = 1)$ is diagonal.
 As a consequence, 
 we obtain an $\R$-equivariant minimal 
 surface in $\Nil$ defined on some strip  $\St^\prime$  containing the real axis by applying 
 the Sym formula stated in Section \ref{sc:Sym}.
\begin{Remark}
 In both cases above the choice of ``initial condition''  
 $S \in \LISU$ is not unique. Here is what happens for different choices: 

 {\bf Case 1.} In the case of $\hat M_t (\lambda = 1) = \id$ 
 different initial conditions generally yield different equivariant minimal surfaces, 
 see Section \ref{subsc:translation}. 

 {\bf Case 2.} Assume the eigenvalues of $\hat M_t (\lambda = 1)$ 
 are unimodular and different.
 Let $\tilde S \in \LISU$ be  another initial condition such 
 that $\tilde S \hat M_t \tilde S |_{\lambda =1} = S \hat M_t S^{-1}|_{\l =1}$.
 Then  $\tilde S = \delta S$ with some loop $\delta \in \LISU$
 such that $\delta|_{\l =1 }$ is diagonal. 
 Let $F_\ell $ and $\tilde {F}_\ell$ be the corresponding 
 general extended frames associated with the initial conditions 
 $S$ and $\tilde S$, respectively. Then we obtain $\tilde {F}_\ell = \delta F_\ell$.
 Inserting $F_\ell $ and $\tilde {F}_\ell $ into the Sym formula, the resulting minimal surfaces
 are the same up to a rigid motion (see the proof of (b) of Theorem \ref{thm:symmetry} 
 for the computation). 
\end{Remark}

 \subsection{Degree one potentials} 
 In the last subsection we have seen that for every $\R$-equivariant minimal
  surface in $\Nil$ its normal Gauss map is an $\R$-equivariant harmonic map 
 into $\mathbb{H}^2$.
  These maps have been investigated in \cite{BRS:Min}. It will be more helpful to us 
 to follow the approach of \cite{Burstall-Kilian}, translated into our setting.
 Here is our rendering of results of these two papers which are particularly 
 relevant to this paper.  

 We consider  $f : \St \to \Nil$ to be an $\R$-equivariant 
 minimal surface relative to the one-parameter group $(\gamma_t, \rho_t)$,  
 $\gamma_t.z = z + t$, that is, 
\[
 f(\gamma_t.z) = \rho_t. f(z).
\]
 Let $g : \St \rightarrow \mathbb{H}^2$ denote its (non-holomorphic) harmonic  
 normal Gauss map and $\hat F$ an $\R$-equivariant 
 general extended frame for $g$ which 
 attains the value identity at $0$.  Let $\hat M_t(\lambda) \in \LISU$ denote the
  monodromy of $\hat F$.
 
 By following \cite[Section 3]{Burstall-Kilian}  in our setting and \cite{BRS:Min} 
 we obtain the following
 characterization of all $\R$-equivariant minimal surfaces in $\Nil$: 
 
 \begin{Theorem}\label{thm:equivariant}
 Every $\R$-equivariant non-holomorphic harmonic map $g: \St \rightarrow \Ha^2$
 associated 
 with an $\R$-equivariant minimal surface in $\Nil$ 
 can be obtained  from a constant holomorphic potential 
 $\eta = D dz $  of the form 
\begin{equation}\label{eq:Dlambda} 
 D \in \Lisu, \quad 
 D(\l) = \lambda^{-1} w_{-1} + w_0+ \lambda w_{1},
 \quad \det D(\l = 1)\geq 0,
 \quad (w_{-1})_{12} \neq 0,
\end{equation}
 where all $w_j$ are independent of $\lambda$ and $z$ and $(w_{-1})_{12}$ denotes the $(1, 2)$-entry of $w_{-1}$. 
 In particular $D$ has purely imaginary eigenvalues for $\lambda = 1$.
 
 Conversely, every constant $\eta = D dz$ as in \eqref{eq:Dlambda} with
 initial condition $S \in \LISU$ such that 
 $S D S^{-1}|_{\l =1}$  is diagonal,
 generates an $\R$-equivariant harmonic map $g: \St \rightarrow \Ha^2$ defined on some strip 
 $\mathbb S \subset \mathbb C$ parallel to the real axis, and,
 by the Sym-formula  \eqref{eq:symNil}, generates an $\R$-equivariant minimal surface in $\Nil.$ 
\end{Theorem}
\begin{proof}
 Following the proof of  \cite[Section 3]{Burstall-Kilian} verbatim we obtain the first two statements of \eqref{eq:Dlambda}. 
 The last statement expresses the fact that we assume $f$ 
 to be an immersion at the base point ``$z = 0$''.
 Hence, to finish the proof of the first part of the claim we 
 only need to prove the statement about the eigenvalues of $D$.
 But for the monodromy $\hat M_t$ of $\hat F$ defined in the last subsection 
 we have  $\hat M_t(\lambda)  =  F_{\ell} (0,\lambda)^{-1} M_t(\lambda) 
 F_{\ell} (0,\lambda)$. 
 Thus $\hat M_t(\lambda = 1)$ has the same eigenvalues as 
 $M_t(\lambda = 1)$, where $M_t$ is the monodromy of a general
 extended frame $F_\ell$.
 But $\hat M_t(\lambda) =  
 \exp(t D(\lambda))$ by definition of $D(\lambda$), see \cite{Burstall-Kilian},  and we know that  $\hat M_t(\lambda = 1) $ is diagonalizable for all $t$. Hence 
 $D(\lambda = 1)$ has only purely imaginary eigenvalues and the claim follows.
 The proof of the second part of the claim follows  from  \cite[Section 3]{Burstall-Kilian}
 and the fact that we need diagonalizable monodromy in our situation.
\end{proof}

\begin{Remark} 
\mbox{}
\begin{enumerate}
 \item  The potential $\eta = D dz$ will be called the {\it degree one potential} 
 of an $\R$-equivariant minimal surface $f$.
 \item  The theorem above does not specify the size of the strip  
 $\St$
 in the second part of the theorem, since the Iwasawa decomposition
 of $\exp (z D)$ is not global.
 This issue will be discussed in the 
 forthcoming paper \cite{K:Explicit}.
\item Since $D \in \Lisu$, 
 the diagonalizablity condition in Theorem \ref{thm:equivariant}
 immediately implies that $\det D(\lambda =1) \geq 0$ and moreover, 
 when $\det D(\lambda =1) =0$, then $D (\lambda =1) =0$ follows. On the contrary, 
 in Proposition 
 \ref{prp:non-equivariant} when $\det D(\lambda =1) =0$ and $D (\lambda =1) \neq 0$ 
 or $\det D(\lambda =1) < 0$, we obtain non-equivariant minimal immersions 
 which have an equivariant normal Gauss map.
 
\end{enumerate}
\end{Remark}

 With the notation of Theorem \ref{thm:equivariant}, and the explanation of the construction principle in the previous subsection, the procedure of constructing
 $\R$-equivariant minimal surfaces
 in $\Nil$ from degree one potentials $D$ is as follows:
 
 Let us consider the solution $C$,  taking values in $\LSL$,  of the holomorphic ODE 
 $d C = C \eta$ with  $\eta = D dz$ and initial condition $\id$, 
Hence we obtain  $C(z, \l) = \exp (z D(\l)).$ Then 
we  perform an Iwasawa decomposition of $C$ near $z = 0.$ We obtain
\[
C = \hat{F} V_+, 
\]
 where $\hat{F}$ and $V_+$ take values in $\LISU$ and $\LSLP$, respectively.
 We then choose $S \in \LISU$ such that it diagonalizes 
  $D$ for $\lambda = 1$, that is $S(\l) \exp (t D (\l)) S(\l)^{-1}$ is 
 diagonal at $\l =1$.
 Since $ S(\l) \exp (t D (\l)) S(\l)^{-1}$ takes values in $\LISU$, we have
   for $F_\ell = S(\lambda) \hat{F}$
\[
 F_\ell(z+ t, \l) =M_t(\l) F_\ell (z, \l), \quad M_t (\l) = S(\l) \exp (t D (\l)) S(\l)^{-1}.
\] 
 Then by the construction, $F_\ell $ is a general extended frame of some 
 $\R$-equivariant harmonic map 
 $g:\mathbb S \to \mathbb H^2$. Moreover since $M_t(\l =1)$ is diagonal by 
 construction, the corresponding minimal surface $f$  in $\Nil$ is 
 also $\R$-equivariant:
\[
 f(z+t) = \rho_t .f(z) 
\] 
 where $\rho_t \in \isoo$.
\subsection{Monodromy matrices and  symmetries induced by $\R$-equivariant actions}
 Note that to compute $\rho_t$ for all $\R$-equivariant minimal surfaces, which are 
 obtained from degree one potentials, it is not necessary 
 to work out  the Iwasawa 
 decomposition explicitly. It suffices to know  the monodromy 
 $M_t(\l) = S(\l)\exp (t D (\l)) S(\l)^{-1}.$ 
 In particular, the transformation behaviour of the  $\R$-equivariant minimal 
 surface $f$ in $\Nil$ under the transformation $z\mapsto z + t$:
\[
 f(z+ t) =  \rho_t. f(z),\quad \rho_t =((p_t, q_t, r_t), e^{it \theta}), 
\] 
 is determined by $M_t$ explicitly. 
 In fact we consider a degree one potential 
 $\eta = D dz, \; z \in \C$ and $\lambda \in \C^*$, 
 and write  the matrix $D$ in the form
\begin{equation}\label{eq:X9}
 D(\l) = 
\lambda^{-1}
\begin{pmatrix} 
 0 & a \\ b & 0
\end{pmatrix}
+ 
\begin{pmatrix}
i c & 0 \\ 0 & -i c
\end{pmatrix}
+
\lambda
\begin{pmatrix} 
0 & \bar b \\
\bar  a & 0
\end{pmatrix}
=
\begin{pmatrix} 
ic & \lambda^{-1} a+ \lambda \bar b \\
\lambda^{-1} b + \lambda \bar a& -ic
\end{pmatrix},
 \end{equation}
 where $a \in \C^{\times}, b \in \C$, $c \in \mathbb R$
 and $\det D(\l = 1)=  c^2 -|a + \bar b|^2 \geq 0 $.
 Then Theorem \ref{thm:symmetry} actually tells us 
 how to compute $e^{i t \theta}$ and  $\rho_t$.  
  Let $\hat f$ be the immersion obtained by inserting 
 $F_\ell = S(\lambda) \hat{F}$ 
 into the Sym formula \eqref{eq:symNil} with $\lambda=1$. 
 Then the proof of Theorem \ref{thm:symmetry} shows that $\hat f$ changes 
 under  $\gamma_t$ as  follows
\begin{align*}
\hat f(\gamma_t. z)  =& \left.\left\{\ad(M_t) \hat f(z)  + 
 \frac{1}{2}[X_t^o, (\ad(M_t) f_{\Min}(z))^o]^d  + X_t^o +Y_t^d\right\}\right|_{\lambda =1},
 \end{align*}
 where $f_{\Min}$ is the map defined in \eqref{eq:SymMin}, and
\[
 X_t =  -i \lambda (\partial_{\lambda} M_t) M_t^{-1}, \quad
 Y_t =- \frac{i}{2} \lambda \partial_{\lambda} X_t 
 =  -\frac{1}{2}  \lambda \partial_{\lambda} (\lambda (\partial_{\lambda} M_t) M_t^{-1}).
\]
 As proved in Theorem \ref{thm:symmetry}, the resulting minimal surface 
 $f$ satisfies 
 \begin{align*} 
f(\gamma_t. z)  & = \rho_t. f(z) \quad \mbox{with}\;\;\rho_t = 
 \left( (p_t, q_t, r_t), \; 
 e^{it \theta} \right) 
 \intertext{where we set  $\theta = \det D|_{\l =1} = c^2 - |a + \bar b|^2 \geq 0$,}
\nonumber 
X_t|_{\lambda =1} &= \frac{1}{2}\begin{pmatrix}* & - q_t +  i p_t \\ - q_t -  i p_t & * \end{pmatrix}
\quad \mbox{and}\quad 
 Y_t|_{\lambda =1} = \frac{1}{2}\begin{pmatrix}- i r_t &*\\ * & i r_t \end{pmatrix}.
\end{align*} 
 We want to compute $X_t$ and $Y_t$ in more detail.
 For this we write $\lambda = e ^{iv}$, then for any function $H(\lambda)$ we have
$\dot H = \frac{d}{dv} H = i \lambda \frac{d}{d\lambda}H$. Thus 
\[
X_t = -\dot M_t M_t^{-1}, \quad Y_t = \frac{1}{2}\left\{\ddot M_t M_t^{-1} -
 (\dot M_t M_t^{-1})^2\right\}.
\]
 A straightforward computation shows the following corollary. 
\begin{Corollary}\label{coro:XY}
 If $M_t= S \hat M_t S^{-1}$, 
 then $X_t$ and $Y_t$ can be computed as
\begin{align*}
X_t|_{\l =1} &= -S\left([S^{-1}\dot{S}, \hat M_t]+\dot{\hat{M}}_t\right)  \hat M_t^{-1}S^{-1} |_{\l =1},\\
Y_t|_{\lambda =1} &= \frac{1}{2}
S\left([S^{-1}\dot{S}, L_t]+\dot{L}_t - L_t \hat M^{-1} L_t\right)  \hat M_t^{-1}S^{-1} |_{\l =1},
\end{align*}
 where we set $L_t = [S^{-1}\dot{S}, \hat M_t]+\dot{\hat{M}}_t$.
\end{Corollary}
Note, an inspection of the last two formulas yields that $X_t|_{\l =1}$ and $Y_t|_{\l =1}$, and therefore also $\rho_t$, 
can be computed from $D$.

\begin{Remark}\label{rm:eigenvalues}
 The condition $\det D (\l = 1)>0$, that is, the monodromy matrix 
 $M_t(\l) = S(\l)\exp (t D(\l)) S(\l)^{-1}$ has unimodular 
 eigenvalues at $\l =1$, 
 is purely local, since  $\det D (\l)$ takes non-positive values
 in general only for some $\l \in \mathbb{S}^1$.
\end{Remark}
\subsection{Translation invariant minimal surfaces}\label{subsc:translation}
 It is clear that all $\R$-equivariant minimal surfaces induce 
 some one-parameter group $\{\rho_t\}_{t \in \R} \subset \isoo$, and by Theorem \ref{thm:oneparameter}, 
 such one-parameter groups describe a helicoidal motion or a translation motion.
 Therefore in the following sections we characterize helicoidal and 
 translation invariant minimal 
 surfaces by the degree one-potentials in detail.

In this section we characterize translation invariant 
\eqref{eq:translationmotion} minimal surfaces in $\Nil$.
\begin{Theorem} \label{transinv}
 Let $f$ be a translation invariant minimal surface.
 Then $f$ is $\R$-equivariant. Moreover, the corresponding degree 
 one potential $\eta = D dz$
 as in \eqref{eq:X9} satisfies $D|_{\l = 1} = 0$.

 Conversely, let $\eta = D dz$ be as in \eqref{eq:X9} 
 a degree one potential satisfying 
 $D|_{\l = 1} = 0$. Then the resulting $\R$-equivariant 
 minimal surface is a translation invariant minimal surface. 
\end{Theorem}
 \begin{proof}
 Let $f$ be a translation invariant minimal surface. Then 
 it is clear that $f$ does not have a fixed point on the surface and
 thus it is an $\R$-equivariant surface by Theorem \ref{thm:X1} 
 and Theorem \ref{thm:Requivariant}
 and thus there exists a degree 
 one potential $D dz$ with $D$ as in \eqref{eq:X9}. 
 We also know $f(z+t) = \rho_t.f(z)$ with $\rho_t$ a one-parameter group  of isometries of  $\Nil$ as described in  (\ref{eq:translationmotion}).
 In general, the rotation part of a symmetry $\rho_t$ yields, up to a factor $1/2$ the eigenvalues of 
$M_t(\l)$ at $\lambda =1$. Under our assumption the rotation part of $\rho_t$ is trivial, whence the eigenvalues of $M_t(\l)$ are identically $1$ at $\lambda =1$. But then the eigenvalues of $D(\lambda = 1)$ vanish and since this matrix is diagonalizable, $D(\lambda = 1)=0$
follows.
 
 Conversely, let us start from some degree one potential $D$ 
satisfying  $D|_{\l =1} =0$. From this we infer that 
 $\hat M_t|_{\l =1} = \exp (t D)|_{\l =1} =\id$, 
 whence the resulting equivariant surface does not have a rotation part, 
 that is, $\theta =0$.
 Hence by Theorem \ref{thm:symmetry}, we conclude that the original 
 one-parameter group in $\isoo$  actually is contained in $\Nil$.
 Therefore the surface is a translation invariant minimal surface.
 \end{proof}
 We now compute  the one-parameter group 
 $\{\rho_t\}_{t \in \R}$ with $\rho_t =(p_t, q_t, r_t) \in \Nil$ 
 given by the degree one potential $D d z$ with $D|_{\l =1}=0$
 as follows.  Since $D|_{\l =1}=0$, we obtain that $D$ has the form 
\begin{equation*}
 D(\l) =
\begin{pmatrix} 
 0 & a(\lambda^{-1}  - \lambda)  \\
\bar a (-\lambda^{-1} + \lambda) & 0
\end{pmatrix}, \quad a \in \C^{\times}.
 \end{equation*}
 We know from Section \ref{subsc:construction} 
 that in the present case we can choose for $C(z, \l)$ 
 any in initial condition $S(\lambda)$ taking values in $\LISU$.

\begin{Example}
 We first choose the initial condition $S(\lambda)  \equiv \id$.
 Then $M_t = \exp (t D)$ and by Corollary \ref{coro:XY}, we have 
\[
\mbox{ $(1, 2)$-entry  of $X_{t}|_{\lambda =1}$}=
 2 i a t
\;\;\mbox{and}\;\;
\mbox{$(1, 1)$-entry  of $Y_{t}|_{\lambda =1}$}=0.
\]
 Thus  $\rho_t$ is given by 
\begin{equation}\label{eq:translationdirection1}
\rho_t = ((p_t, q_t, r_t), 1)= ((4 t \Re a, 4 t \Im a, 0 ), 1).
\end{equation}
 Thus the surface is a translation invariant minimal surface 
 with a direction $\rho_t$ given in \eqref{eq:translationdirection1}.
\end{Example}

\begin{Example}\label{ex:translation}
 We next normalize without loss of generality
 to $a =1$: Conjugate, if necessary,  $D$ by a diagonal matrix so that $a$ is changed into a positive real number. Then change the complex coordinates by scaling.
 Now we choose another initial condition $S$, namely
 $S|_{\l=1} =$ ``boost'',
\[
 S|_{\l =1} = 
 \begin{pmatrix}
 \cosh p & e^{i q}\sinh p \\ e^{-i q}\sinh p& \cosh p
 \end{pmatrix} \in \ISU, \quad (p, q \in \R).
\]
 Note, any $\tilde S \in \LISU$ can 
 be decomposed as 
\[
 \tilde S = \di (e^{i \ell}, e^{-i\ell }) S, \quad \left(\ell \in \R, \;\; 
 \tilde S 
 \in \LISU\right), 
\] 
 where $S|_{\l=1}$ is a boost. Then the resulting surface defined  by using 
 the initial condition $S$ is congruent to the surface given by the initial 
 condition $\tilde S$. Thus we only need to consider a boost as an initial condition.
 Without loss of generality we can assume $p\geq 0$ and $q \in [0, 2 \pi)$.
 Since the Iwasawa decomposition of $\exp (z D) = F V_+$ can be computed directly as 
\[
 \exp (z D) = 
\exp\begin{pmatrix} 0  & z \l^{-1} - \bar z \l  \\
		     -z \l^{-1} + \bar z \l  &0
		   \end{pmatrix}
\exp \begin{pmatrix} 0  &  (\bar z- z) \l  \\
		      (- \bar z+z) \l  &0
		   \end{pmatrix},
\] 
  a straightforward computation yields
\[
 S F|_{\lambda =1} = 
\begin{pmatrix}
 \cosh s \cosh p - i  e^{i q} \sinh s \sinh p &  i \sinh s \cosh p + e^{i q} \cosh s \sinh  p \\ 
-i \sinh s \cosh p + e^{-i q} \cosh s \sinh  p & \cosh s \cosh p + i  e^{-i q} \sinh s \sinh p
\end{pmatrix},
\]
 where $ s= 2 \Im (z)$. From this it is easy to see that as spinors $\psi_1$ 
 and $\psi_2$ at $\lambda =1$ one can choose
\[
\psi_1|_{\lambda =1} = \sqrt{i} ( \cosh s \cosh p - i  e^{i q} \sinh s \sinh p), \quad 
\psi_2|_{\lambda = 1} = \sqrt{i} (i \sinh s \cosh p + e^{i q} \cosh s \sinh  p ).
\]
 Then  another straightforward computation shows that 
 the conformal factor of the metric of the resulting surface is 
\[
 e^{u/2 } = 2 (|\psi_1|^2 + |\psi_2|^2 )= 2 m 
 \cosh \left(
 4 y + \cosh^{-1} \left(\frac{\cosh 2 p}{m} \right)
 \right),
\]
 where $m =\sqrt{(\cosh 2 p)^2 - (\sin q \sinh 2p)^2}$ and $z = x + i y$. 
 Here note that $m>0$.
 In particular if $q = 0$, then $ |\psi_1|^2 + |\psi_2|^2  = \cosh 2 p \cosh 4 y$.
 From this, for any pair $(p, q) \in [0, \infty) \times  [0, 2 \pi)$, there exists 
 a $(\tilde p, 0) \in [0, \infty) \times  \{0\}$ such that the conformal factors 
 are the same function up to a translation in $y$. Therefore the resulting 
 translation invariant minimal surfaces are parameterized by $p \in [0, \infty)$. 
 
 For the present case, where $q=0,$ the resulting translation invariant minimal 
 surface can be computed as follows:
 \[
  SF = S \begin{pmatrix}
  \cosh s & i \sinh s \\
  - i \sinh s & \cosh s
	 \end{pmatrix},
 \]
 where $s = 2 \Im (\lambda^{-1}z)$ and 
 $S|_{\lambda = 1} = \begin{pmatrix}
  \cosh p & \sinh p \\
  \sinh p & \cosh p
	 \end{pmatrix}$.
 Then the resulting surface $\hat f$  can be computed as in the proof of 
 Theorem \ref{thm:symmetry}
\begin{align*}
 \hat f(z) = 
\left(\ad (S) f_{\Min}(z)\right)^o  +
\left(\ad(S)  \left(-\frac{i}{2} \lambda \partial_{\lambda} f_{\Min}(z)\right)\right)^d
 + X ^o  +  \left(\frac{1}{2}[X, \ad(S) f_{\Min}(z)] 
 +Y\right)^{d}\Big|_{\lambda = 1},
\end{align*}
 where $X = - i \lambda (\partial_{\lambda} S) S^{-1}$,
 $Y = -\frac{i}2 \lambda \partial_{\lambda} X$,
\[
f_{\Min}(z)|_{\lambda = 1} = \begin{pmatrix} - \frac{i}2 \cosh (4 y) & 2 i x - \frac12\sinh (4y) \\ - 2 i x - \frac12\sinh (4y)& \frac{i}2 \cosh (4 y)	  \end{pmatrix}
\]
 and 
\[
-\frac{i}{2} \lambda \partial_{\lambda} f_{\Min}(z)|_{\lambda = 1} = 
\begin{pmatrix} - i x \sinh (4 y)&  - i y -x \cosh (4 y)\\ i y - x \cosh (4 y) &
 i x \sinh (4 y) \end{pmatrix}.
\]
  Let us consider the minimal surface  
 \[
\hat f_S (z) =  \left(\ad (S) f_{\Min}(z)\right)^o  +
\left(\ad(S)  \left(-\frac{i}{2} \lambda \partial_{\lambda} f_{\Min}(z)\right)\right)^d
\Big|_{\lambda = 1}.
 \]
 Then the term $ X ^o  +  \left(\frac{1}{2}[X, \ad(S) f_{\Min}(z)] 
 +Y\right)^{d}\Big|_{\lambda = 1}$ denotes the translation of 
 $\hat f_S (z)$. Moreover, the resulting minimal translation invariant 
 surface $f_S (z)$ is explicitly given by 
\begin{equation}\label{eq:translationmoduli}
 f_S (z) = \begin{pmatrix}
 4 \cosh (2 p) x+ \sinh (2 p) \cosh (4 y)\\ 
 \sinh (4 y) \\
 -2 \sinh (2 p) y + 2 \cosh(2 p) x \sinh (4 y) 
  \end{pmatrix}, 
\end{equation}
 where $z = x + i y$. It is easy to see that $f_S (z)$ satisfies 
 \eqref{eq:translation} and thus it is a translation invariant minimal surface.
\end{Example}

\begin{Remark}
 Note that $ f_S $ in \eqref{eq:translationmoduli}
 is exactly the same surface as the following one given 
 in \cite[Theorem 6]{FMP}, \cite[Part II, Example 1.8]{IKOS}:
\begin{equation}\label{eq:translation}
 x_3 = \frac{x_1 x_2}{2} + c 
\left( \frac{x_2 \sqrt{1 + x_2^2}}{2} + \frac{1}{2} \ln \left(x_2 + \sqrt{1 + x_2^2}\right)  \right)
\end{equation}
 with $c = \sinh 2 p \in \R$ 
 (see also \cite[Example 8.2]{Daniel:GaussHeisenbergw}).
 These surfaces are products of two appropriate curves 
 (see \cite[Part II, Example 1.8]{IKOS}, \cite{ILM}).
\end{Remark} 

\subsection{Helicoidal minimal surface}\label{sbsc:helicoid}
 Next we consider helicoidal surfaces,  
 in particular $\R$-equivariant surfaces  for which 
 $\rho_t$ is not contained entirely in $\Nil$. By Theorem \ref{thm:equivariant} and Theorem \ref{transinv} this is exactly 
 the case when the degree one potential 
 $D$ satisfies  $\det D|_{\l=1} = c^2 - |a + \bar b|^2 >0$.
 
 Computations with general coefficients $a,b,c$ are obviously quite laborious. Therefore we will restrict here to the case \eqref{eq:c} below.
 Note that coefficients can be changed/simplified  by using scalings of coordinates and/or immersions and one can move from one surface to another one in the same associated family etc. It is conjectured, that up to such manipulations the basic helicoidal surfaces can all be generated from the ones with $a=1$ and $c =2$.
 Therefore, we normalize $a$ and $c$ as
\begin{equation}\label{eq:c}
a =1 \quad \mbox{and}\quad c =2,
\end{equation}
 respectively. 
 It seems that we can prove that without loss 
 of generality $a$ and  $c$ can be normalized as in \eqref{eq:c}, however, 
 it is rather complicated and we postpone the proof until 
 the forthcoming paper \cite{K:Explicit}.

 Then the condition $\det D|_{\l=1}>0$ is equivalent to 
 that $b$ is inside the open disk 
\begin{equation}\label{eq:diskD}
\D = \left\{ b \in \C \;\;|\;\; |1 + b|^2 < 4 \right\},
\end{equation}
 that is, the disk with center $(-1, 0)$ and radius $2$ in the complex plane. Thus we have the following theorem.
\begin{Theorem}
 Let $f$ be a helicoidal minimal surface in $\Nil$.
 Then the corresponding degree one potential $\eta = D dz$
 satisfies $\det D|_{\l =1}>0$.
 Conversely,  let $\eta = D dz$ be a degree one potential 
 which satisfies condition \eqref{eq:c} and $\det D|_{\l =1}>0$.  
 Then there exists a helicoidal minimal surface with respect 
 to the axis through the point 
 $\alpha = a^h \in \Nil$ parallel to the $e_3$-axis with
 pitch $\cc$, where $\alpha$ and $\cc$ are defined by
\begin{equation}\label{eq:beta}
 \alpha = \frac{i (2 +\ell)(-6 +\bar b +  b(3 + 2 \Re b) + 4 \ell)}
 {\ell^2 (1 + b)\sqrt{4 - \ell^2}},
\end{equation}
\begin{equation}\label{eq:s}
\cc =  \frac{-2(3\Re b -   (\Re b)^2  - |b|^2  \Re b - |b|^2)}
 {\ell^4},
\end{equation}
 with $\ell = \sqrt{\det  D|_{\l=1}} = \sqrt{3 - 2 \Re b -|b|^2}<2$.
 
 Moreover, the minimal helicoidal surface becomes a rotational 
 surface $($for obvious reasons usually called {\rm catenoid}$)$
 if and only if the pitch $\cc$ vanishes, that is, if
 \begin{equation}\label{eq:catcond}
3 \Re b  -  (\Re b)^2 -  |b|^2 \Re b - |b|^2 =0
 \end{equation}
 holds.
\end{Theorem}
\begin{proof}
 Clearly, any helicoidal minimal surface $f$ does not 
 have a fixed point on the surface and
 thus it is an $\R$-equivariant surface by 
 Theorem \ref{thm:Requivariant}.
 Thus the normal Gauss map $g$ is also equivariant and thus there exists 
 a degree one potential $\eta = D dz$ by Theorem \ref{thm:equivariant}. 
 Since $f$ it is not a translation minimal surface, the eigenvalues of the 
 monodromy matrix $M_t$ are unimodular and distinct, thus $D$ satisfies  
 $\det D|_{\l=1}>0$.

 Conversely, let $\eta = D dz$ be a degree one potential which satisfies 
  condition \eqref{eq:c} and $\det D|_{\l=1}>0$. 
  
 Then let $\Vec{e}_1$ and  $\Vec{e}_2$ denote orthonormal 
 (with respect to the indefinite  Hermitian inner product) 
 eigenvectors of $D|_{\l = 1}$. Then $(\Vec{e}_1,\Vec{e}_2 ) \in \ISU$ and 
 the matrix $S$, given by
 \begin{equation}\label{eq:C0}
 S^{-1} = \di (\l^{1/2}, \l^{-1/2}) (\Vec{e}_1, \Vec{e}_2) \di (\l^{-1/2}, \l^{1/2}),
 \end{equation}
 is contained in  $\LISU$.
 If we choose $S$ as an initial condition for the solution to $dC = C \eta$, then we obtain
\[
 M_t|_{\l =1} = S \exp (t D) S^{-1}|_{\l = 1} 
 = \di (e^{i t \ell } , e^{-i t \ell}).
\]
 Then by using Corollary \ref{coro:XY},
 $X_t = -i \l (\partial_{\l} M_t) M_t$ and 
 $Y_t =  \frac{1}{2}  \lambda \partial_{\lambda} (\lambda (\partial_{\lambda} M_t) M_t^{-1})$ can be computed as
 \begin{enumerate}
  \item[] the $(1, 2)$-entry of $X_t|_{\lambda = 1}= 
 \displaystyle \frac{i}{2} \alpha\left(1- e^{2 i \ell t}\right)$,
  \item[] the $(1, 1)$-entry of $Y_t|_{\lambda = 1}= \displaystyle -\frac{i}{2} \left(\cc  2 \ell t -\frac{|\alpha|^2}{2} \sin 2 \ell t \right)$,
 \end{enumerate}
 where $\alpha$, $\cc$ and $\ell$ are given in \eqref{eq:beta} and \eqref{eq:s}, 
 respectively.
 Thus in the relation  $f(\gamma_t .z ) = \rho_t . f(z)$ the one-parameter group $\rho_t$ can be computed:
 \begin{equation*}
 \rho_t = \left( \left( \Re (\alpha(1-e^{2 i \ell t })),  \;\Im (\alpha(1- e^{2 i\ell t})), \;\cc 2 \ell t -\frac{|\alpha|^2}{2} \sin 2 \ell t \right), \;e^{2 i\ell t}\right).
\end{equation*}
 From \eqref{eq:helicoidalmotion}, $\rho_t$ is a helicoidal motion
  with angle $ 2 \ell t$ through the point  $(\Re (\alpha), \Im(\alpha), 0)$ and the pitch $\cc$.
  
 Finally, from \eqref{eq:beta} and \eqref{eq:s} it is easy to see that 
 the helicoidal motion gives a rotation if and only if the pitch $\cc$ vanishes, 
 that is, \eqref{eq:catcond} holds. This completes the proof.
\end{proof} 

\begin{Remark}
\mbox{}
\begin{enumerate}
 \item 
 Let us consider the case $b = 0$ in \eqref{eq:X9} with $a =1$ and $c =2$. 
 It is easy to see that $\det D|_{\l =1}>0$ holds. Moreover, this case 
 was already considered 
 in \cite{DIKAsian}, and the resulting 
 surface is a \textit{horizontal plane} or a 
 \textit{horizontal umbrella} depending on the initial condition $S$. 
 Since we  are interested in the case of equivariant minimal surfaces, we consider
 only horizontal planes.
\item Let us consider the case $b = 1$ in \eqref{eq:X9} with $a =1$ and $c =2$. 
 It is easy to see that $\det D|_{\l =1}>0$ holds. Moreover, this case 
 was already considered 
 in \cite{DIKAsian}, and the resulting 
 surface is a \textit{horizontal plane}.

\item Let us consider the case $a = 1$ in \eqref{eq:X9} with $c =1-b$ and 
 $0< b< 1$.  It is easy to see that $\det D|_{\l =1}>0$ holds.
 It is known that the resulting spacelike CMC surface $f_{\Min}$  in $\Min$,
 see Figure 3 in \cite{BRS:Min}, is given by elementary functions.
 It has been called \textit{semitrough} \cite[page 98]{HTTW}
 and the corresponding minimal surface 
 is the same surface as the one given in Example 8.4 of \cite{Daniel:GaussHeisenbergw}. 
\end{enumerate}

\end{Remark}

\subsection{Minimal surfaces with $\R$-equivariant normal Gauss maps}
 As we have shown that equivariant minimal surfaces $\Nil$ have 
 equivariant non-holomorphic harmonic normal 
 Gauss maps and they induce the degree one potentials $\eta = 
 D \> d z$. Conversely, $\eta = D \> d z$
 with $D|_{\l=1} =0$ or $\det D|_{\l=1} > 0$ induces an equivariant minimal surface
 in $\Nil$.  In particular in the case of $\det D|_{\l=1}>0$, 
 the initial condition $S \in \LISU$ is important to construct an helicoidal 
 minimal surface, and it is essentially unique. If we choose an arbitrary 
 intial condition $S \in \LISU$, then the resulting minimal surface is no longer 
 equivariant.
\begin{Corollary}
 Let $\eta = D dz$ be a degree one potential 
 which satisfies the condition \eqref{eq:c} and $\det D|_{\l =1}>0$.  
 Then there exist a two-parameter family of minimal surfaces which are 
 symmetric with respect to $(\gamma, \rho)$ given by 
 $\gamma : z \mapsto z + 2 \pi/\sqrt{\det D}\big|_{\l=1}$ 
 and $\rho = ((p, q, r), 1)$ given in \eqref{eq:pqr}, that is, 
 the resulting surface is periodic, but it is not 
 equivariant in general.
\end{Corollary}
\begin{proof}
 We choose an initial condition $\hat S$ 
 in the construction of the resulting minimal surface $f$
 given by the degree one potential $\eta =D dz$
 such that 
\[
 \hat S = B_0 S \in \LISU, \quad \mbox{where}\quad 
 B_0|_{\l=1}= \begin{pmatrix}
 \cosh p & e^{i q}\sinh p \\ e^{-i q}\sinh p& \cosh p
 \end{pmatrix}\quad (p, q \in \R),
\] 
 and $S$ is the initial condition given in \eqref{eq:C0}.
 Then the monodromy matrix 
\[
 \hat M_t = B_0 S \exp (t D) S^{-1} B_0^{-1}
\]
 at $\l =1$ can be computed as $\hat M_t|_{\l=1} = B_0 \di (e^{i \ell t}, e^{-i \ell t})B_0^{-1}|_{\l=1}$, where $\ell = \sqrt{\det D}|_{\l=1}>0$. 
 Therefore, for $t_0 = 2 \pi/ \ell$, we obtain
 $\hat M_{t_0} (\l =1) = \id$, and thus the resulting surface is symmetric with 
 respect to $(\gamma, \rho)$, where 
 $\gamma : z \to z + 2 \pi/\ell$ and $\rho = ((p, q, r), 1)$ 
 and $p, q, r \in \R$ are given by
\begin{equation}\label{eq:pqr}
\hat X_{t_0}|_{\lambda =1} = \frac{1}{2}\begin{pmatrix}* & - q +  i p \\ - q -  i p & * \end{pmatrix},
\quad \quad 
\hat Y_{t_0}|_{\lambda =1} = \frac{1}{2}\begin{pmatrix}- i r &*\\ * & i r \end{pmatrix},
\end{equation}
 with
\begin{align*}
 \hat X_{t_0}|_{\l=1} &= - \dot {\hat M}_{t_0}{\hat M}_{t_0}^{-1}\big|_{\l =1}, \quad \quad 
 \hat Y_{t_0}|_{\l=1} = \frac{1}{2} 
 \left\{\ddot {\hat M}_{t_0} {\hat M}_{t_0}^{-1} -(\dot {\hat M}_{t_0}{\hat M}_{t_0}^{-1})^2\right\}\big|_{\l =1}.
\end{align*}
 Here $\cdot $ denotes the derivative with respect to $v,  \l = e^{i v}$.
 This completes the proof.
\end{proof}
 It is also natural to think about the remaining cases, that is, 
 the cases where $\det D|_{\l=1}=0$ with $D|_{\l=1}\neq 0$ or $\det D|_{\l=1}<0$.
 It is easy to see that the resulting normal Gauss maps from such  degree one
 potentials $\eta = D \>d z$ are $\R$-equivariant, however, the minimal surfaces in $\Nil$
 are not equivariant. 
\begin{Proposition}\label{prp:non-equivariant}
 Let $\eta = D dz$ be a degree one potential which satisfies the condition 
\begin{equation*}
\mbox{$\det D|_{\l=1}=0\;$ with $\quad D|_{\l=1}\neq 0$ or $\det D|_{\l=1}<0$.}
\end{equation*}
 Then the normal Gauss map of the resulting minimal surface in $\Nil$
 is equivariant, however the resulting surface itself does not have any symmetry.
\end{Proposition}
\begin{proof}
 From the construction, it is clear that the normal Gauss map is 
 equivariant.
 Since the monodromy matrix given by the potential $\eta$ 
 does not have unimodular eigenvalues, thus the resulting surface 
 does not have any symmetry by Theorem \ref{thm:symmetry}.
\end{proof}

\appendix

\section{Preliminary results on $\Nil$, surfaces in $\Nil$ and flat connections for the harmonic normal Gauss map}\label{app:Pre}

\subsection{Heisenberg group $\Nil$}\label{subsc:nil3}
 As in \cite{DIKAsian} we realize the three-dimensional Heisenberg group $\Nil$
 by $\R^3$ with the group multiplication
 \begin{equation*}
 (a_1,a_2,a_3) \cdot (x_1,x_2,x_3) =  \left(a_1 + x_1, a_2 + x_2, a_3 + x_3 + 
 \frac{1}{2} (a_1 x_2 - a_2 x_1)\right).
 \end{equation*}
 and the left-invariant metric 
\[
 ds^2 = dx_1^2 + dx_2^2 + \left(dx_3 + \frac{1}{2} (x_2 d x_1 - x_1 dx_2)\right)^2.
\] 
 The Lie algebra of $\Nil$ will be denoted $\mathfrak{nil}_3$.
 The standard basis $e_1, e_2, e_3$ of $\mathfrak{nil}_3 \cong \R^3$ 
 induces left-invariant vector fields which will be denoted by $E_1, E_2,E_3$,
 see \eqref{eq:Ej}.
 By $\D$ we will always denote a non-compact simply-connected Riemann surface.
 Usually this will mean $\D$ the unit disk or the complex plane.
\subsection{Surfaces in $\mathrm{Nil}_3$}\label{subsc:SurfaceTheory}
 Let $f:\Rim \to \mathrm{Nil}_3$ be a conformal immersion of a Riemann surface.

 We consider the $1$-form $f^{-1}\pz fdz=\varPhi{d}z$ on a simply connected domain 
 $\mathbb{D}\subset \Rim$ (or the universal cover of $\Rim$) that takes values 
 in the complexification $\mathfrak{nil}_3^{\mathbb{C}}$ of the 
 Lie algebra $\mathfrak{nil}_3$. With respect to the natural 
 basis $\{e_1,e_2,e_3\}$ of $\mathfrak{nil}_3$, we expand $\varPhi$ as 
 $\varPhi =\sum_{k=1}^3 \phi_k e_k$ and
 obtain  that $(\phi_1)^2+(\phi_2)^2+(\phi_3)^2=0,$ since $f$ is conformal. 
 Then there exist
 complex valued functions
 $\psi_1$ and $\psi_2$ such that
 \begin{equation*}
 \phi_1 = (\overline{\psi_2})^2 - \psi_1^2, \;\;
 \phi_2 = i ((\overline{\psi_2})^2 + \psi_1^2), \;\;
 \phi_3 = 2 \psi_1 \overline{\psi_2},
 \end{equation*}
 where $\overline{\psi_2}$ denotes the complex conjugate of $\psi_2$.
 It is easy to check that 
 $\psi_{1}\sdz$ and $\psi_{2}\sdzb$ 
 are well defined on $\Rim$. 
 More precisely, $\psi_{1}\sdz$ and $\psi_{2}\sdzb$ are respective 
 sections  of the spin bundles $\varSigma$ and $\bar{\varSigma}$ over $\Rim$.
  
 The sections $\psi_{1}\sdz$ and 
 $\psi_{2}\sdzb$ are called the 
 \textit{generating spinors} of the conformally immersed 
 surface $f$ in $\mathrm{Nil}_3$. 
 The conformal factor $e^{u}$ of the 
 induced metric $\langle df,df\rangle$ and 
 the left translated vector field 
 $f^{-1}N$ 
 of the unit normal $N$ to $\mathfrak{nil}_3$ 
 can be expressed  by the generating spinors as follows:
 \begin{equation}\label{eq:metric}
 e^{u} =4 (|\psi_1|^2+|\psi_2|^2)^2,
 \end{equation}
and
 \begin{equation}\label{eq:Nell}
 f^{-1} N = \frac{1}{|\psi_1|^2+ |\psi_2|^2}
 \left( 
 2 \Re(\psi_1 \psi_2) e_1 +2 \Im(\psi_1 \psi_2)e_2 + 
 (|\psi_1|^2 -|\psi_2|^2)e_3
 \right),
 \end{equation}
 where $\Re$ and $\Im$ denote the real and 
 the imaginary part of a complex number respectively.
 We define a function $h$ by
 \begin{equation} \label{support}
 h =e^{u/2}\langle f^{-1} N, e_3 \rangle  =  2(|\psi_1|^2-|\psi_2|^2).
 \end{equation}
 Then we get a section $h\sdz\sdzb$ of 
 $\varSigma \otimes \bar{\varSigma}$. 
 This section is called the \textit{support} of $f$.  
 The coefficient function $h$ is called the \textit{support function} 
 of $f$ with respect to $z$. The support function $h$ is represented as 
 $h=e^{u/2}\cos \vartheta$. Here $\vartheta$ denotes the angle between 
 $N$ and the Reeb vector field $E_3$
 (called the \textit{contact angle} of $f$). 
 From \cite[Proposition 3.3]{DIKAsian}, it is known that 
 $f$ has support zero at $p $, that is,  $h(p) =0$ if and only if 
 $E_3$ is tangent to $f$ at $p$. Thus a surface $f$ is said to be 
 \textit{nowhere vertical} if it is nowhere tangent to $E_3$. 
 
 In this paper we will usually assume that 
 any surface considered in this paper is nowhere vertical.
 In this case, the map $f^{-1} N$ has a nowhere vanishing third component.
 We usually normalize things so that this component is positive. 
\begin{Remark}
 From \eqref{eq:metric} it follows that $f$ has branch points exactly where
 $\psi_1(p) = \psi_2 (p) = 0$ holds. 
 From  (\ref{support}) it follows that $f$ is vertical exactly, where
 $|\psi_1(p)| = |\psi_2 (p)| $ holds.
 Hence a nowhere vertical surface has no branch points and thus will be an immersion.
 \end{Remark}
 \subsection{The normal Gauss map}\label{subsc:normalGauss}
 We identify the Lie algebra $\mathfrak{nil}_3$ of $\Nil$ 
 with Euclidean $3$-space $\mathbb R^3$
 via the natural basis $\{e_1, e_2,e_3\}$.
 Under this identification, 
 the map $f^{-1} N$ can be considered as a map into the 
 unit $2$-sphere $\mathbb S^2\subset \mathfrak{nil}_3$. 
 We now consider the \textit{normal Gauss map} $g$ of the surface $f$
 in $\Nil$. The map $g$ is defined as the composition of the stereographic 
 projection $\pi$ from the south pole with
 $f^{-1} N$,  that is, $g = \pi \circ f^{-1} N: \D \to \C \cup \{\infty\}$
 and thus, applying the stereographic projection to  $f^{-1} N$ 
 defined in \eqref{eq:Nell}, we obtain
 \begin{equation*}
 g= \frac{\>\>\psi_2\>\>}{\overline{\psi_1}} \;.
 \end{equation*}
 Note that the unit normal $N$ is represented 
 in terms of the normal Gauss map $g$ as 
 \begin{equation}\label{normalGausstounitnormal}
 f^{-1}N=\frac{1}{1+|g|^2}
 \left(
 2\Re (g) e_1+2\Im (g) e_2+(1-|g|^2)e_3
 \right).
 \end{equation}
 The formula (\ref{normalGausstounitnormal}) implies that 
 $f$ is nowhere vertical if and only if $|g|<1$ or $|g|>1$,
 and our usual assumptions imply that always $|g|<1$ holds.
\begin{Remark} 
 The normal Gauss map of a vertical plane satisfies 
 $|g|=1$. Conversely, if the normal Gauss map $g$ of a conformal 
 minimal immersion $f$ satisfies $|g| \equiv 1$, then $f$ is a vertical plane.  
\end{Remark} 

 
 \subsection{Nonlinear Dirac equation and the 
 Abresch-Rosenberg differential}\label{subsec:BT}
 It is known that the generating spinors $\psi_1$ and $\psi_2$ satisfy the following 
\textit{nonlinear Dirac equation}, see \cite{BT:Sur-Lie, DIKAsian} for example:
 \begin{equation}\label{Dirac1}
\slashed{D} \begin{pmatrix} 
\psi_1
\\ 
\psi_2
\end{pmatrix} 
:=
\begin{pmatrix}
\pz\psi_{2}+\mathcal{U}\psi_1
\\
-\pzb\psi_1+\mathcal{V}\psi_2
\end{pmatrix} 
=
\left(
\begin{array}{c}
0
\\
0
\end{array}
\right),
\end{equation}
 where 
 \begin{equation}\label{Dirac2}
 \mathcal U = \mathcal V = - \frac{H}{2}e^{u/2} + \frac{i}{4}h, 
 \end{equation}
 and $e^{u/2}$ and $h$ are expressed by $\psi_1$ and $\psi_2$ via 
 \eqref{eq:metric} and \eqref{support}.  \footnote{The potential in \cite{BT:Sur-Lie}  differs from ours  
 by multiplication $-2$.}
 The complex function $\mathcal U (=\mathcal V)$ 
 is called the \textit{Dirac potential} of the nonlinear 
 Dirac operator $\slashed{D}$.

 The \textit{Hopf differential} $A \, dz^2$ is 
 the $(2,0)$-part of the second fundamental form of $f$ derived from $N$.
 It is easy to see that $A$ can be expanded as 
\begin{equation*}
 A =  2 (\psi_1 \pz \overline{\psi_2} - \overline{\psi_2} \pz \psi_{1})
      +4 i \psi_1^2 (\overline{\psi_2})^2.
\end{equation*}
Next, define $B$ as the complex valued function 
 \begin{equation*}
 B = \frac{1}{4}(2 H + i) \tilde A, \;\; 
 \;\;\; \mbox{where} \;\;\; \tilde A = 
 A + \frac{\phi_3^2}{2 H +i}.
 \end{equation*}
 Here $A$ and $\phi_3$ are respectively the Hopf differential and 
 the $e_3$-component of $f^{-1} \pz f$ for $f$ in $\mathrm{Nil}_3$. 
 The complex quadratic differential $\tilde A\, dz^2$ 
 will be called the \textit{Berdinsky-Taimanov differential}. 
 It is known that $2 Bdz^2$ is 
 the original Abresch-Rosenberg differential \cite{Fer-Mira2, Abresch-Rosenberg}.
 In this paper, by abuse of notation, we call $B d z^2$ 
 the \textit{Abresch-Rosenberg differential}. 
 We define a function $w$ using the Dirac potential $\mathcal U(=\mathcal V)$ 
 by 
 \begin{equation}\label{def-exp(w/2)}
 e^{w/2} =\mathcal{U} 
 = \mathcal V = - \frac{H}{2}e^{u/2} + \frac{i}{4}h.
 \end{equation}
 Here, to define the complex function $w$, 
 we need to assume that the mean curvature $H$ and 
 the support function $h$ do not have any common zero.
 For nonzero constant mean curvature surfaces 
 this is no restriction, however, 
 for minimal surfaces, this assumption is equivalent to 
 that $h$ never vanishes, that is, that these surfaces are
 nowhere vertical. 
 The opposite, minimal vertical surfaces which are always vertical 
 are just vertical planes, as explained above.
 \begin{Theorem}[\cite{Ber:Heisenberg}]
 Let $\D$ be a simply connected domain in $\C$ and $f: \D \rightarrow \Nil$
 a conformal immersion and $w$ the complex function defined in 
 \eqref{def-exp(w/2)}. 
 Then the vector $\tilde \psi = (\psi_1, \psi_2)$
 satisfies the system of equations
 \begin{equation}\label{eq:Lax-Niltilde}
 \pz \tilde \psi = \tilde \psi \tilde U , \;\;
 \pzb \tilde \psi =  \tilde \psi \tilde V, \;\;
 \end{equation}
 where 
 \begin{equation}\label{eq:U-V1}
 \tilde U =
 \begin{pmatrix}
 \frac{1}{2} \pz w  + \frac{1}{2} \pz H e^{-w/2+u/2}&
 - e^{w/2} \\ 
 B e^{-w/2} & 0
 \end{pmatrix}, \;\;
 \tilde V =
 \begin{pmatrix}
 0 & - \bar B e^{-w/2}\\
 e^{w/2} & \frac{1}{2}\pzb w +\frac{1}{2} \pzb H e^{- w/2 + u/2}
 \end{pmatrix}.
 \end{equation}
 Here $e^w$  never vanishes on $\D$.

 Conversely, every vector solution $\tilde \psi$ 
 to  \eqref{eq:Lax-Niltilde}, where $e^w$  never vanishes on $\D$ and
 where  \eqref{def-exp(w/2)}, \eqref{eq:U-V1},
 \eqref{eq:metric} and \eqref{support} are satisfied,
  is a solution to the nonlinear Dirac equation
 \eqref{Dirac1} with \eqref{Dirac2} and therefore is induced by some conformal 
 immersion into $\Nil$.
\end{Theorem}
\subsection{Loop groups}\label{sc:loopgroups}
 Here we recall definitions of various loop groups, see 
 \cite{PreS:LoopGroup} in detail.
 Let $\SL$ be a special linear Lie group of degree $2$, and 
 define a twisted loop group of $\SL$, that is,
 a space of maps from $\mathbb S^1$ into $\SL$:
\begin{equation}
\LSL = \{g: \mathbb S^1 \to \SL \;|\; g(-\l) = \sigma g(\l)\}, 
\end{equation}
 where $\sigma = \ad(\sigma_3)$. We induce a suitable topology (such 
 as a Wiener topology) on $\LSL$ such that $\LSL$ becomes an infinite
 dimensional Banach 
 Lie group. Then we can define several subgroups of $\LSL$:
\begin{equation}
 \LISU  =
  g \in \LSL \;|\; \sigma_3 \left(g(1/ \bar \l)^{t}\right)^{-1} \sigma_3 = g(\l) \},
\end{equation}
\begin{equation}
\LSLPM=\{g \in \LSL \;|\; \mbox{g can be
extended holomorphically to $D^{\pm}$} \}, 
\end{equation}

 where $D^{+}$ (resp. $D^{-}$) denotes 
 inside (resp. outside) of the unit disk on the extended  
 plane $\C \cup \{\infty\}$. These subgroups $\LISU$, $\LSLP$ and 
 $\LSLM$ are called 
 the twisted loop group of $\ISU$, the ``positive'' and 
 the ``negative'' loop groups of $\SL$, respectively.
 By $\LSLPI$ we denote the subgroup of elements of $\LSLP$ 
 which take the value identity at zero. Similarly, by $\LSLMI$
 we denote the subgroup of elements of $\LSLM$ which take
 the value identity at infinity.
\subsection{Flat connections}\label{sbsc:Flat}
 Recall that from our assumptions we know  that the unit normal 
 $f^{-1} N$ is upward, that is, the $e_3$-component of $f^{-1} N$ is positive.
 We assume from now on that 
\[
 \mbox{$H = $ constant.}
\] 
 Hence the matrices 
 $\tilde U$ and $\tilde V$ in \eqref{eq:U-V1} above simplify. 
 Next we introduce a parameter $\lambda$ as follows
 
 \begin{equation} 
 \tilde U^{\l} =
 \begin{pmatrix}
 \frac{1}{2} \pz w  & - \l^{-1}e^{w/2} \\ 
 \l^{-1} B e^{-w/2} & 0
 \end{pmatrix}, \;\;
 \tilde V^{\l} =
 \begin{pmatrix}
 0 & - \l \bar B e^{-w/2}\\
 \l e^{w/2} & \frac{1}{2}\pzb w
 \end{pmatrix}.
 \end{equation}
 At this point we state a result which is crucial for the rest of the paper.
 \begin{Theorem}
 Assume that the mean curvature $H$ is constant.
 Then equation \eqref{eq:Lax-Niltilde} is solvable if and only 
 if the matrix zero-curvature condition 
 \begin{equation}\label{matrixzero-curvaure}
 \tilde U_{\bar{z}}^\lambda - \tilde V_z^\lambda 
 = [\tilde U^\lambda, \tilde V^\lambda]
 \end{equation}
 holds.
 \end{Theorem}
 \begin{proof}
 Writing out the integrability condition for \eqref{eq:Lax-Niltilde} 
 we obtain an equation, where  $(\psi_1, \psi_2)$ is multiplied to 
 $\tilde U_{\bar{z}}^\lambda - \tilde V_z^\lambda - [\tilde U^\lambda, 
 \tilde V^\lambda]$.
 Working out the equation \eqref{matrixzero-curvaure} and subtracting 
 one side from the other, we obtain a diagonal matrix of trace $0$. 
 Since  $(\psi_1, \psi_2)$ only vanishes on a nowhere dense set, 
 the integrability condition is equivalent to that the diagonal coefficients vanish.
 But this is the claim.
 \end{proof}
 From \eqref{matrixzero-curvaure}, it follows that there exists a matrix valued function 
 $\tilde F:\mathbb D \to \LGL$ such that $\tilde F^{-1}d \tilde F 
 = \tilde U^{\l} dz + \tilde V^{\l} d\bar z$.
 
 For the purposes of this paper it will be convenient to change the matrices
  $\tilde U^{\l}$ and $\tilde V^{\l}$ 
 by the gauge $\di ( e^{-w/4}, e^{-w/4})$. We thus obtain the equation
\begin{equation}\label{eq:U-V1lambda}
 \alpha^{\l} = U^{\l} dz + V^{\l} d\bar z
\end{equation}
 with coefficient matrices
 \begin{equation} \label{naturalMC}
 U^{\l} =
 \begin{pmatrix}
 \frac{1}{4} \pz w  &
 - \l^{-1}e^{w/2} \\ 
 \l^{-1}B e^{-w/2} &  -\frac{1}{4} \pz w
 \end{pmatrix},  \quad 
 V^{\l} =
 \begin{pmatrix}
  -\frac{1}{4} \pzb w & - \l \bar B e^{-w/2}\\
 \l e^{w/2} & \frac{1}{4}\pzb w
 \end{pmatrix}.
 \end{equation} 
 Note that this system of equations still is integrable, that is, satisfies the integrability condition 
  \eqref{matrixzero-curvaure} for the new coefficient matrices.
 Using this matrix zero-curvature condition. 
 we can show that  minimal surfaces in $\Nil$ are characterized in terms of 
 their normal Gauss map as follows. We first recall Theorem 5.3 in \cite{DIKAsian}.
 \begin{Theorem}\label{thm:mincharact-1}
 Let $f : \D \to \Nil$ be a conformal immersion which is 
 nowhere vertical and $\alpha^{\l}$
 the $1$-form defined in \eqref{eq:U-V1lambda}.
 Moreover, assume that the unit normal $f^{-1} N$ is upward.
 Then the following statements are equivalent{\rm:}
 \begin{enumerate}
 \item $f$ is a minimal surface.
 \item $d + \alpha^{\l}$ is a family of flat connections of the trivial bundle $\D \times  \ISU$.
\item \Red{The normal Gauss map $g$ for $f$ is a non-holomorphic harmonic 
  map into the hyperbolic $2$-space $\mathbb{H}^2=\mathrm{SU}_{1,1}/\mathrm{U}_1$.}
 \end{enumerate}
 \end{Theorem}
\begin{Remark}
\mbox{}
\begin{enumerate}
\item \Red{
 The equivalence (1) $\Leftrightarrow$ (3) has been proven by 
 \cite{Figueroa}, see also \cite{Ino, Daniel:GaussHeisenbergw}. 
 We have given a new proof for this reulst in \cite{DIKAsian}.
}

 \item The statement that the non-holomorphic harmonic normal Gauss map into 
$\mathbb{H}^2$ implies the item $(2)$ also holds and will be discussed in greater generality below.
 \item We also note that the non-holomorphicity of the normal Gauss map derives
 from the fact that the upper right corner of the $(1, 0)$-part of $\alpha^{\l}$ (that is, 
 $U^{\l}$)  is purely imaginary, and never vanishes, since the surface is nowhere
  vertical.\end{enumerate}
  \end{Remark}
 By $(2)$ of Theorem \ref{thm:mincharact-1}, there exists an
 $F : \mathbb D \to \LISU$ such that $F^{-1} d F = \alpha^{\l}$.
 The argument leading to   $(5.8)$ in the proof
 of Theorem 5.3, \cite{DIKAsian}, shows that actually  the following  matrix, written in terms of 
 the generating spinors, solves this equation for $\lambda = 1$:
 \begin{equation} \label{special frame}
 F|_{\l =1} = 
\frac{1}{\sqrt{|\psi_1|^2-|\psi_2|^2}} 
 \begin{pmatrix}
 \sqrt{i}^{-1} \psi_1 & 
 \sqrt{i}^{-1} \psi_2 \\ 
 \sqrt{i} \;\overline{\psi_2} & 
  \sqrt{i}\; \overline{\psi_1}
 \end{pmatrix}
 \end{equation}
 The frame $F$ as  given in \eqref{special frame} will be called an
 \textit{extended frame} of the minimal surface $f$.
 \begin{Remark}
 The formula above can be rewritten by using a ``hidden symmetry'':
 In view of \eqref{def-exp(w/2)}
 we obtain for minimal surfaces in $\Nil$ the relation
\begin{equation} \label{special relation}
 e^{w/2} = \frac{i}{2}(|\psi_1|^2-|\psi_2|^2),
\end{equation}
where by \eqref{naturalMC} the right upper corner of the $(1,0)$-part of  
 $\alpha$, the Maurer-Cartan form of the moving frame 
 $F(z, \bar z, \lambda)$ for $\lambda = 1$, is $ - e^{w/2}$.
 \end{Remark}
 
 \section{The loop group construction of harmonic maps from $\D$
  into $\mathbb{H}^2=\mathrm{SU}_{1,1}/\mathrm{U}_1$.}\label{app:Loop}
 In Appendix \ref{app:Pre} we have considered a minimal immersion into $\Nil$ and have 
  recalled the construction of an $\mathbb S^1$-family $\alpha^\lambda$ 
 of flat connections.
  Moreover, we have pointed out that for such a minimal immersion  
  the normal Gauss map $g$ is a harmonic map into the unit disk
 $\mathbb{H}^2=\ISU/\Uone$. More precisely, $g$ is obtained from $f^{-1}N$ by a stereographic 
 projection (where this is carried out in $\isu \cong \R^3$ which is considered 
 as a Euclidean space).

 In this section we briefly recall  other realizations of the hyperbolic $2$-space $\mathbb{H}^2$
 and how all harmonic maps into $\mathbb{H}^2=\mathrm{SU}_{1,1}/\mathrm{U}_1$ can be constructed by the loop group method.
 This construction is one of the two main tools for the construction of all minimal surfaces in $\Nil$
 by the loop group method.

\subsection{Realizing Minkowski $3$-space $\Min$ as the usual 
Euclidean $3$-space $\R^3$}\label{subsc:Realizing}
  To relate the setting of the theory of harmonic maps into 
 $\mathbb{H}^2=\ISU/\Uone$  
 to our setting we need to consider a natural isomorphism between
 the usual Euclidean space $\mathfrak{nil}_3 \cong \R^3$ 
 with natural basis $e_1, e_2, e_3,$ and 
 the Minkowski $3$-space  $\Min$ realized by the Lie algebra  
 $\mathfrak{nil}_3$ with natural basis  $\mathcal{E}_1$, $\mathcal{E}_2$, 
 and $\mathcal{E}_3$, spelled out explicitly below.
 
 The Killing form of $\isu$ 
 induces a Lorentz metric  on $\isu$. 
 Thus we regard $\mathfrak{su}_{1,1}$ as 
 the Minkowski 3-space $\Min$.
 The basis of $\Min \cong \isu$
\begin{equation}\label{eq:basis}
 \mathcal{E}_1 = \frac{1}{2} \begin{pmatrix} 0 & i \\ -i &0 \end{pmatrix}, \;\;
 \mathcal{E}_2 = \frac{1}{2} \begin{pmatrix} 0 & -1 \\ -1 & 0 \end{pmatrix}\;\;
 \mbox{and}\;\;\;
 \mathcal{E}_3 = \frac{1}{2} \begin{pmatrix} -i & 0\\ 0 &i \end{pmatrix}.
\end{equation}
  is an orthonormal basis of $\isu=\Min$ with timelike vector 
 $\mathcal{E}_3$ relative to the non-degenerate bilinear form 
  $\langle A , B \rangle = 2 \tr AB$.
  
 An explicit isometry  $ J: \Min \rightarrow \isu $ is
 given by the map 
\[
(x_1,x_2,x_3)^t \mapsto x_3 \mathcal{E}_3 - x_1  \mathcal{E}_2 - x_2  \mathcal{E}_1.
\]
 It is easy to verify that that this map is an isomorphism of Lie algebras, where 
 the Lie algebra structure of $\Min$ is given by the usual cross product.
 
 Note that the group $\ISU$ acts on $\isu=\Min$ by the adjoint representation.
 In particular, the timelike vector $\mathcal{E}_3$ generates
 the rotation group $\mathrm{SO}_2 \cong \ad (\exp(t \mathcal{E}_3 ))$ 
 which acts isometrically on 
 $\Min$ by rotations around the $x_3$-axis. On the other hand, the
 isometries $\exp(t\mathcal{E}_1)$ and $\exp(t\mathcal{E}_2)$ are so called
 \textit{boosts}.
\subsection{Realizing the left translated unit normal 
 and the normal Gauss map in $\isu$}\label{sbsc:Relization}
 From Section \ref{subsc:normalGauss} 
 we know that $f^{-1}N$  and $g$ are realized in the same $3$-dimensional 
 vector space which we will consider to be the natural $\R^3$ as well as to be the 
 three-dimensional Minkowski space  $\Min$. 
 These (identical vector) spaces will be provided with 
 the usual non-degenerate bilinear forms relative to the natural
 basis $e_1, e_2, e_3$ respectively and with $e_3$ the timelike vector
 in the Minkowski case.
 
 By what was said in Section \ref{subsc:normalGauss} we know that $f^{-1}N$ 
 takes values in the two sphere $\mathbb S^2$ relative to the definite metric, 
 actually in the upper hemisphere $\mathbb{S}_+^2$, and $g$ takes values in 
 $\mathbb{H}^2$, 
 realized by the hyperbolic $2$-space $\mathbb H^2$ as the unit disc (in the definite metric)  
 in the complex plane $\C$ perpendicular (in both metrics) to the $e_3$-direction.

 
  The stereographic projection  $\pi: \mathbb S^2_{+} \rightarrow \mathbb{H}^2$ 
 (relative to the definite metric) maps $f^{-1}N$ bi-holomorphically onto 
 $\mathbb{H}^2$. The group $\ISU$ acts on $\C$ by M\"obius transformations, 
 leaving $\mathbb{H}^2$ invariant,
 and this action transforms via the stereographic projection to a group 
 of conformal transformations on $\mathbb S^2$ which leaves  $\mathbb{S}_+^2$
 invariant.
  
 It is well known that the linear fractional action of $\ISU$ on $\mathbb{S}_+^2$
 just mentioned is induced by the standard linear action of 
 $\mathrm{SO}^+(2,1)$ on  $\Min$.

 More precisely, for  a concrete realization one considers the forward light
 cone with vertex at the  ``south pole''  $-e_3$ on the $x_3$-axis and 
its  boundary intersecting 
 the $x_1x_2$-plane in the unit circle.
 Then the stereographic projection from the south pole to the $x_1x_2$-plane and 
  the stereographic projection to the hyperboloid

 \[
 \mathbb{Q}^2 = \left\{v \in \Min\,|\,\langle v,v\rangle = -1,
 \>\> \langle v,e_3\rangle<0\right\}
 \]
inside the open forward light cone give diffeomorphisms, 
 and an isometry from the unit disk $\mathbb{H}^2$
 to the hyperbolid $\mathbb{Q}^2$.
 
 These projections are equivariant relative to the group actions of $\ISU$ discussed above.
 In particular, the action of $\ISU$ is linear and implemented by the adjoint representation.
\subsection{General extended frames of harmonic maps into $\mathbb{H}^2=\ISU/\Uone$}
 In the last sections we have considered three \textit{diffeomorphic}
 space forms of negative curvature,
 $\mathbb{H}^2$, $\mathbb S_+^2$, and $\mathbb{Q}^2$.
 
 For a given minimal surface $f:\D \rightarrow \Nil$ we have correspondingly three normal Gauss maps: 
\begin{itemize}
 \item The normal Gauss map $g:\D \rightarrow  \mathbb H^2$, 
 see definition above. 
 
 \item  The translated unit normal 
 $f^{-1}N = \pi^{-1} \circ g : \D \rightarrow \mathbb S_+^2$, with $\pi$ a 
 stereographic projection, see above.
 
 \item  The corresponding map $N_{\Min}: \D \rightarrow  \mathbb{Q}^2$. 
\end{itemize}
 It is known \cite{DIKAsian} that the normal Gauss map is harmonic. 
 Since the other maps are obtained from $g$ by equivariant 
 conformal diffeomorphisms, they are harmonic as well.
  
 By \cite{DPW}, each of these harmonic maps can be 
 obtained by the loop group method:
 Let us explain briefly how this works the case of $g$.
  Here we have as target space $\mathbb{H}^2=\mathrm{SU}_{1,1}/\mathrm{U}_1$. 
 First one chooses some frame $F:\D \rightarrow \ISU$, 
 which is unique up to right multiplication by an element in $\Uone$.
 Note that this implies $g = F$ mod $\Uone$ in our case.
 Then one introduces (as usual, see for example Section \ref{sbsc:Flat}
 for more details) 
 the loop parameter  into the Maurer-Cartan form $\alpha = F^{-1}dF$, 
 arriving at $\alpha^\lambda$ as above.
 Solving  $F^{-1}dF = \alpha^\lambda$ one obtains what we call for the time being a 
 ``general extended frame $F(z, \bar z, \lambda)$''. 
 From this extended frame one obtains the 
 (meromorphic) 
 \textit{normalized frame} $F_-(z,  \lambda)$ 
 by a Birkhoff decomposition (see for example Section \ref{subsc:GWR}
 for more details).
 
 The Maurer Cartan form $\eta_-(z, \lambda) = F_-(z, \bar z, \lambda)^{-1} dF_-(z, \lambda)$
 is called the \textit{normalized potential}.
 This is a meromorphic one-form defined on $\D$ which has a special form, 
 see for example \eqref{eq:xim}.
 
 Starting, conversely, from any  normalized potential as stated above, 
 one can reverse the steps: first solve an ODE, then find a $\lambda$-dependent 
 frame $F \in \LISU$ by an Iwasawa decomposition (see Step {\rm I\!I} in 
 Section \ref{subsc:GWR}) and finally one obtains a harmonic 
 map into $\mathbb H^2$ by projection to the quotient space $\mathbb H^2$.  
\begin{Definition}\label{dfn:generalext}
 The extended frames defined above have not restrictions 
 on the initial conditions nor on any special additional property. 
 They are therefore called the \textit{general extended frames}.
\end{Definition}

\begin{Theorem} \label{Theorem B.2}
 For a harmonic map $g: \D \rightarrow \mathrm{SU}_{1,1}/\mathrm{U}_1$  any two general 
 extended frames  $F(z, \bar{z}, \lambda)$ and   
 $\tilde{F}(z, \bar z, \lambda)$ for $g$ satisfy
 \[
   \tilde{F}(z, \bar z, \lambda) = A(\lambda) F(z, \bar z, \lambda) k(z, \bar z)
 \]
 with some $A(\lambda) \in \LISU$ satisfying $A(\lambda = 1) = \id$ and 
 $k \in  \mathrm{U}_1$.
 \end{Theorem}
 \begin{proof}
 Let  $F(z, \bar z, \lambda)$ and   $\hat{F}(z, \bar z, \lambda)$  
 be general extended frames of $g$, that is, 
 $F(z, \bar z, \lambda = 1)$ and   $\tilde{F}(z, \bar z, \lambda = 1)$ are frames of $g$. Therefore 
  $F(z, \bar z, \lambda = 1) = \tilde{F}(z, \bar z, \lambda = 1) k(z, \bar z)$ for some $k \in U_1$.
  Now the claim follows.
 \end{proof}
 \begin{Remark}
\mbox{}
\begin{enumerate}
 \item An extended frame $F$ of a minimal surface $f$ as in 
 \eqref{special frame} is of course 
 a general extended frame of the harmonic map induced by 
 the normal Gauss map of $f$. Moreover, two extended frames  
 $\tilde F$ and $F$ of $f$  are related by 
\begin{equation}\label{eq:twoextmin}
 \tilde F = A F, 
\end{equation}
 with some $A(\lambda) \in \LISU$ satisfying $A(\lambda = 1) = \id$.
 Here $k \in \Uone$ is identity since $F|_{\lambda =1}$ 
 is given by the generating spinors $\psi_1, \psi_2$ of the minimal surface $f$.
 \item 
 If one wants the two loop group procedures outlined above to be  
 inverse to each other, then one can achieve this by choosing 
 some fixed base point $z_0 \in \D$ and assume that all matrix functions 
 occurring above attain the value $\id$ at $z_0$.
\end{enumerate}
 \end{Remark}

 \section{The loop group construction of minimal surfaces in $\Nil$}\label{app:DPW}
 \subsection{Extended frames of minimal surfaces in $\Nil$ and extended frames 
 of harmonic maps into $\mathbb H^2 \cong \mathrm{SU}_{1,1}/\mathrm{U}_1$}
 For the purposes of this paper we need to use special frames 
 in order to construct minimal surfaces in $\Nil$.

 For this we would like to point out, that in the proof of Theorem 6.1 in \cite{DIKAsian} 
 it was shown that the map $N_{\Min}: \D \rightarrow \mathbb{Q}^2$, 
 equivalent to $g$, has a frame 
 of the form \eqref{special frame}. 
 Moreover, the $(1,2)$-entry of the $(1,0)$-part of the  Maurer-Cartan form of this frame
 never vanishes on $\D$, since we only considered minimal immersions into $\Nil$ there.
 We generalize this result by  proving the following ``folk theorem'':
 \begin{Theorem}\label{thm:folktheorem}
 Assume the matrix valued function 
 $\hat F:\mathbb D \to \LISU$ satisfies 
\[
 \hat{\alpha}^\lambda =  \hat F^{-1}d \hat F,
\] 
 where
\begin{equation}\label{eq:U-V 2lambda}
 \alpha^{\l} = \hat{U}^{\l} dz + \hat{V}^{\l} d\bar z
\end{equation}
 with 
 \begin{equation*}
 \hat{U}^{\l} =
 \begin{pmatrix}
 a  &
 - \l^{-1} b \\ 
 \l^{-1} c &  -a
 \end{pmatrix},  \quad \quad 
 \hat{V}^{\l} =
 \begin{pmatrix}
  q & - \l b\\
 \l r & -q
 \end{pmatrix},
 \end{equation*} 
 and where $b$ never vanishes  on $\D$.
 Then $a = q = iu$ with $u$ a real valued function, 
 as well as $p = \bar c$ and $r = \bar b.$
  Moreover, after a diagonal gauge in $\LISU$ we can assume that $b$ is purely imaginary
 and never vanishes on $\D$.
 In this case, after writing $b$ in the form $b =  - e^{w/2}$ the matrices $ \hat{U}^{\l}$ and 
 $ \hat{V}^{\l}$ attain the explicit form stated in equation \eqref{eq:U-V1lambda}.
 \end{Theorem}
\begin{proof}[Sketch of the proof] 
 The first claim follows from $\hat F \in \LISU.$
 Writing $b$ in the form $b = i v e^{is}$ with $v$ and $s$ real valued functions 
 we see that the diagonal gauge in $\LISU$  with $(1, 1)$-entry $e^{-is/2}$  
 verifies the second claim. Assuming the first two claims are satisfied, 
 then the last claim follows by an evaluation of the integrability condition   
 of $\hat{\alpha}^\lambda$.
 \end{proof}
 
 \begin{Corollary} \label{special frame for g}
  If  $\hat F:\mathbb D \to \LISU$ is a general extended frame of a harmonic map 
 $N_{\Min}:\D \rightarrow \mathbb{Q}^2$, such that the $(1, 2)$-entry 
 of the $(1,0)$-part of the Maurer Cartan form 
 of $\hat{F}$ never vanishes on $\D$,  then there exists a matrix function
 $k:\D \rightarrow \mathrm{U}_1$ 
 such that $\hat{F} = F k$ with $F$ a general extended frame of $g$ which 
 satisfies  \eqref{eq:U-V1lambda} and is of the form \eqref{special frame} 
 for all $\lambda \in \mathbb S^1$. Moreover, equation \eqref{special relation} 
 holds for all $\lambda \in \mathbb S^1$. 
  \end{Corollary}
\begin{proof}
 By the theorem above we can assume without loss of generality that 
 the Maurer-Cartan form of $\hat{F}$
 has the form stated in  \eqref{eq:U-V1lambda}. 
 Using \eqref{special relation} we can define for all $\lambda \in \mathbb S^1$
 the function $h$ which is supposed to become 
 $2 (|\psi_1|^2 - |\psi_2|^2)$. Putting 
 \[
  \hat{F}_{11} = \psi_1 \sqrt{2} (i h)^{-1/2} \quad \mbox{and}\quad
  \hat{F}_{12} = \psi_2 \sqrt{2} (i h)^{-1/2},
 \] 
 we have rewritten $\hat{F}$ for all $\lambda \in 
 \mathbb S^1$ in the special form \eqref{special frame}.

 \end{proof}  
 By the results above we have found very special frames for harmonic maps into 
 $\mathbb{Q}^2$. What we still want  to show is that the functions $\psi_j$ occurring 
 in these frames define a minimal surface in $\Nil$. 
 We will achieve this in the next subsection.
\subsection{Sym-formula}\label{sc:Sym}
 We regard $\mathfrak{su}_{1,1}$ as 
 the Minkowski 3-space $\Min$ as in Section \ref{subsc:Realizing}.
 We identify the Lie algebra 
 $\mathfrak{nil}_3$ of $\Nil$ with the 
 Lie algebra $\isu$ as a \textit{real vector space}.  
 Then the corresponding  linear isomorphism $\Xi:\mathfrak{su}_{1,1}\to 
 \mathfrak{nil}_3$ is given by
 \begin{equation*}
\mathfrak{su}_{1,1} \ni 
x_1 \mathcal{E}_1 + x_2 \mathcal{E}_2 + x_3 \mathcal{E}_3
\longmapsto
x_1 e_1 +  x_2 e_2 +  x_3 e_3 \in \mathfrak{nil}_{3}.
\end{equation*}
 It should be remarked that the linear isomorphism $\Xi$ is 
 \textit{not} a Lie algebra 
 isomorphism. For geometric meaning of this 
 linear isomorphism, see Appendix \ref{mysterious}.
 
 Next we consider the exponential map 
 $\exp:\mathfrak{nil}_3\to \mathrm{Nil}_3$.   
 We define a smooth bijection  
 $\Xi_{\rm nil}:\isu \to \Nil$ by $\Xi_{\rm nil}:=\exp \circ \Xi$.
Under this identification $\mathrm{Nil}_3=\mathfrak{su}_{1,1}$, and 
$\mathrm{SO}_2=\{\exp(t\mathcal{E}_3)\}_{t\in\mathbb{R}}$ 
 acts isometrically on $\mathrm{Nil}_3$ by rotations around the $x_3$-axis.
 
 In what follows we will take derivatives 
 for the variable $\l$.
 Note that for $\l=e^{i\theta} \in \mathbb S^1$, we have
 $\partial_{\theta}=i \l \partial_{\l}$.
The following result is essentially Theorem 6.1 of \cite{DIKAsian}, but has weaker assumptions.
It turns out  that the proof stays correct for the slightly more general assumptions stated just below.
\begin{Theorem} \label{thm:Sym}
 Let  $F: \mathbb D \to \LISU$ be a general extended frame of a harmonic map 
 $g: \D \rightarrow \mathbb{H}^2$, such that the $(1, 2)$-entry of the 
 $(1,0)$-part of the Maurer Cartan form 
 of $F$ never vanishes on $\D$,  and such that $F$ satisfies the conclusions of
 {\rm Corollary \ref{special frame for g}}.

 Define the maps $f_{\Min}$ and $N_{\Min}$ 
 respectively by
 \begin{equation}\label{eq:SymMin}
 f_{\Min}=-i \l (\partial_{\l} F) F^{-1} 
 - \frac{i}{2} \ad (F) \sigma_3.\;\;
 \mbox{and} \;\;
 N_{\Min}= \frac{i}{2} \ad (F) \sigma_3,
 \end{equation}
 where $\sigma_3 = \left( \begin{smallmatrix} 1 &0 \\ 0 & -1 \end{smallmatrix}\right)$.
 Moreover, define a map $f^{\l}:\mathbb{D}\to \mathrm{Nil}_3$ by%
\begin{equation}\label{eq:symNil}
 f^{\l}:=\Xi_{\mathrm{nil}}\circ \hat{f}\quad\mbox{with}
\quad
 \hat f = 
    (f_{\Min})^o -\frac{i}{2} \l (\partial_{\l}  f_{\Min})^d, 
\end{equation}
 where the superscripts ``$o$'' and ``$d$'' denote the off-diagonal and 
 diagonal part, 
 respectively. 
 Then, for each $\l \in \mathbb{S}^1$, the following 
 statements hold:
\begin{enumerate}
 \item The map $f_{\Min}$ is a spacelike 
 constant mean curvature surface with mean curvature $H=1/2$ in $\Min$
 and $N_{\Min}$ is the timelike unit normal vector of $f_{\Min}$.
 \item The map $f^{\l}$ is a minimal surface in $\Nil$ and 
 $N_{\Min}$ is the isometric image of the normal Gauss map of $f^{\l}$ in 
 the hyperboloid $\mathbb{Q}^2$ under the natural isometry from the 
 unit disk $\mathbb{H}^2$
 onto $\mathbb{Q}^2$ $($see  {\rm Section \ref{sbsc:Relization} } for details$)$.
 In particular, any general extended frame of $g$  is an extended frame of
 some minimal surface $f$. Furthermore, $f^{\l}|_{\l =1}$ and $f$ are the same 
 up to a translation.
\end{enumerate}
 Conversely, for each minimal surface $f : \D \rightarrow \Nil$ 
 there exists an extended frame such that the Sym-formula applied to 
 this frame produces for $\lambda = 1$ the given immersion $f$.
\end{Theorem}

\begin{proof}
 We only need to prove the ``converse'' statement. 
 To do this, we choose an extended frame $F$
of the form (\ref{special frame}) for the given surface $f$. By the results above we know
that $F$ does induce (for $\lambda = 1$) a minimal surface $\tilde{f}$ in $\Nil$ via the Sym formula. 
Moreover, $f$ and $\tilde{f}$ only differ by a translation in $\Nil$. It thus suffices to prove that there exists a matrix $A(\lambda) \in \LISU$ satisfying $A(\lambda = 1) = \id$ such that the frame 
$\check{F} = A(\lambda) F$ induces, via the Sym formula, 
 exactly the original surface $f$ for $\lambda = 1$.
 In fact if we choose 
 \[
  A (\lambda) = \exp B(\lambda) \quad \mbox{with}\quad 
 B(\lambda) = \frac{1}{4}B_1 (\lambda - \lambda^{-1}) + \frac{1}{8}B_2 (\lambda - \lambda^{-1})^2
 \]
 such that $B(\lambda) \in \Lisu$, then $A(\lambda) \in \LISU$ and $A(\lambda = 1)=\id$.
 Moreover, the respective minimal surfaces given by the frames $F$ and 
 $\check F$  differ by a translation $T=(p, q, r)$ with
\[
 p = -\Re (b_{112}), \quad  q = -\Im (b_{112}) \quad \mbox{and} \quad r= -i b_{2 11},
\] 
 respectively, where $b_{112}$ is the $(12)$-entry of $B_1$ and 
$b_{211}$ is the $(11)$-entry of $B_2$.
\end{proof}

In view of  Corollary \ref{special frame for g},  we obtain
\begin{Corollary}
 Let  $F: \mathbb D \to \LISU$ be a general extended frame of a harmonic map 
 $g: \D \rightarrow \mathbb{H}^2$, such that the $(1, 2)$-entry of 
 the $(1,0)$-part of the Maurer Cartan form 
 of $F$ never vanishes on $\D$. 
 Define the maps $f_{\Min}$ and $N_{\Min}$ as in the last theorem.
 Then the conclusions of the theorem above also hold.
\end{Corollary}

 Moreover, from Corollary \ref{special frame} for $g$
 and Theorem \ref{thm:Sym} we have
\begin{Corollary}\label{cro:extendedframe}
 Let $F$ be an extended frame of a minimal surface $f$ 
 as defined in \eqref{special frame}, and let  $\alpha^{\l}$
 denote the Maurer-Cartan form of $F$. Moreover let $\hat F$ be a any 
 solution of $\hat F^{-1} d \hat F = \alpha^{\l}$ which takes values in 
 $\LISU$, that is, $F$ and $\hat F$ are related in the form  
 $F = A \hat F$ with some $z$-independent matrix $A \in \LISU$.
 Then plugging $\hat F$ into the Sym formula \eqref{eq:symNil}, we obtain 
 a minimal surface $\hat{f}$ in $\Nil$ and $\hat{F}$ is an extended frame for $\hat{f}.$
 \end{Corollary}
 \begin{Remark}
 In general,  this surface $\hat f$ is not isometric to the original 
 minimal surface $f$, see Example \ref{ex:translation}.
\end{Remark}

\subsection{Generalized Weierstrass type representation}\label{subsc:GWR}
 We now briefly summarize the results of the 
 generalized Weierstrass type representation in \cite[Section 7]{DIKAsian} as follows:
 Let $F$ be an extended frame of some minimal surface $f$  
 as in \eqref{special frame} defined on a simply connected domain $\mathbb D$.
 The Birkhoff decomposition, see \cite[Theorem 7.1]{DIKAsian} or 
 \cite{PreS:LoopGroup}, 
 of $F$ is given as
\[
 F = F_{-} F_{+}, \quad F_{-} \in \LSLMI, \quad F_{+} \in \LSLP.
\] 
 Then by
 \cite[Theorem 7.2]{DIKAsian} 
 $F_{-}$ is meromorphic with respect to $z$ and moreover, the Maurer-Cartan 
 form $F_{-}^{-1} d F_{-}$ satisfies 
 \begin{equation}\label{eq:xim}
 \xi_{-} = F_{-}^{-1} d F_{-} = \l^{-1} 
 \begin{pmatrix}
 0 & -p \\ B p^{-1} & 0
\end{pmatrix}
 d z,
 \end{equation}
 where $p$ is a meromorphic function on $\mathbb D$ and $B dz^2$ 
 is the Abresch-Rosenberg differential which is a holomorphic 
 quadratic differential.
 The meromorphic $1$-form $\xi_{-}$ as in \eqref{eq:xim} will be 
 called the \textit{normalized potential}. 
 
 Conversely,
 
 \textbf{Step I.} Let $\xi_{-}$ be a meromorphic 1-form of the form stated  in \eqref{eq:xim} 
 which has a global  meromorphic solution to $ d C = C \xi_{-}$  and 
 solve the  linear ODE:
 \begin{equation*}
   d C = C \xi_{-} \quad \mbox{with $C(z_0, \l) \in \LSL$}.
 \end{equation*}
 
 \textbf{Step I\!I.} Apply the unique  Iwasawa decomposition 
 as stated in \cite[Remark 8.1]{DIKAsian} for $C$ near $z_0$, 
 that is,
 \begin{equation*}
  C = F V_+  \in \LISU \cdot \LSLP \quad \mbox{or} \quad  C = F \omega_0 V_+ 
 \in \LISU \cdot \omega_0\cdot \LSLP , 
 \end{equation*}
 where 
 $\omega_0 = \left( \begin{smallmatrix} 
 0 & \l \\
  - \l^{-1} & 0	   
 \end{smallmatrix}\right)$.
 Then from Theorem 8.2 in \cite{DIKAsian}, 
 it follows that there exists some diagonal matrix  $ D \in \LISU$ such that $FD$ or 
 $\omega_0FD$ is an extended frame of some minimal surface in $\Nil$ in the sense of 
 Corollary \ref{cro:extendedframe}.

\textbf{Step I\!I\!I.} In the final step, 
 minimal surfaces in $\Nil$ can be obtained by the Sym formula in Theorem \ref{thm:Sym}.
\begin{Remark}
 We note that the normal Gauss map $N_{\Min}$ of the resulting minimal surface can 
 be obtained by the extended frame $FD$ or $\omega_0 F D$ by 
\[
 \frac{i}{2}\ad (F) \sigma_3 \quad \mbox{or} \quad  \frac{i}{2}\ad (\omega_0 F)  \sigma_3,
\] 
 which is in fact the unit normal to the spacelike constant mean curvature 
 $H =1/2$ surface $f_{\Min}$ in $\Min$ defined in \eqref{eq:SymMin}.
\end{Remark}

We will explain how to produce all minimal surfaces by our method.
 The main point is Birkhoff splittability of an extended frame of 
 a minimal surface which satisfies \eqref{special frame}. 
 Starting from some minimal surface we obtain a special frame $\tilde F$
 as in \eqref{special frame}. Note that $\tilde F$ is independent of $\l$.
 Choose some fixed base point $z_0 \in \D$ and consider $B_0 = \tilde F(z_0)$.
 Now consider $B(\l) = \left(
 \begin{smallmatrix}
 b_{11} &  \lambda^{-1} b_{12} \\ \l \overline{b_{12}} & \overline{b_{11}}
 \end{smallmatrix}\right) \in \LISU$, 
 where the $b_{ij} \;(1\leq i, j\leq 2)$ are the entries of $B_0$.
 Note that $B(\lambda)$ is Birkhoff splittable $B = B_- B_+$ with
 $B_- = \left( \begin{smallmatrix}1 & \lambda^{-1} \\ 0 &1 \end{smallmatrix} \right)$, 
 and $B_+$ lower triangular, as can be verified by a simple computation. 
 Note that $B_{+11}$ never vanishes.

 Next solve the Maurer-Cartan form equation for $F$ with initial condition 
 $B(\lambda)$ for each 
 $\lambda \in \mathbb S^1$. This will  produce an extended frame which coincides 
 with the original $\tilde F$ for $\lambda =1$ and which will be Birkhoff splittable 
 near the base point $z_0$.

\section{Real form involution and global meromorphicity} 
\label{sc:Meroextension}
 Let $\eta (z,  \lambda)$ be a potential 
 for a minimal surface in $\Nil$.
 Consider the solution to $dC = C \eta$, satisfying 
 $C(0,\lambda) =  \id$.
 Let $\varphi$ denote the involution 
 which characterizes the 
 real form 
 $\LISU$ in $\LSL$.
 Then we have $\varphi (g) 
 = \sigma_3 \overline{{}^t g(1/ \bar \lambda)}^{-1} \sigma_3$ 
 for $g \in \LSL$.
 By abuse of notation, put 
\begin{equation}\label{eq:varphi}
 \varphi\left(\eta (w,\lambda)\right) = - \sigma_3 
 \overline{{}^t\eta( \bar w, 1/ \bar \lambda)} \sigma_3.
\end{equation}
 We now introduce $\iota: A(z, w, \l) \mapsto A(w, z, \l)$ 
 for $A : \D \times \overline{\D} \to \LSL$
 and define (group level)
\[
\hat \varphi\left(A(z,w,\lambda)\right) =  
 \iota(\varphi A(z,w,\l)) = \sigma_3 \overline{{}^t A(\bar w, \bar z, 
 1/ \bar \l)}^{-1} \sigma_3.
\]
In this sense we abbreviate
\[
R(w, \lambda) = \varphi(C(w,\lambda)) = 
 \sigma_3 \overline{{}^t C(\bar w, 1/ \bar \lambda)}^{-1} \sigma_3. 
\]
Now, analogous to the usual loop group approach to the construction of integrable surfaces we consider next $Q(z,w, \lambda) =  R(w, \lambda)^{-1} C(z, \lambda)$ and consider its (meromorphic) Birkhoff decomposition
\begin{equation}\label{eq:Birkhoffdouble}
 Q(z,w, \lambda) = R(w, \lambda)^{-1} C(z, \lambda)  
 = V_-^{-1}(z,w,\lambda)S(z,w) V_+(z,w, \lambda),
\end{equation}
 where $V_+$ and $V_-$ have leading term $\id$, 
 and $S$ is a $\l$-independent diagonal matrix.
 As pointed out in \cite{DPW}, 
 the entries of $V_+, V_-$ and $S$ are quotients of the entries of $C^{-1}R$. 
 As a consequence they are meromorphic functions on $\D \times \overline{\D}$. From \eqref{eq:Birkhoffdouble}, it is easy to see that 
\begin{equation}\label{eq:U}
 U = C V_+^{-1} = R V_-^{-1} S.
\end{equation}

\subsection{Iwasawa decomposition and the decomposition of $B$}
\label{subsc:AppIwasawa}
Eventually, we want to determine $S$ in more detail.
To start with we observe 
\[
\hat \varphi(R^{-1} C) = (R^{-1} C)^{-1}. 
\]
From this we infer the equations
\begin{equation} \label{relations}
\hat \varphi(V_+) = V_- \quad  \mbox{and} \quad \hat \varphi(S) = 
 S^{-1}.
\end{equation}
For $U = C V_+^{-1} = R V_-^{-1} S$,
we thus obtain $\hat \varphi(U) = U S^{-1}$.
We want to prove: $S= \pm (\hat \varphi l)^{-1} l$ 
 for some $\l$-independent diagonal matrix $l$.
To begin with we consider $S(0,0)$. We observe that \eqref{relations} 
implies that $S(0,0)$ is real (and non-zero anyway). 

Case 1: $S(0,0) >0$:
 Writing $S(z,w) = \di (e^{b(z,w)}, e^{-b(z, w)})$, 
 we see that to prove our claim we need to find some function 
 $a(z,w)$ such that
\[
 b(z,w) = a(z,w) + \overline{a(\bar w,\bar z)}
\]
 with $a(0,0)$ real. But $\hat \varphi S = S^{-1}$ implies 
\[
\overline{ b( \bar{w},\bar{z})}= b(z,w).
 \]
Using this and a  power series expansion of $b$ and setting
\begin{equation*}
a(z,w) = \sum_{0 < n < m} b_{nm} z^n w ^m + \frac{1}{2} \sum_{n = 0} b_{nn} z^n w ^n.
\end{equation*}
we obtain $b(z, w) = a(z, w) + \overline{a(\bar w,\bar z)}$. 
Hence (so far at least locally) we obtain, as desired, 
$S =  (\hat \varphi l)^{-1} l$.
Moreover,  $\hat{U} = U l^{-1}$ satisfies 
 $\hat \varphi (\hat{U}) = \hat{U}$.
In addition
\begin{equation*}
C = \hat{U} \hat{V}_+,
\end{equation*}
 with $\hat V_+ = l V_+$.

Case 2: $S(0,0) <0$:
 Write $S = - S_0,$ then $\hat \varphi S_0 = S_0^{-1}$  
 and $S_0(0,0) >0$ holds. The argument given just above 
 produces some $k$  satisfying $S_0 = (\hat \varphi k)^{-1} k$.
 Then $\check{U} = U k^{-1}$ satisfies $\hat \varphi(\check{U}) 
 = - \check{U}$, 
 what we are not interested in.
 Therefore we reconsider
\[
 R^{-1} C = \hat{V}_-^{-1} ( - \id) \hat{V}_+ 
 =  \hat{V}_-^{-1} ( (\hat \varphi \omega_0)^{-1} \omega_0) \hat{V}_+, 
\]
with $\omega_0 = \left(
\begin{smallmatrix}
0 & \lambda \\
- \lambda^{-1} & 0
\end{smallmatrix}\right)$.
 We obtain 
 \begin{equation}\label{eq:hatUonIw}
\hat{U} = C \hat{V}_+^{-1} \omega_0^{-1} 
 = R \hat{V}_-^{-1} (\varphi \omega_0)^{-1}. 
\end{equation}
 Consequently we arrive at
 \begin{equation*}
 C = \hat{U} \omega_0 \hat{V}_+ 
 \quad \mbox{and} \quad  \hat \varphi(\hat{U}) = \hat{U}.
 \end{equation*}
 When $w = \bar z$, then $\hat \varphi$ is the anti-linear 
 involution defining $\LISU$ and thus $\hat U$ takes values $\LISU$.
 Moreover the leading term of $\hat V_+ = l V_+$ has real entries.
 Let $C = F V_+$ be the (unique) 
 Iwasawa decomposition on $z \in \mathcal I_e$ as in \eqref{IwasawaIe}. Then 
 we have $\hat U = F$ and thus 
 \[
  F l = U 
 \]
 holds, and $F l$ has a unique meromorphic extension. Moreover, on $z
 \in \mathcal I_{w}$, we have 
 \[
  \hat U = F l k^{-1} w_0^{-1}
 \]
 for $\hat U$ defined in \eqref{eq:hatUonIw}.
%
\section{Geometric meaning of the linear isomorphism 
$\mathfrak{su}_{1,1}$ and $\mathfrak{nil}_3$}\label{mysterious}
\subsection{Unimodular Lie algebras}
Let us consider a $3$-dimensional 
real \textit{unimodular} Lie 
algebra $\mathfrak{g}$ 
with basis $\{e_1,e_2,e_3\}$. This Lie algebra is 
defined by the commutation relations:
\[
[e_1,e_2]=c_3\>e_3,\ \ 
[e_2,e_3]=c_1\>e_1,\ \ 
[e_3,e_1]=c_2\>e_2.
\]
We introduce an inner product 
on $\mathfrak{g}$ so that $\{e_1,e_2,e_3\}$ is 
orthonormal with respect to it.

Here we introduce auxiliary parameters $\mu_1$, $\mu_2$, $\mu_3$ by 
\[
\mu_{i}=\frac{1}{2}(c_1+c_2+c_3)-c_i, \ \ 
i=1,2,3.
\]
Now we restrict our attention to the range:
\[
c_1=c_2=:c\leq 0, \ \ \ c_3=:2\tau\geq 0.
\]
We denote the metric Lie algebra by $\mathfrak{g}(c,\tau)$.
The corresponding simply connected 
Lie group with left invariant metric is denoted by $G(c,\tau)$.

Then we have the following table of sectional curvatures:
\[
K(e_1\wedge e_2)=-3\tau^2+2c\tau, 
\ \  
K(e_2\wedge e_3)=K(e_1 \wedge e_3)=\tau^2.
\]

The quantity $\kappa:=K(e_1\wedge e_2)+3\tau^2=2c\tau$ is called the 
\textit{base curvature} of  $G(c,\tau)$.

\begin{Example}[$\mathrm{Nil}_3$]
Let us choose $c=0$ then $\mathfrak{g}(0,\tau)$ is 
isomorphic to $\mathfrak{nil}_3(\tau)$.
We have $\mu_1=\mu_2=-\mu_3=\tau$, so we get
$K(e_1\wedge e_2)=-3\tau^2$, $K(e_2\wedge e_3)=K(e_1 \wedge e_3)=\tau^2$. 
Hence $\kappa=0$.
\end{Example}
\begin{Example}[$\mathrm{SU}_{1,1}$]
Next let us consider the case $c<0$. 
In this case, the Lie algebra is 
isomorphic to $\mathfrak{su}_{1,1}$ and the isometry 
group 
of the corresponding simply connected Lie group $G(c,\tau)$ is 4-dimensional and 
$K(e_1\wedge e_2)=-3\tau^2+2c\tau$, 
$K(e_2\wedge e_3)=K(e_1 \wedge e_3)=\tau^2$. Hence $\kappa=2c\tau<0$.
\end{Example}
One can see that $\mathfrak{nil}_3(\tau)=\lim_{c\to 0}\mathfrak{g}(c,\tau)$.
We can show that there is a real analytic collapsing 
$G(c,\tau)\to \mathrm{Nil}_3(\tau)$. Note that for $c<0$, 
$G(c,\tau)$ is the universal covering of $\mathrm{SU}_{1,1}$.

\subsection{Anti de Sitter space}
Now we consider the metric induced from the Killing form 
of $\mathfrak{su}_{1,1}$.

First we take the basis $\{e_1,e_2,e_3\}$ of $\mathfrak{g}(c,\tau)$ as
before. Next we choose $c$ so that $c=-2\tau>0$. Moreover we 
define a scalar product $\langle\cdot,\cdot\rangle_{L}$ by the rule
$\{e_1,e_2,e_3\}$ is orthogonal and 
\[
\langle e_1,e_1\rangle_L=
\langle e_2,e_2\rangle_L=
-\langle e_3,e_3\rangle_L=1.
\] 
Denote by $\omega$ the left invariant 
$1$-form on $G(-2\tau,\tau)$ dual to $e_3$.
Then the two scalar products are related by 
$\langle\cdot,\cdot\rangle_L=
\langle\cdot,\cdot\rangle-2\omega^2$. 

This scalar product is given explicitly by
\[
\langle X,Y\rangle_L=\frac{1}{2\tau^2}\mathrm{tr}\>(XY).
\]
This shows that the induced Lorentzian metric is bi-invariant 
and proportional to the Killing metric.
Since the metric is bi-invariant, we 
have
\[
\langle R(X,Y)Y,X\rangle=\frac{1}{4}
\langle [X,Y], [X,Y]\rangle_L.
\]
This implies that $G(-2\tau,\tau)$ is of constant curvature 
$-\tau^2$.

From these observations we can interpret the isomorphism 
$\mathfrak{nil}_3(1/2)\to \mathfrak{su}_{1,1}$ in the following way.

\begin{enumerate}
\item For $\tau>0$ and $c<0$, we consider 
the unimodular Lie algebra $\mathfrak{g}(c,\tau)$ with basis 
$\{e_1,e_2,e_3\}$ and equip  a scalar product $\langle\cdot,\cdot\rangle
=\langle\cdot,\cdot\rangle_{c,\tau}$.
\item Take $c=-2\tau$ and change the inner product to the scalar product
$\langle\cdot,\cdot\rangle_L$. Then we have the Minkowski 3-space 
$\Min=\mathbb{R}e_1\oplus\mathbb{R}e_2\oplus\mathbb{R}e_3$;
\[
\Min:=(\mathfrak{g}(-2\tau,\tau),\langle\cdot,\cdot\rangle_L).
\]
The Lie algebra is $\mathfrak{su}_{1,1}$.
\item On the other hand, fixing the inner product 
$\langle\cdot,\cdot\rangle$ on $\mathfrak{g}(c,\tau)$. 

Then the resulting $\lim_{c\to 0}\mathfrak{g}(c,\tau)$ 
is Euclidean $3$-space 
$\mathbb{R}^3=\mathbb{R}e_1\oplus\mathbb{R}e_2\oplus\mathbb{R}e_3$
with nilpotent Lie algebra structure. Thus 
$\lim_{c\to 0}\mathfrak{g}(c,\tau)$ is $\mathfrak{nil}_3(\tau)$.
\end{enumerate}
Thus there is a linear isomorphism (identity map) 
\[
\mathfrak{su}_{1,1}=\mathfrak{g}(-2\tau,\tau)
\longleftrightarrow \mathfrak{g}(0,\tau)=\mathfrak{nil}_3(\tau)
\]
given by $e_i\longleftrightarrow e_i$.

Thus the isomorphism first observed by Cartier \cite{Cartier} is just 
the \emph{identity map}. 
Note that the simply connected Lie group $G(-2\tau,\tau)$ equipped with 
left invariant Riemannian metric is the model space 
$\widetilde{\mathrm{PSL}}_2$ of Thurston geometry.

\subsection{Explicit models}
Take the following split-quaternion basis:
\begin{equation*}
\Vec{i}=\begin{pmatrix} i & 0\\ 0 &-i \end{pmatrix},
\ \ 
\Vec{j}^{\prime}=
\begin{pmatrix} 0 & -i \\ i &0 \end{pmatrix},
\ \ 
\Vec{k}^{\prime}=
\begin{pmatrix} 0 & 1 \\ 1 & 0 \end{pmatrix}.
\end{equation*}
of $\mathfrak{su}_{1,1}$.
We define the basis $\{\mathcal{E}^\tau_1,\mathcal{E}_2^\tau,
\mathcal{E}^\tau_3\}$ by
\[
\mathcal{E}^\tau_1=-\tau\Vec{j}^{\prime},
\ \ 
\mathcal{E}_2^\tau=-\tau\Vec{k}^{\prime},
\  \
\mathcal{E}_3^\tau=-\tau\Vec{i}.
\]
This basis satisfies 
\[
[\mathcal{E}^\tau_1, \mathcal{E}^\tau_2]=2\tau\mathcal{E}^\tau_3,
\ \ 
[\mathcal{E}^\tau_2, \mathcal{E}^\tau_3]=-2\tau\mathcal{E}^\tau_1,
\ \
[\mathcal{E}^\tau_3, \mathcal{E}^\tau_1]=-2\tau\mathcal{E}^\tau_2.
\]
We use the scalar product
\[
\langle X,Y\rangle_\tau:=\frac{1}{2\tau^2}\mathrm{tr}\>(XY),
\]
then  $\{\mathcal{E}^\tau_1,\mathcal{E}_2^\tau,
\mathcal{E}^\tau_3\}$ is orthonormal. The sectional curvature is $-\tau^2$.
If we put $e_i=\mathcal{E}^\tau_i$, then 
$c_1=c_2=-2\tau<0$ and $c_3=2\tau>0$.

Thus we have the following fact.
\begin{Theorem}
For a positive number $\tau$, we take a basis
$\{e_1,e_2,e_3\}$ of $\mathfrak{su}_{1,1}$ defined by 
$e_i=-\tau\>\mathcal{E}^\tau_i$.
Introduce two scalar products on $\mathfrak{su}_{1,1}$ by
\begin{itemize}
\item The inner product defined by the rule, $\{e_1,e_2,e_3\}$ is 
orthonormal with respect to it.
\item The Lorentzian scalar product 
\[
\langle X,Y\rangle_L=\frac{1}{2\tau^2}\mathrm{tr}\>(XY).
\]
\end{itemize}
Then we have
\begin{itemize}
\item With respect to the Riemannian metric, $\mathrm{SU}_{1,1}$
has sectional curvatures
\[
K(e_1\wedge e_2)=-7\tau^2,\  \
K(e_2\wedge e_3)=K(e_3\wedge e_1)=\tau^2.
\]
The base curvature is $-4\tau^2$.
\item With respect to the Lorentzian metric, $\mathrm{SU}_{1,1}$ is of
constant curvature $-\tau^2$.
\end{itemize}
In both cases the quotient space 
$\mathbb{H}^2=\mathrm{SU}_{1,1}/\mathrm{U}_1$
is of constant curvature $-4\tau^2$.
\end{Theorem}
If we choose $\tau=1/2$, we recover the situations 
in this article.
 
If we define the \textit{sign} $\epsilon$ by 
\[
\epsilon=
\begin{cases}
+1 &  \mbox{Riemannian metric}\\
-1 & \mbox{Lorentzian metric}.
\end{cases}
\]
The we have the unified formula for 
the sectional curvatures:
\[
K(e_1\wedge e_2)=-3\epsilon\tau^2-4\tau^2,\  \
K(e_2\wedge e_3)=K(e_3\wedge e_1)=\epsilon\tau^2.
\]

\subsection{Sister surfaces \cite{Daniel:iso}}

Let us take a minimal surface $f:\mathbb{D}\to\mathrm{Nil}_3(\tau)$.
Then its sister surface $\tilde{f}:\mathbb{D}\to G(c,\tilde{\tau})$ is 
defined by the relation
\[
-4\tau^2=\tilde{\kappa}-4\tilde{\tau}^2,\ \  
\tau^2=\tilde{\tau}^2+\tilde{H}^2,
\ \ 
\tilde{H}=-\tilde{\kappa}/4,  \ \ 
\tilde{\kappa}=2c\tilde{\tau}.
\]
If we choose $c=-2\tilde{\tau}$, we get 
$-4\tilde{H}^2=\tilde{\kappa}=-4\tilde{\tau}^2$. Thus we may choose $\tilde{H}=\tilde{\tau}>0$.
Thus $\tilde{\tau}=\tau/\sqrt{2}$.
Hence $\tilde{f}$ is a constant mean curvature surface
in $G(-\sqrt{2}\tau,\tau/\sqrt{2})$ with mean curvature 
$\tau/\sqrt{2}$. 
}

 {\bf Acknowledgements}\ \ This work was started during a visit of the third named author
  at the Technical University of Munich and a visit of the first named author 
 at Hirosaki University.  They would like to express their sincere gratitude for 
 the hospitality extended to them by the corresponding departments. They also thank 
 to the referee for a thorough reading and thoughtful remarks.  

\bibliographystyle{plain}
\def\cprime{$'$}

\end{document}